\documentclass[a4paper]{amsart}

\usepackage{amsmath,amsthm,amssymb,enumerate}
\usepackage{epsfig}
\usepackage{amssymb}
\usepackage{amsmath}
\usepackage{amssymb}
\usepackage{amsmath,amsthm}
\usepackage[latin1]{inputenc}
\usepackage[T1]{fontenc}
\usepackage{path}
\usepackage{ae,aecompl}
\usepackage{amsfonts}
\usepackage{amsxtra}
\usepackage{euscript,mathrsfs}
\usepackage{color}
\usepackage[left=3.3cm,right=3.3cm,top=4cm,bottom=3cm]{geometry}
\usepackage[citecolor=blue,colorlinks=true]{hyperref}
\allowdisplaybreaks

\usepackage{enumitem}
\setenumerate{label={\rm (\alph{*})}}

\usepackage{amsfonts}
\usepackage{amsxtra}

\usepackage[frak,hyperref,theorem_section]{paper_diening}
\numberwithin{equation}{section}

\newcommand{\vr}{\varrho}
\newcommand{\vu}{\vc{u}}
\newcommand{\vc}[1]{{\bf #1}}
\newcommand{\tor}{\mathbb{T}^3}
\newcommand{\expe}[1]{ \mathbb{E} \left[ #1 \right] }

\newcommand{\Div}{\divergence}

\newcommand{\R}{\mathbb R}
\newcommand{\N}{\mathbb N}

\newcommand{\E}{\mathbb E}
\newcommand{\p}{\mathbb P}

\newcommand{\dd}{\mathrm{d}}
\newcommand{\dx}{\,\mathrm{d}x}
\newcommand{\dt}{\,\mathrm{d}t}
\newcommand{\dxt}{\,\mathrm{d}x\,\mathrm{d}t}
\newcommand{\ds}{\,\mathrm{d}\sigma}
\newcommand{\dxs}{\,\mathrm{d}x\,\mathrm{d}\sigma}

\newcommand{\dif}{\mathrm{d}}

\newcommand{\mf}{\mathscr{F}}
\newcommand{\mr}{\mathbb{R}}
\newcommand{\prst}{\mathbb{P}}
\newcommand{\stred}{\mathbb{E}}
\newcommand{\ind}{\mathbf{1}}
\newcommand{\mn}{\mathbb{N}}

\newcommand{\mt}{\mathbb{T}^3}

\DeclareMathOperator{\diver}{div}

\newcommand{\distr}{\overset{d}{\sim}}
\newcommand{\tec}{{\overset{\cdot}{}}}

\newcommand{\bu}{\mathbf u}
\newcommand{\bq}{\mathbf q}

\allowdisplaybreaks

\begin{document}


\title{Stochastic Navier-Stokes equations for compressible fluids}

\author{Dominic Breit}
\address[D. Breit]{LMU Munich, Mathematical Institute, Theresienstra\ss e 39, 80333 Munich, Germany\newline
Department of Mathematics, Heriot-Watt University, Riccarton Edinburgh EH14 4AS, UK}
\email{d.breit@hw.ac.uk}

\author{Martina Hofmanov\'a}
\address[M. Hofmanov\'a]{Max Planck Institute for Mathematics in the Sciences, Inselstra\ss e 22, 04103 Leipzig, Germany\newline Technical University Berlin, Institute of Mathematics, Stra\ss e des 17. Juni 136, 10623 Berlin, Germany}
\email{hofmanov@math.tu-berlin.de}

\begin{abstract}
We study the Navier-Stokes equations governing the motion of an isentropic compressible fluid in three dimensions driven by a multiplicative stochastic forcing. In particular, we consider a stochastic perturbation of the system as a function of momentum and density. We establish existence of a so-called finite energy weak martingale solution under the condition that the adiabatic constant satisfies $\gamma>3/2$. The proof is based on a four layer approximation scheme together with a refined stochastic compactness method and a careful identification of the limit procedure.
\end{abstract}


\date{\today}

\maketitle


\section{Introduction}



We consider the Navier-Stokes system for isentropic compressible viscous fluids driven by a multiplicative stochastic forcing and prove existence of a solution that is weak in both PDE and probabilistic sense. To be more precise,  let $\mt=[0,1]^3$ denote the three-dimensional torus, let $T>0$ and set $Q=(0,T)\times\mt$. We study the following system which governs the time evolution of density $\varrho$ and velocity $\bu$ of a compressible viscous fluid:
\begin{subequations}\label{eq:}
 \begin{align}
  \dif \varrho+\diver(\varrho\bu)\dif t&=0,\label{eq1}\\
  \dif(\varrho\bu)+\big[\diver(\varrho\bu\otimes\bu)-\nu\Delta\bu-(\lambda+\nu)\nabla\diver\bfu+\nabla p(\varrho)\big]\dif t&=\varPhi(\varrho,\varrho\bu) \,\dif W.\label{eq2}
 \end{align}
\end{subequations}
These equations describe the balance of mass and momentum of the flow.
Here $p(\varrho)$ is the pressure which is supposed to follow the $\gamma$-law, i.e. $p(\varrho)=a\varrho^\gamma$ where $a>0$ and $a$ is the squared reciprocal of the Mach-number (ratio of flow velocity and speed of sound). For the adiabatic exponent $\gamma$ (also called isentropic expansion factor) we suppose $\gamma>\frac{3}{2}$. Finally, the viscosity coefficients $\nu,\,\lambda$ satisfy
$$\nu>0,\quad\lambda+\frac{2}{3}\nu\geq0.$$
The driving process $W$ is a cylindrical Wiener process defined on some probability space $(\Omega,\mf,\prst)$ and the coefficient $\varPhi$ is generally nonlinear and satisfies suitable growth conditions. The precise description of the problem setting will be given in the next section.

The literature devoted to deterministic case is very extensive (see for instance Feireisl \cite{fei3}, Feireisl, Novotn\'y and Petzeltov\'a \cite{feireisl1}, Lions \cite{Li1}, Novotn\'y and Stra\v{s}kraba \cite{novot} and the references therein). The existence of weak solutions in the non-stationary setting is well-known provided $\gamma>\frac{3}{2}$ (in three dimensions, in two dimensions $\gamma>1$ suffices instead). This might not be optimal but already covers important examples like mono-atomic gases where $\gamma=\frac{5}{3}$. In the stationary situation the results have been recently extended to $\gamma>1$, see \cite{FMS,PlVa}.\\
The theory for the stochastic counterpart still remains underdeveloped. The only available results (see Feireisl, Maslowski and Novotn\'y \cite{feireisl2} for $d=3$ and \cite{To} in the case $d=2$) concern the Navier-Stokes system for compressible barotropic fluids under a stochastic perturbation of the form $\varrho\,\dif W$. 
This particular case of a multiplicative noise permits reduction of the problem. After applying some transformation it can be solved pathwise and therefore existence of a weak solution was established using deterministic arguments. This method has the drawback that the constructed solutions do not necessarily satisfy an energy inequality and are not progressively measurable (hence the stochastic integral is not defined).
We are not aware of any results concerning the Navier-Stokes equations for compressible fluids driven by a general multiplicative noise. Nevertheless, study of such models is of essential interest as they were proposed as models for turbulence, see Mikulevicius and Rozovskii \cite{mikul}. In case of a more general noise, the simplification mentioned before is no longer possible and methods from infinite-dimensional stochastic analysis are required.

There is a bulk of literature available concerning stochastic versions of the incompressible Navier-Stokes equations. Let us mention the pioneering paper by Bensoussan and Temam \cite{BeTe} and for an overview of the known results, recent developments, as well as further references, we refer to \cite{debusergo}, \cite{Fl2} and \cite{Ma}. The literature concerning other fluid types is rather rare. Just very recently first results on stochastic models for Non-Newtonian fluids appeared (see \cite{Br}, \cite{TeYo} and \cite{Yo}). Incompressible non-homogenous fluids with stochastic forcing were studied in \cite{Fu} and more recently in \cite{Sa}; one-dimensional stochastic isentropic Euler equations in \cite{berthvovelle}.

We aim at a systematic study of compressible fluids under random perturbations.
Our main result is the existence of a weak martingale solution to \eqref{eq:} in the sense of Definition \ref{def:sol}, see Theorem \ref{thm:main}.
Our solution satisfies an energy inequality which shows the time evolution of the energy compared to the initial energy. The setting includes in particular the case of
$$\Phi(\varrho,\varrho\bfu)\,\dd W= \Phi_1(\varrho) \,\dd W^1+\Phi_2(\varrho\bfu)\,\dd W^2$$
with two independent cylindrical Wiener processes $W^1$ and $W^2$ and suitable growth assumptions on $\varPhi_1$ and $\varPhi_2$, which is the main example we have in mind. Here the first term describes some external force; the case $\Phi_1(\varrho)=\varrho$ studied in \cite{feireisl2} is included but we could also allow nonlinear dependence in $\varrho$ (the case $\Phi(\varrho,\varrho\bfu)\,\dd W= \varrho\,\dd W$ corresponds to the forcing $\varrho\,\bff$ from deterministic models). The second term is a friction term; the model case is $\Phi_2(\varrho\bfu)$ being proportional to the momentum $\varrho\bfu$ but the dependence can be nonlinear as well. The solution is understood weakly in space-time (in the sense of distributions) and also weakly in the probabilistic sense (the underlying probability space is part of the solution). Such a concept of solution is very common in the theory of stochastic partial differential equations (SPDEs), in particular in fluid dynamics when the corresponding uniqueness is often not known. We refer the reader to Subsection \ref{subsec:solution} for a detailed discussion of this issue. 

The proof of Theorem \ref{thm:main} relies on a four layer approximation scheme that is motivated by the technique developed by Feireisl, Novotn\'y and Petzeltov\'a \cite{feireisl1} in order to deal with the corresponding deterministic counterpart. In each step we are confronted with the limit procedure in several nonlinear terms and in the stochastic integral. There is one significant difference in comparison to the deterministic situation leading to the concept of martingale solution: In general it is not possible to get any compactness in $\omega$ as no topological structure on the sample space $\Omega$ is assumed. To overcome this difficulty, it is classical to rather concentrate on compactness of the set of laws of the approximations and apply the Skorokhod representation theorem. It yields existence of a new probability space with a sequence of random variables that have the same laws as the original ones and that in addition converges almost surely. However, a major drawback is that the Skorokhod representation Theorem is restricted to metric spaces. The structure of the compressible Navier-Stokes equations naturally leads to weakly converging sequences. On account of this we work with the Jakubowski-Skorokhod Theorem which is valid on a large class of topological spaces (including separable Banach spaces with weak topology). Further discussion of the key ideas of the proof is postponed to Subsection \ref{subsec:outline}.

The exposition is organized as follows. In Section \ref{sec:framework} we continue with the introductory part: we introduce the basic set-up, the concept of solution and state the main result, Theorem \ref{thm:main}. Once the notation is fixed we present also a short outline of the proof, Subsection \ref{subsec:outline}. The remainder of the paper is devoted to the detailed proof of Theorem \ref{thm:main} that proceeds in several steps.


\section{Mathematical framework and the main result}
\label{sec:framework}

To begin with, let us set up the precise conditions on the random perturbation of the system \eqref{eq:}. Let $(\Omega,\mf,(\mf_t)_{t\geq0},\prst)$ be a stochastic basis with a complete, right-continuous filtration. The process $W$ is a cylindrical Wiener process, that is, $W(t)=\sum_{k\geq1}\beta_k(t) e_k$ with $(\beta_k)_{k\geq1}$ being mutually independent real-valued standard Wiener processes relative to $(\mf_t)_{t\geq0}$. Here $(e_k)_{k\geq1}$ denotes a complete orthonormal system in a sepa\-rable Hilbert space $\mathfrak{U}$ (e.g. $\mathfrak{U}=L^2(\mt)$ would be a natural choice).
To give the precise definition of the diffusion coefficient $\varPhi$, consider $\rho\in L^\gamma(\mt)$, $\rho\geq0$, and $\bfv\in L^2(\mt)$ such that $\sqrt\rho\bfv\in L^2(\mt)$. We recall that we assume $\gamma>\frac{3}{2}$.\\ Denote $\bfq=\rho\bfv$ and let $\,\varPhi(\rho,\bq):\mathfrak{U}\rightarrow L^1(\mt)$ be defined as follows
$$\varPhi(\rho,\bq)e_k=g_k(\cdot,\rho(\cdot),\bq(\cdot)).$$
The coefficients $g_{k}:\mt\times\mr\times\mr^3\rightarrow\mr^3$ are $C^1$-functions that satisfy uniformly in $x\in\mt$
\begin{align}\label{growth1}
\sum_{k\geq 1}|g_{k}(x,\rho,\bfq)|^2&\leq C\big(\rho^2+|\rho|^{\gamma+1}+|\bfq|^2\big),\\
\sum_{k\geq 1}|\nabla_{\rho,\bfq} g_{k}(x,\rho,\bfq)|^2&\leq C\big(1+|\rho|^{\gamma-1}\big)\label{growth2}.
\end{align}
Remark that in this setting $L^1(\mt)$ is the natural space for values of the operator $\varPhi(\rho,\rho\bfv)$. Indeed, due to lack of a priori estimates for \eqref{eq:} it is not possible to consider $\varPhi(\rho,\rho\bfv) $ as a mapping with values in a space with higher integrability. This fact brings difficulties concerning the definition of the stochastic integral in \eqref{eq:}. In fact, the space $L^1(\mt)$ does neither belong to the class 2-smooth Banach spaces nor to the class of UMD Banach spaces where the theory of stochastic It\^o-integration is well-established (see e.g. \cite{b2}, \cite{veraar}, \cite{ondrejat3}). However, since we expect the momentum equation \eqref{eq2} to be satisfied only in the sense of distributions anyway, we make use of the embedding $L^1(\mt)\hookrightarrow W^{-b,2}(\mt)$ (which is true provided $b>\frac{3}{2}$). Hence we understand the stochastic integral as a process in the Hilbert space $W^{-b,2}(\mt)$. To be more precise, it is easy to check that under the above assumptions on $\rho$ and $\bfv$, the mapping $\varPhi(\rho,\rho\bfv)$ belongs to $L_2(\mathfrak{U};W^{-b,2}(\mt))$, the space of Hilbert-Schmidt operators from $\mathfrak{U}$ to $W^{-b,2}(\mt)$. Indeed, due to \eqref{growth1} there holds
\begin{align}
\big\|\varPhi(\rho,\rho\bfv)\big\|^2_{L_2(\mathfrak{U};W^{-b,2}_x)}&=\sum_{k\geq1}\|g_k(\rho,\rho\bfv)\|_{W^{-b,2}_x}^2\leq C\sum_{k\geq1}\|g_k(\rho,\rho\bfv)\|_{L^1_x}^2\nonumber\\
&\leq C(\rho)_{\mt}\int_{\mt}\bigg(\sum_{k\geq1}\rho^{-1}|g_{k}(x,\rho,\rho\bfv)|^2\bigg)\,\dif x\label{stochest}\\
&\leq C(\rho)_{\mt}\int_{\mt}\big(\rho+\rho^\gamma+\rho|\bfv|^2\big)\,\dif x<\infty,\nonumber
\end{align}
where $(\rho)_{\mt}$ denotes the mean value of $\rho$ over $\mt$.
Consequently, if\footnote{Here $\mathcal{P}$ denotes the predictable $\sigma$-algebra associated to $(\mf_t)$.}
\begin{align*}
\rho&\in L^\gamma(\Omega\times(0,T),\mathcal{P},\dif\prst\otimes\dif t;L^\gamma(\mt)),\\
\sqrt\rho\bfv&\in L^2(\Omega\times(0,T),\mathcal{P},\dif\prst\otimes\dif t;L^2(\mt)),
\end{align*}
and the mean value $(\rho(t))_{\mt}$ is essentially bounded
then the stochastic integral $\int_0^\tec\varPhi(\rho,\rho\bfv)\,\dif W$ is a well-defined $(\mf_t)$-martingale taking values in $W^{-b,2}(\mt)$. Note that the continuity equation \eqref{eq1} implies that the mean value $(\varrho(t))_{\mt}$ of the density $\varrho$ is constant in time (but in general depends on $\omega$).
Finally, we define the auxiliary space $\mathfrak{U}_0\supset\mathfrak{U}$ via
$$\mathfrak{U}_0=\bigg\{v=\sum_{k\geq1}\alpha_k e_k;\;\sum_{k\geq1}\frac{\alpha_k^2}{k^2}<\infty\bigg\},$$
endowed with the norm
$$\|v\|^2_{\mathfrak{U}_0}=\sum_{k\geq1}\frac{\alpha_k^2}{k^2},\qquad v=\sum_{k\geq1}\alpha_k e_k.$$
Note that the embedding $\mathfrak{U}\hookrightarrow\mathfrak{U}_0$ is Hilbert-Schmidt. Moreover, trajectories of $W$ are $\prst$-a.s. in $C([0,T];\mathfrak{U}_0)$ (see \cite{daprato}).

\subsection{The concept of solution and the main result}
\label{subsec:solution}

We aim at establishing existence of a solution to \eqref{eq:} that is weak in both probabilistic and PDEs sense. Let us devote this subsection to the introduction of these two notions. From the point of view of the theory of PDEs, we follow the approach of \cite{feireisl1} and consider the so-called finite energy weak solutions. In particular, the system \eqref{eq:} is satisfied in the sense of distributions, the corresponding energy inequality holds true and, moreover, the continuum equation \eqref{eq1} is satisfied in the renormalized sense.

From the probabilistic point of view, two concepts of solution are typically considered in the theory of stochastic evolution equations, namely, pathwise (or strong) solutions and martingale (or
weak) solutions. In the former notion the underlying probability space as well as the driving process is fixed in advance while
in the latter case these stochastic elements become part of the solution of the problem. Clearly, existence of a pathwise solution is stronger and implies existence of a martingale solution. In the present work we are only able to establish existence of a martingale solution to \eqref{eq:}. Due to classical Yamada-Watanabe-type argument (see e.g. \cite{krylov}, \cite{pr07}), existence of a pathwise solution would then follow if pathwise uniqueness held true. However, uniqueness for the Navier--Stokes equations for compressible fluids is an open problem even in the deterministic setting. In hand with this issue goes the way how the initial condition is posed: there is given a probability measure on $L^\gamma(\mt)\times L^\frac{2\gamma}{\gamma+1}(\mt)$, hereafter denoted by $\Lambda$. It fulfills some further assumptions specified in Theorem \ref{thm:main} and plays the role of an initial law for the system \eqref{eq:}. That is, we require that the law of $(\varrho(0),\varrho\bfu(0))$ coincides with $\Lambda$.

Let us summarize the above in the following definition.

\begin{definition}[Solution]\label{def:sol}
Let $\Lambda$ be a Borel probability measure on $L^\gamma(\mt)\times L^\frac{2\gamma}{\gamma+1}(\mt)$. Then
$$\big((\Omega,\mf,(\mf_t),\prst),\varrho,\bfu,W)$$
is called a finite energy weak martingale solution to \eqref{eq:} with the initial data $\Lambda$ provided
\begin{enumerate}
\item $(\Omega,\mf,(\mf_t),\prst)$ is a stochastic basis with a complete right-continuous filtration,
\item $W$ is an $(\mf_t)$-cylindrical Wiener process,
\item the density $\vr$ satisfies $\vr \geq 0$, $t \mapsto \left< \vr(t, \cdot), \psi \right> \in C[0,T]$ for any
$\psi \in C^\infty(\tor)$
$\mathbb{P}$-a.s., the function $t \mapsto \left< \vr(t, \cdot), \psi \right>$
is progressively measurable,
and
\[
\E\bigg[\sup_{t \in [0,T]} \| \vr(t,\cdot) \|^p_{L^\gamma(\tor)} \bigg] < \infty \ \mbox{for all}\ 1 \leq p < \infty;
\]
\item the velocity field $\vu$ is adapted, $\vu \in L^2(\Omega \times (0,T); W^{1,2}(\tor))$,
\[
\E\bigg[\bigg( \int_0^T \| \vu \|^2_{W^{1,2}(\tor)} \ \dt \bigg)^p \bigg] < \infty\ \mbox{for all}\ 1 \leq p < \infty;
\]
\item the momentum $\vr \vu$ satisfies $t \mapsto \left< \vr \vu, \bfphi \right> \in C[0,T]$ for any $\bfphi \in C^\infty(\tor)$
$\mathbb{P}$-a.s., the function $t \mapsto \left< \vr \vu, \phi \right>$ is progressively measurable,
\[
\expe{ \sup_{t \in [0,T]} \left\| \vr \vu \right\|^p_{L^{\frac{2 \gamma}{\gamma + 1}}} } < \infty\ \mbox{for all}\  1 \leq p < \infty;
\]
\item $\Lambda=\prst\circ \big(\varrho(0),\varrho\bfu(0) \big)^{-1}$.
\item $\varPhi(\varrho,\varrho\bfu)\in L^2(\Omega\times[0,T],\mathcal{P},\dif\prst\otimes\dif t;L_2(\mathfrak{U};W^{-l,2}(\mt)))$ for some $l>\frac{3}{2}$,

\item for all $\psi\in C^\infty(\mt)$ and
 $\bfvarphi\in C^\infty(\mt)$ and all $t\in[0,T]$ there holds $\prst$-a.s.
\begin{align*}
\big\langle\varrho(t),\psi\big\rangle&=\big\langle\varrho(0),\psi\big\rangle+\int_0^t\big\langle\varrho\bfu,\nabla\psi\big\rangle\,\dif s,\\
\big\langle\varrho\bfu(t),\bfvarphi\big\rangle&=\big\langle\varrho \bfu(0),\bfvarphi\big\rangle+\int_0^t\big\langle\varrho\bfu\otimes\bfu,\nabla\bfvarphi\big\rangle\,\dif s-\nu\int_0^t\big\langle\nabla\bfu,\nabla\bfvarphi\big\rangle\,\dif s\\
&\quad-(\lambda+\nu)\int_0^t\big\langle\diver\bfu,\diver\bfvarphi\big\rangle\,\dif s+a\int_0^t\big\langle\varrho^\gamma,\diver\bfvarphi\big\rangle\,\dif s\\
&\quad+\int_0^t\big\langle\varPhi(\varrho,\varrho\bfu)\,\dif W,\bfvarphi\big\rangle,
\end{align*}
\item for all $p\in[1,\infty)$ the following energy inequality holds true
\begin{equation}\label{energy}
\begin{split}
&\stred\bigg[\sup_{0\leq t\leq T}\int_ {\mt} \Big(\frac{1}{2} \varrho(t)\big| \bfu(t)\big|^2+\frac{a}{\gamma-1}\varrho^\gamma (t)\Big)\dx\bigg]^p\\
&\quad+\stred\bigg[\int_0^{T}\int_ {\mt}\Big(\nu |\nabla \bfu |^2+(\lambda+\nu)|\diver\bfu|^2\Big)\dx\,\dif s\bigg]^p\\
& \leq \,C(p)\,\stred\bigg[\int_{\mt} \Big(\frac{1}{2} \frac{ |\varrho \bfu(0)|^2 }{\varrho(0)} +\frac{a}{(\gamma-1)} \varrho(0)^\gamma\Big)\dif x+1\bigg]^p.
\end{split}
\end{equation}
\item Let $b\in C^1(\R)$ such that $b'(z)=0$ for all $z\geq M_b$. Then
for all $\psi\in C^\infty(\mt)$ and all $t\in[0,T]$ there holds $\prst$-a.s.
\begin{align*}
\big\langle b(\varrho(t)),\psi\big\rangle&=\big\langle b(\varrho(0)),\psi\big\rangle+\int_0^t\big\langle b(\varrho)\bfu,\nabla\psi\big\rangle\,\dif s-
\int_0^t\big\langle \big(b'(\varrho)\varrho-b(\varrho)\bfu)\big)\diver\bfu,\psi\big\rangle\,\dif s.
\end{align*}
\end{enumerate}
\end{definition}

\begin{remark}
In (j) above, the continuity equation is assumed to hold in the renormalized sense. This
concept was introduced in \cite{DL}. It is an essential tool
to pass to the limit in the nonlinear pressure and therefore common in compressible fluid mechanics.
\end{remark}
\begin{remark}
The condition (g) was included in order to point out that the stochastic integral in \eqref{eq:} is a well-defined stochastic process with values in $W^{-b,2}(\mt)$, in particular, the integrand is progressively measurable. Nevertheless, the conditions on $\varrho$ and $\bfu$ together with the energy inequality \eqref{energy} already imply that $\varPhi(\varrho,\varrho\bfu)$ takes values in $L_2(\mathfrak{U};W^{-b,2}(\mt))$.
\end{remark}

To conclude this subsection we state our main result.

\begin{theorem}\label{thm:main}
Let $\gamma>3/2.$ Assume that for the initial law $\Lambda$ there exists $M\in(0,\infty)$ such that
\begin{equation*}\label{initial1}
\Lambda\Big\{(\rho,\bfq)\in L^\gamma(\mt)\times L^\frac{2\gamma}{\gamma+1}(\mt);\, \rho\geq0,\;(\rho)_{\mt}\leq M,\;\bfq(x)=0\;\text{whenever}\;\rho(x)=0\Big\}=1,
\end{equation*}
and that for all $p\in[1,\infty)$ the following moment estimate holds true
\begin{equation}\label{initial}
\int_{L^\gamma_x\times L^\frac{2\gamma}{\gamma+1}_x}\bigg\|\frac{1}{2}\frac{|\bfq|^2}{\rho}+\frac{a}{\gamma-1}\rho^\gamma\bigg\|_{L^1_x}^p\,\dif\Lambda(\rho,\bfq)<\infty.
\end{equation}
Then there exists a finite energy weak martingale solution to \eqref{eq:} with the initial data $\Lambda$.
\end{theorem}

\begin{remark}
Note that the condition \eqref{initial} is directly connected to the energy inequality \eqref{energy}. More precisely,
\begin{align*}
\int_{L^\gamma_x\times L^\frac{2\gamma}{\gamma+1}_x}\bigg\|\frac{1}{2}\frac{|\bfq|^2}{\rho}+\frac{a}{\gamma-1}\rho^\gamma\bigg\|_{L^1_x}^p\,\dif\Lambda(\rho,\bfq)=\stred \bigg[\int_{\mt}\frac{1}{2}\frac{|\varrho\bfu(0)|^2}{\varrho(0)}+\frac{a}{\gamma-1}\varrho(0)^\gamma\,\dif x\bigg]^p
\end{align*}
which is the quantity that appears on the right hand side of \eqref{energy} (cf. Proposition \ref{prop:apriori1}). It follows from our proof that $C$ does not depend on $a$, $\gamma$, $\lambda$ or $\nu$.
\end{remark}

\begin{remark}
In order to simplify the computations we only study the case of periodic boundary conditions (note that the density does not require any boundary assumptions in the weak formulation). However, with a bit of additional work our theory can also be applied to the case of no-slip boundary conditions. Furthermore, the reader might observe that the assumption upon the initial law $\Lambda$ that implies $(\varrho(0))_{\mt}\leq M$ a.s. can be weakened to
$$\stred\big|(\varrho(0))_{\mt}\big|^p<\infty\qquad \forall p\in[2,\infty).$$
Furthermore, the total mass remains constant in time, i.e.
$$(\varrho(t))_{\mt}=(\varrho(0))_{\mt}\qquad\forall t\in[0,T].$$
\end{remark}

\begin{remark}
In dimension two the result of Theorem \ref{thm:main} even holds under the weaker assumption $\gamma>1$.
\end{remark}

\subsection{Outline of the proof}
\label{subsec:outline}

Our proof relies on a four layer approximation scheme whose core follows the technique developed by Feireisl, Novotn\'y and Petzeltov\'a \cite{feireisl1} in order to deal with the corresponding deterministic counterpart. To be more precise, we regularize the continuum equation by a second order term and modify correspondingly the momentum equation so that the energy inequality is preserved. In addition, we consider an artificial pressure term that allows to weaken the hypothesis upon the adiabatic constant $\gamma$. Thus we are led to study the following approximate system
\begin{subequations}\label{eq:approx}
 \begin{align}
  \dif \varrho+\diver(\varrho\bu)\dif t&=\varepsilon\Delta\varrho\,\dif t,\label{eq:approx1}\\
  \dif(\varrho\bu)+\big[\diver(\varrho\bu\otimes\bu)-\nu\Delta\bu-(\lambda+\nu)\nabla\diver\bfu\nonumber\\
+a\nabla \varrho^\gamma+\delta\nabla\varrho^\beta+\varepsilon\nabla\bu\nabla\varrho\big]\dif t&=\varPhi(\varrho,\varrho\bu) \,\dif W,\label{eq:approx2}
 \end{align}
\end{subequations}
where $\beta>\max\{\frac{9}{2},\gamma\}$. The term $\varepsilon\nabla\bfu\nabla\varrho$ is added to the momentum equation to maintain the energy balance. In order to ensure its convergence to 0
in the vanishing viscosity limit the artificial pressure $\delta\varrho^\beta$ is needed (it implies higher integrability of $\varrho$). It yields an estimate for $\sqrt{\varepsilon}\nabla\varrho$ which is uniformly in $\varepsilon$ by \eqref{eq:approx1}.

The aim is to pass to the limit first in $\varepsilon\rightarrow0$ and subsequently in $\delta\rightarrow0$, however, in order to solve \eqref{eq:approx} for $\varepsilon>0$ and $\delta>0$ fixed we need two additional approximation layers. In particular, we employ a stopping time technique to establish the existence of a unique solution to a finite-dimensional approximation of \eqref{eq:approx}. We gain so-called Faedo-Galerkin approximation, on each random time interval $[0,\tau_R)$ where the stopping time $\tau_R$ is defined as
$$\tau_R=\inf\big\{t\in[0,T];\|\bfu\|_{L^2}\geq R\big\}\wedge\inf\bigg\{t\in[0,T];\bigg\|\int_0^t\varPhi^{N}\big(\varrho,\varrho\bfu\big)\,\dif W\bigg\|_{L^2}\geq R\bigg\}$$
(with the convention $\inf\emptyset=T$), where $\varPhi^N$ is a suitable finite-dimensional approximation of $\varPhi$. It is then showed that the blow up cannot occur in a finite time. So letting $R\rightarrow\infty$ gives a unique solution to the Faedo-Galerkin approximation on the whole time interval $[0,T]$. 
The passage to the limit as $N\rightarrow\infty$ yields existence of a solution to \eqref{eq:approx}.

Except for the first passage to the limit, i.e. as $R\rightarrow\infty$, we always employ the stochastic compactness method. Let us discuss briefly its main features. The compactness method is widely used for solving various PDEs: one approximates the model problem, finds suitable uniform estimates proving that the set of approximate solutions is relatively compact in some path space and this leads to convergence of a subsequence whose limit is shown to fulfill the target equation.
The situation is more involved in the stochastic setting due to presence of the additional variable $\omega$. Indeed, generally it is not possible to get any compactness in $\omega$ as no topological structure on $\Omega$ is assumed. To overcome this issue, one concentrates rather on compactness of the set of laws of the approximations and then the Skorokhod representation theorem comes into play. It gives existence of a new probability space with a sequence of random variables that have the same laws as the original ones (so they can be shown to satisfy the same approximate problems though with different Wiener processes) and that in addition converge almost surely.

Powerful as it sounds there is one drawback of the classical Skorokhod representation theorem (see e.g. \cite[Theorem 11.7.2]{dudley}): it is restricted to random variables taking values in separable metric spaces. Nevertheless, Jakubowski \cite{jakubow} gave a suitable generalization of this result that holds true in the class of so-called quasi-Polish spaces. That is, topological spaces that are not metrizable but retain several important properties of Polish spaces (see \cite[Section 3]{onsebr} for further discussion). Namely, separable Banach spaces equipped with weak topology or spaces of weakly continuous functions with values in a separable Banach space belong to this class which perfectly covers the needs of our paper.

Another important ingredient of the proof is then the identification of the limit procedure. To be more precise, the difficulties arise in the passage of the limit in the stochastic integral as one now deals with a sequence of stochastic integrals driven by a sequence of Wiener processes. One possibility is to pass to the limit directly and such technical convergence results appeared in several works (see \cite{bensoussan} or \cite{krylov}), a detailed proof can be found in \cite{debussche1}. Another way is to show that the limit process is a martingale, identify its quadratic variation and apply an integral representation theorem for martingales, if available (see \cite{daprato}). Our proof relies on neither of those and follows a rather new general and elementary method that was introduced in \cite{on1} and already generalized to different settings (see \cite{on2} for the application to quasi-Polish spaces). The keystone is to identify not only the quadratic variation of the corresponding martingale but also its cross variation with the limit Wiener process obtained through compactness. This permits to conclude directly without use of any further difficult results.

\section{The Faedo-Galerkin approximation}
\label{sec:galerkin}

In this section, we present the first part of our proof of Theorem \ref{thm:main}. In particular, we prove existence of a unique solution to a Faedo-Galerkin approximation of the following viscous problem \eqref{eq:approx}
where $\varepsilon>0,\,\delta>0$ and $\beta> \max\set{\frac{9}{2},\gamma}$. To be more precise, let us consider a suitable orthogonal system formed by a family of smooth functions $(\bfpsi_n)$. We choose
 $(\bfpsi_n)$ such that it is an orthonormal system with respect to the $L^2(\mathbb T^3)$ inner product which is orthogonal with respect to the the $W^{l,2}(\mathbb T^3)$ inner product  where $l>\frac{5}{2}$ is fixed. Now, let us define the finite dimensional spaces
$$X_N=\mathrm{span}\{\bfpsi_1,\dots,\bfpsi_N\},\quad N\in\mn,$$
and let $P_N:L^2(\mt)\rightarrow X_N$ be the projection onto $X_N$ which also acts as a linear
projection $P_N:W^{l,2}(\mt)\rightarrow X_N$.

The aim of this section is to find a unique solution to the finite-dimensional approximation of \eqref{eq:approx}. Namely, we consider
\begin{subequations}\label{eq:galerkin}
 \begin{align}
  \dif \varrho+\diver(\varrho\bu)\dif t&=\varepsilon\Delta\varrho\,\dif t,\label{eq:galerkin11}\\
  \dif(\varrho\bu)+\big[\diver(\varrho\bu\otimes\bu)-\nu\Delta\bu-(\lambda+\nu)\nabla\diver\bfu\nonumber\\
+\,a\nabla \varrho^\gamma+\delta\nabla\varrho^\beta+\varepsilon\nabla\bu\nabla\varrho\big]\dif t&=\varPhi^{N}(\varrho,\varrho\bu) \,\dif W,\label{eq:galerkin2}\\
  \varrho(0)=\varrho_0,\qquad (\varrho\bu)(0)&=\bq_0.\label{eq:galerkin3}
 \end{align}
\end{subequations}
The equation \eqref{eq:galerkin2} is to be understood in the dual space $X^\ast_N$ .
The coefficient in the stochastic term is defined as follows:
\begin{align}\label{eq:PhiepsN}
\varPhi^{N}(\rho,\bfq)e_k&=g_k^{N}(\rho,\bfq),\qquad g_k^{N}(\rho,\bfq)=\mathcal M^{\frac{1}{2}}[\rho]P_N\bigg(\frac{g_k(\rho,\bfq)}{\sqrt{\rho}}\bigg),
\end{align}
where for $\varrho\in L^1(\mt)$ with $\varrho\geq0$ a.e.
\begin{equation}\label{M}
\mathcal{M}[\varrho]:X_N\longrightarrow X^*_N,\quad\big\langle\mathcal{M}[\varrho]\bfv,\bfw\big\rangle=\int_{\mt}\rho\,\bfv\cdot\bfw\,\dif x,\quad \bfv,\,\bfw\in X_N.
\end{equation}
Note that we can identify $X_N^\ast$ with $X_N$ via the natural embedding such that $\mathcal M[\rho]$ is a positive symmetric semidefinite operator on a Hilbert space having a unique square root in the same class. It follows from the definition of $\mathcal{M}[\rho]$ that
$$\mathcal{M}[\rho]\bfv=P_N(\rho\bfv).$$
Note further that we can extend $\mathcal M[\rho]$ to $L^2(\mt)$ in case 
of bounded $\rho$ or to $W^{l,2}(\mt)$ if $\rho\in L^2(\mt)$ by setting
$$\mathcal{M}[\rho]\bfv=P_N(\rho \,P_N \bfv).$$
More details on the properties of $\mathcal{M}$ can be found in \cite[Section 2.2]{feireisl1} and in Appendix \ref{sec:appendix}.

The initial condition $(\varrho_0,\bfq_0)$ is a random variable with the law $\Gamma$, where $\Gamma$ is a Borel probability measure on $C^{2+\kappa}(\mt)\times C^2(\mt)$, with $\kappa>0$, satisfying 
$$\Gamma\Big\{(\rho,\bfq)\in C^{2+\kappa}(\mt)\times C^2(\mt);\;0<\underline{\rho}\leq \rho\leq \overline{\rho}\Big\}=1,$$
and for all $p\in[2,\infty)$
\begin{equation}\label{initial2}
\int_{C^{2+\kappa}_x\times C^2_x}\bigg\|\frac{1}{2}\frac{|\bfq|^2}{\rho}+\frac{a}{\gamma-1}\rho^\gamma+\frac{\delta}{\beta-1}\rho^\beta\bigg\|_{L^1_x}^p\,\dif\Gamma(\rho,\bfq)\leq C.
\end{equation}

As in \cite[Section 2]{feireisl1}, the system \eqref{eq:galerkin} can be equivalently rewritten as a fixed point problem
\begin{equation}\label{eq:galerkin1}
\begin{split}
\bu(t)&=\mathcal{M}^{-1}\big[\mathcal{S}(\bu)(t)\big]\bigg(\bq_0^\ast+\int_0^t\mathcal{N}\big[\mathcal{S}(\bu),\bu\big]\dif s\\
&\qquad+\int_0^t\varPhi^{N}\big(\mathcal{S}(\bu),\mathcal{S}(\bu)\bu\big)\,\dif W\bigg).
\end{split}
\end{equation}
In the brackets the stochastic integral is interpreted as an element of $X_N^\ast$.
Here $\mathcal{S}(\bfu)$ is a unique classical solution to \eqref{eq:galerkin11} with a strictly positive initial condition $\varrho_0\in C^{2+\kappa}(\mt)$, i.e. $0<\underline{\varrho}\leq\varrho_0\leq\overline{\varrho}$.
This classical solution exists (and belongs to $C([0,T];C^{2+\kappa}(\mt))$) provided $\bfu\in C([0,T],C^2(\mt))$. A maximum principle applies in this
case such that for all $x\in\mt$
\begin{align}\label{eq:max}
\underline{\varrho} \exp\bigg(-\int_0^t\|\Div\bfu\|_\infty\ds\bigg)\leq\mathcal{S}(\bu)(t,x)\leq\overline\varrho\exp\bigg(\int_0^t\|\Div\bfu\|_\infty\ds\bigg).
\end{align}
For the properties of $\mathcal S$ we refer to \cite[Lemma 2.2]{feireisl1}.
 The operators
$\mathcal{M}[\varrho]$
are invertible provided $\varrho$ is strictly positive. We further define
$$
\big\langle\mathcal{N}[\varrho,\bfu],\bfpsi\big\rangle=\int_{\mt}\big[\nu\Delta\bfu-\diver(\varrho\bfu\otimes\bfu)+\nabla\big((\lambda+\nu)\diver\bfu-a\varrho^\gamma-\delta\varrho^\beta\big)-\varepsilon\nabla\bfu\nabla\varrho\big]\cdot\bfpsi\,\dif x
$$
for all $\bfpsi\in X_N$. Note that for $\varrho$ and $\bfu$ satisfying the conditions above $\mathcal{N}[\varrho,\bfu]$ is well-defined.
In order to study \eqref{eq:galerkin1}, we shall fix some notation.
For $\mathbf{v}=\sum_{i=1}^N\alpha_i\bfpsi_i\in X_N$ and $R\in\mn$ let us define the following truncation operators
$$\bfv^R=\sum_{i=1}^N\theta_R(\alpha_i)\alpha_i\bfpsi_i.$$
Here $\theta_R$ is a smooth cut-off function with support in $[-2R,2R]$ such that $\theta(z)=1$ on $[-R,R]$. Note that by construction the mapping $\Theta_R:\bfv\mapsto \bfv^R$ satisfies
\begin{align}\label{thetaR}
\Theta_R:X_N\longrightarrow X_N,\quad \|\Theta_R(\bfv)-\Theta_R(\bfu)\|_{X_N}\leq C(N)\|\bfv-\bfu\|_{X_N},
\end{align}
for all $\bfu,\bfv\in X_N$.

Let $N\in\mn$, $R\in\mn$ be fixed. In the first step, we will solve the following problem \eqref{eq:galerkin1a} by using the Banach fixed point theorem in the Banach space $\mathcal{B}=L^2(\Omega;C([0,T_{\ast}];X_N))$ with $T_*$ sufficiently small. Repeating the same technique shows existence and uniqueness on the whole time interval $[0,T]$. Finally we pass to the limit as $R\rightarrow\infty$. Consider
\begin{equation}\label{eq:galerkin1a}
\begin{split}
\bu(t)&=\mathcal{M}^{-1}\big[\mathcal{S}\big(\bu^R\big)(t)\big]\bigg[\big(\varrho_0\bfu_0^R\big)^*+\int_0^t\mathcal{N}\Big[\mathcal{S}\big(\bu^R\big),\bu^R\Big]\dif s\\
&\qquad+\Theta_R\bigg(\int_0^t\varPhi^{N}\Big(\mathcal{S}\big(\bu^R\big),\mathcal{S}\big(\bu^R\big)\bu^R\Big)\,\dif W\bigg)\bigg]
\end{split}
\end{equation}
with $\bfu_0=\mathcal M^{-1}[\varrho_0]\bfq_0^*$. Note that now we have
$\bfu(0)=\bfu_0^R$.
Let $\mathscr{T}:\mathcal{B}\rightarrow\mathcal{B}$ be the operator defined by the above right hand side. We will show that it is a contraction.
The deterministic part $\mathscr{T}_{det}$ can be estimated using the approach of \cite[Section 2.3]{feireisl1} and there holds
$$
\|\mathscr{T}_{det}\bfu-\mathscr{T}_{det}\bfv\|_\mathcal{B}^2\leq T_{\ast}C(N,R,T_\ast)\|\bfu-\bfv\|_{\mathcal B}^2,
$$
where the constant does not depend on the initial condition.
In several points one needs the fact that we are working on a finite dimensional space: equivalence of norms is used and also Lipschitz continuity of $\mathcal M^{-1}$ in $\varrho$ (see \cite[(2.12)]{feireisl1}).
Let us focus on the stochastic part $\mathscr{T}_{sto}$. We have
\begin{equation*}
\begin{split}
\|\mathscr{T}_{sto}\bfu&-\mathscr{T}_{sto}\bfv\|_\mathcal{B}^2=\stred\sup_{0\leq t\leq T_{\ast}}\bigg\|\mathcal{M}^{-1}\big[\mathcal{S}\big(\bu^R\big)(t)\big]\Theta_R\bigg(\int_0^t\varPhi^{N}\Big(\mathcal{S}\big(\bu^R\big),\mathcal{S}\big(\bu^R\big)\bu^R\Big)\,\dif W\bigg)\\
&\hspace{1.5cm}\qquad-\mathcal{M}^{-1}\big[\mathcal{S}\big(\bfv^R\big)(t)\big]\Theta_R\bigg(\int_0^t\varPhi^{N}\Big(\mathcal{S}\big(\bfv^R\big),\mathcal{S}\big(\bfv^R\big)\bfv^R\Big)\,\dif W\bigg)\bigg\|_{X_N}^2\\
&\leq C\stred\sup_{0\leq t\leq T_{\ast}}\Big\|\mathcal{M}^{-1}\big[\mathcal{S}\big(\bu^R\big)(t)\big]-\mathcal{M}^{-1}\big[\mathcal{S}\big(\bfv^R\big)(t)\big]\Big\|_{\mathcal L(X_N^\ast,X_N)}^2\\
&\hspace{1.5cm}\qquad\times\bigg\|\Theta_R\bigg(\int_0^t\varPhi^{N}\Big(\mathcal{S}\big(\bu^R\big),\mathcal{S}\big(\bu^R\big)\bu^R\Big)\,\dif W\bigg)\bigg\|_{X_N}^2\\
&\;+ C\stred\sup_{0\leq t\leq T_{\ast}}\Big\|\mathcal{M}^{-1}\big[\mathcal{S}\big(\bfv^R\big)(t)\big]\Big\|_{\mathcal L(X_N^\ast,X_N)}^2\bigg\|\Theta_R\bigg(\int_0^t\varPhi^{N}\Big(\mathcal{S}\big(\bu^R\big),\mathcal{S}\big(\bu^R\big)\bu^R\Big)\,\dif W\bigg)\\
&\hspace{1.5cm}\qquad-\Theta_R\bigg(\int_0^t\varPhi^{N}\Big(\mathcal{S}\big(\bfv^R\big),\mathcal{S}\big(\bfv^R\big)\bfv^R\Big)\,\dif W\bigg)\bigg\|_{X_N}^2\\
&=\mathscr S_1+\mathscr S_2.
\end{split}
\end{equation*}
As a consequence of the assumption $\underline{\rho}>0$
we have by the definition of $\mathcal M$, \eqref{eq:max} and equivalence of norms a.s.
\begin{align*}
&\Big\|\mathcal{M}^{-1}\big[\mathcal{S}\big(\bfv^R\big)(t)\big]\Big\|_{\mathcal L(X_n^\ast,X_n)}^2\leq\Big(\inf_{x\in\mt}\mathcal{S}\big(\bfv^R\big)(t)\Big)^{-1}\\
&\quad \leq\Big(\underline{\rho}\exp\Big(-\int_0^{T_\ast}\|\Div \bfv^R\|_\infty\,\dif s\Big)\Big)^{-1}\leq C(N,R).
\end{align*}
Hence we gain by Burgholder-Davis-Gundy inequality
\begin{equation*}
\begin{split}
\mathscr S_2&\leq C(N,R)\,\stred\sup_{0\leq t\leq T_\ast}\bigg\|\int_0^t\varPhi^{N}\Big(\mathcal{S}\big(\bu^{R}\big),\mathcal{S}\big(\bu^{R}\big)\bu^{R}\Big)-\varPhi^{N}\Big(\mathcal{S}\big(\bfv^R\big),\mathcal{S}\big(\bfv^{R}\big)\bfv^{R}\Big)\,\dif W\bigg\|_{X_N}^2\\
&\leq C(N,R)\,\stred\int_0^{T_{\ast}}\sum_{k\geq1}\Big\|\, g_k^{N}\Big(\mathcal{S}\big(\bu^{R}\big),\mathcal{S}\big(\bu^{R}\big)\bu^{R}\Big)
-\, g_k^{N}\Big(\mathcal{S}\big(\bfv^{R}\big),\mathcal{S}\big(\bfv^{R}\big)\bfv^{R}\Big)\Big\|_{X_N}^2\dif s.
\end{split}
\end{equation*}
Due to the construction of $g_k^{N}$ in \eqref{eq:PhiepsN} we have
\begin{align}\label{eq:I}
\begin{aligned}
I&=\sum_{k\geq1}\Big\|\, g_k^{N}\Big(\mathcal{S}\big(\bu^R\big),\mathcal{S}\big(\bu^R\big)\bu^R\Big)
-\, g_k^{N}\Big(\mathcal{S}\big(\bfv^R\big),\mathcal{S}\big(\bfv^R\big)\bfv^R\Big)\Big\|_{X_N}^2\\
&\leq C\Big\|\mathcal{M}^\frac{1}{2}\big[\mathcal{S}\big(\bfu^R\big)\big]-\mathcal{M}^\frac{1}{2}\big[\mathcal{S}\big(\bfv^R\big)\big]\Big\|^2_{\mathcal{L}(X_N,X_N)}\sum_{k\geq1} \bigg\| \frac{g_k\big(\mathcal{S}\big(\bu^R\big),\mathcal{S}\big(\bu^R\big)\bu^{R}\big)}{\sqrt{\mathcal{S}\big(\bfu^R\big)}}\bigg\|^2_{L^2}\\
&\;+C\Big\|\mathcal{M}^\frac{1}{2}\big[S\big(\bfv^R\big)\big]\Big\|_{\mathcal{L}(X_N,X_N)}^2\sum_{k\geq1}\bigg\|\frac{g_k\big(\mathcal{S}\big(\bu^R\big),\mathcal{S}\big(\bu^R\big)\bu^{R}\big)}{\sqrt{\mathcal{S}\big(\bfu^R\big)}}-\frac{g_k\big(\mathcal{S}\big(\bfv^R\big),\mathcal{S}\big(\bfv^R\big)\bfv^{R}\big)}{\sqrt{\mathcal{S}\big(\bfv^R\big)}}\bigg\|_{L^2}^2\\
&=I_1+I_2.
\end{aligned}
\end{align}
Concerning the first term on the above right hand side, we apply Lemma \ref{lemma:aux}, \eqref{growth1}, \eqref{eq:max} and \cite[Lemma 2.2]{feireisl1} and obtain
\begin{align*}
\begin{aligned}
I_1&\leq C(N,R)\big\|\mathcal{S}\big(\bfu^R\big)-\mathcal{S}\big(\bfv^R\big)\big\|^2_{L^2}\leq T_\ast C(N,R,T_\ast)\sup_{0\leq t\leq T_\ast} \big\|\bfu^R-\bfv^R\big\|_{X_N}^2.
\end{aligned}
\end{align*}
%
For the second term on the right hand side of \eqref{eq:I} we make use of \eqref{23}
and conclude
\begin{align*}
I_2&\leq C(N,R)\Big(\big\|\mathcal{S}\big(\bfu^R\big)-\mathcal{S}\big(\bfv^R\big)\big\|_{L^2}^2+\big\|\bfu^R-\bfv^R\big\|_{L^2}^2\Big)\\
&\leq T_\ast C(N,R,T_\ast)\sup_{0\leq t\leq T_\ast} \big\|\bfu^R-\bfv^R\big\|_{X_N}^2,
\end{align*}
where we applied \cite[Lemma 2.2]{feireisl1} and the Lipschitz continuity  of
$$(\rho,\bfq)\mapsto \sum_{k\geq 1}\frac{g_k(\rho,\bfq)}{\sqrt{\rho}}.$$
The latter follows from \eqref{growth1} and \eqref{growth2} since we only consider $\rho\geq C(N,R)>0$. Consequently,
\begin{equation*}
\begin{split}
\mathscr S_2
&\leq T_{\ast}C(N,R,T_\ast)\|\bfu-\bfv\|_{\mathcal B}^2.
\end{split}
\end{equation*}
For $\mathscr S_1$ we have by \cite[(2.10), (2.12)]{feireisl1}
\begin{equation*}
\begin{split}
\mathscr S_1&\leq C(N,R)\,\stred\sup_{0\leq t\leq {T_{\ast}}}\Big\|\mathcal{S}\big(\bu^R\big)(t)-\mathcal{S}\big(\bfv^R\big)(t)\Big\|_{L^1}^2\\
&\leq T_{\ast}C(N,R,T_\ast)\,\stred\sup_{0\leq t\leq {T_{\ast}}}\big\|\bu^R-\bfv^R\big\|_{X_N}^2=T_{\ast}C(N,R,T_\ast)\|\bfu-\bfv\|_{\mathcal B}^2
\end{split}
\end{equation*}
hence plugging all together we have shown that
\begin{equation*}
\|\mathscr{T}_{sto}\bfu-\mathscr{T}_{sto}\bfv\|_\mathcal{B}^2\leq T_{\ast}C(N,R,T_\ast)\|\bfu-\bfv\|_{\mathcal B}^2.
\end{equation*}
Since we know that also the deterministic part in \eqref{eq:galerkin1a} is a contraction if $T_\ast$ is sufficiently small, we obtain
\begin{equation*}
\|\mathscr{T}\bfu-\mathscr{T}\bfv\|_\mathcal{B}^2\leq \kappa\|\bfu-\bfv\|_{\mathcal B}^2
\end{equation*}
with $\kappa\in(0,1)$. This allows us to apply Banach's fixed point theorem and we obtain a unique solution to \eqref{eq:galerkin1a} on the interval $[0,T_\ast]$. Extension of this existence and uniqueness result to the whole interval $[0,T]$ can be done by considering $kT_\ast,\,k\in\mn,$ as the new times of origin and solving \eqref{eq:galerkin1a} on each subinterval $[kT_\ast,(k+1)T_\ast]$. Note that the time $T_*$ chosen above does not depend on the initial datum.

\subsection{Passage to the limit as $R\rightarrow\infty$}

It follows from the previous section that for every $N\in\mn$ and $R\in\mn$ there exists a unique solution to \eqref{eq:galerkin1a}. As the next step, we keep $N$ fixed, denote the solution to \eqref{eq:galerkin1a} by $\tilde\bfu_R$ and we pass to the limit as $R\rightarrow \infty$ to obtain the existence of a unique solution to \eqref{eq:galerkin}. Towards this end, let us define
\begin{align*}
\tau_R&=\inf\Big\{t\in[0,T];\big\|\tilde\bfu_R(t)\big\|_{L^2}\geq R\Big\}\wedge\inf\bigg\{t\in[0,T];\bigg\|\int_0^t\varPhi^{N}\big(\mathcal{S}(\tilde\bfu_R),\mathcal{S}(\tilde\bfu_R)\tilde\bfu_R\big)\,\dif W\bigg\|_{L^2}\geq R\bigg\}
\end{align*}
(with the convention $\inf\emptyset=T$). Note that $\tau_R$ defines an $(\mf_t)$-stopping time and let $\tilde\varrho_R=\mathcal{S}(\tilde\bfu_R)$. Then $(\tilde\varrho_R,\tilde\bfu_R)$ is the unique solution to \eqref{eq:galerkin} on $[0,\tau_R)$. Besides, due to uniqueness, if $R'>R$ then $\tau_{R'}\geq\tau_R$ and $(\tilde\varrho_{R'},\tilde\bfu_{R'})=(\tilde\varrho_R,\tilde\bfu_{R})$ on $[0,\tau_R)$. Therefore, one can define $(\tilde\varrho,\tilde\bfu)$ by $(\tilde\varrho,\tilde \bfu):=(\tilde\varrho_R,\tilde\bfu_R)$ on $[0,\tau_R)$. In order to make sure that $(\tilde\varrho,\tilde\bfu)$ is defined on the whole time interval $[0,T]$, i.e. the blow up cannot occur in a finite time, we proceed with the basic energy estimate that will be used several times throughout the paper.

\begin{proposition}\label{prop:apriori1}
Let $p\in[1,\infty)$. Then the following estimate holds true
\begin{align}
&\stred\bigg[\sup_{0\leq t\leq T}\int_{\mt}\Big(\frac{1}{2}\tilde\varrho_R|\tilde\bu_R|^2+\frac{a}{\gamma-1}\tilde\varrho^\gamma_R+\frac{\delta}{\beta-1}\tilde\varrho^\beta_R\Big)\,\dif x\label{eq:apriori112}\\
&+\int_0^T\int_{\mt}\nu|\nabla\tilde\bu_R|^2+(\lambda+\nu)|\diver\tilde\bu_R|^2\,\dif x\,\dif s+\varepsilon\int_0^T\int_{\mt}\big(a\gamma\tilde\varrho_R^{\gamma-2}+\delta\beta\tilde\varrho_R^{\beta-2}\big)|\nabla\tilde\varrho_R|^2\,\dif x\,\dif s\bigg]^p\nonumber\\
& \leq C\bigg(1+\stred\bigg[\int_{\mt}\Big(\frac{1}{2}\varrho_0|\bu_0|^2+\frac{a}{\gamma-1}\varrho_0^\gamma+\frac{\delta}{\beta-1}\varrho_0^\beta\Big)\,\dif x\bigg]^p \bigg)\nonumber
\end{align}
with a constant independent of $R$, $N$, $\varepsilon$ and $\delta$.

\begin{proof}
In order to obtain this a priori estimate we observe that restricting ourselves to $[0,\tau_R)$ the two equations \eqref{eq:galerkin1a} and \eqref{eq:galerkin} coincide and we apply It\^{o}'s formula to the functional
$$f:L^2(\mt)\times X^*_N\longrightarrow \mr,\;(\rho,\bq)\longmapsto\frac{1}{2}\big\langle \bfq,\mathcal{M}^{-1}[\rho]\bfq\big\rangle,$$
where $\rho=\tilde\varrho_R$ and $\bfq= \tilde\varrho_R\tilde\bfu_R$. This corresponds exactly to testing by $ \tilde\bfu_R$ in the deterministic case. Indeed, there holds
$$\partial_\bfq f(\rho,\bfq)=\mathcal{M}^{-1}[\rho]\bfq\in X_N,\qquad\partial^2_{\bfq}f(\rho,\bfq)=\mathcal{M}^{-1}[\rho]\in \mathcal{L}(X^*_N,X_N)$$
and
$$\partial_\rho f(\rho,\bfq)=-\frac{1}{2}\big\langle\bfq,\mathcal{M}^{-1}[\rho]\mathcal{M}[\,\cdot\,]\mathcal{M}^{-1}[\rho]\bfq\big\rangle\in \mathcal{L}(L^2(\mt),\mr),$$
and therefore
$$f\big(\tilde\varrho_R,\tilde\varrho_R\tilde\bfu_R\big)=\frac{1}{2}\int_{\mt} \tilde\varrho_R|\tilde\bfu_R|^2\,\dif x,$$
$$\partial_\bfq f\big(\tilde\varrho_R,\tilde\varrho_R\tilde\bfu_R\big)=\tilde\bfu_R,\qquad\partial^2_{\bfq}f\big(\tilde\varrho_R,\tilde\varrho_R\tilde\bfu_R\big)=\mathcal{M}^{-1}[\tilde\varrho_R],$$
$$\partial_\rho f\big(\tilde\varrho_R,\tilde\varrho_R\tilde\bfu_R\big)=-\frac{1}{2}|\tilde\bfu_R|^2.$$
We obtain	
\begin{align*}
\frac{1}{2}\int_ {\mathbb T^3}&  \tilde\varrho(t\wedge\tau_R)\big| \tilde\bfu(t\wedge\tau_R)\big|^2\dx=\frac{1}{2}\int_ {\mathbb T^3}  \varrho_0| \bfu_0|^2\dx-\nu\int_0^{t\wedge\tau_R}\int_ {\mathbb T^3}|\nabla \tilde\bfu|^2\dxs\\
&\quad-(\lambda+\nu)\int_0^{t\wedge\tau_R}\int_{\mt}|\diver\tilde\bfu|^2\dxs\\
&\quad+\int_0^{t\wedge\tau_R}\int_ {\mathbb T^3}\varrho \tilde\bfu\otimes \tilde\bfu:\nabla \tilde\bfu\dxs-\varepsilon\int_0^{t\wedge\tau_R}\int_ {\mathbb T^3}\nabla\tilde\bfu\nabla\tilde \varrho\cdot\tilde\bfu\dxs\\
&\quad+a\int_0^{t\wedge\tau_R}\int_ {\mathbb T^3}\tilde \varrho^\gamma\Div \tilde\bfu\dxs+\delta\int_0^{t\wedge\tau_R}\int_ {\mathbb T^3}\tilde \varrho^\beta\Div \tilde\bfu\dxs\\
&\quad+\sum_{k\geq1}\int_0^{t\wedge\tau_R}\int_ {\mathbb T^3} \tilde\bfu\cdot g^{N}_k(\tilde\varrho, \tilde\varrho \tilde\bfu)\dx\,\dif\beta_k(\sigma)
+\frac{\varepsilon}{2}\int_0^{t\wedge\tau_R}\int_ {\mathbb T^3}\nabla| \tilde\bfu|^2\cdot\nabla\tilde \varrho\,\dif x\,\dif\sigma\\
&\quad-\frac{1}{2}\int_0^{t\wedge\tau_R}\int_{\mt}\nabla|\tilde\bfu|^2\cdot\tilde\varrho\tilde\bfu\,\dif x\,\dif \sigma+\frac{1}{2}\sum_{k\geq1}\int_0^{t\wedge\tau_R}\big\langle\mathcal{M}^{-1}[\tilde \varrho]g_k^{N}(\tilde\varrho,\tilde\varrho\tilde\bfu),g_k^{N}(\tilde\varrho,\tilde\varrho\tilde\bfu)\big\rangle\,\dif\sigma\\
&=J_1+\cdots+J_{11}.
\end{align*}
Now, we observe that $J_5+J_9=0$, $J_4+J_{10}=0$,
\begin{equation*}
\begin{split}
J_6&=-\frac{a}{\gamma-1}\int_0^{t\wedge\tau_R}\int_{\mt}\partial_t\tilde\varrho^\gamma\,\dif x\,\dif \sigma-\varepsilon a\gamma\int_0^{t\wedge\tau_R}\int_{\mt}\tilde\varrho^{\gamma-2}|\nabla\tilde\varrho|^2\,\dif x\,\dif \sigma,
\end{split}
\end{equation*}
similarly for $J_7$. Due to definitions of $g^{N}$ and $\mathcal{M}^{-1}$ we have 
\begin{align*}
\sum_{k\geq 1}\big\langle\mathcal{M}^{-1}[\tilde\varrho]g_k^{N}(\tilde\varrho,\tilde\varrho\tilde\bfu)&,g_k^{N}(\tilde\varrho,\tilde\varrho\tilde\bfu)\big\rangle=\sum_{k\geq 1}\big\langle \mathcal{M}^{-\frac{1}{2}}[\tilde\varrho]g_k^{N}(\tilde\varrho,\tilde\varrho\tilde\bfu),\mathcal{M}^{-\frac{1}{2}}[\tilde\varrho]g_k^{N}(\tilde\varrho,\tilde\varrho\tilde\bfu)\big\rangle\\
&=\sum_{k\geq 1}\int_{\mt}\Big|P_N\Big( \frac{g_k(\tilde\varrho,\tilde\varrho\tilde\bfu)}{\sqrt{\tilde\varrho}}\Big)\Big|^2\,\dif x
\leq\sum_{k\geq 1}\int_{\mt}\Big|\frac{g_k(\tilde\varrho,\tilde\varrho\tilde\bfu)}{\sqrt{\tilde\varrho}}\Big|^2\,\dif x\\
&\leq C\int_{\mt}\big(\tilde\varrho+\tilde\varrho^\gamma+\tilde\varrho|\tilde\bfu|^2\big)\,\dif x.
\end{align*}
Here we also used continuity of $P_N$ on $L^2(\mt)$ and \eqref{growth1}.
We get
\begin{equation*}
J_{11}
\leq \,C\int_0^{t\wedge\tau_R}\int_{\mt}(1+\tilde\varrho^\gamma+\tilde\varrho|\tilde\bfu|^2)\,\dif x\,\dif \sigma.
\end{equation*}
Hence according to the Gronwall lemma we can write
\begin{align*}
&\stred\int_ {\mathbb T^3} \Big(\frac{1}{2} \tilde\varrho(t\wedge\tau_R)\big|\tilde \bfu(t\wedge\tau_R)\big|^2+\frac{a}{\gamma-1}\tilde\varrho^\gamma(t\wedge\tau_R)+\frac{\delta}{\beta-1}\tilde\varrho^\beta(t\wedge\tau_R)\Big)\dx\\
&\quad+\stred\bigg[\int_0^{t\wedge\tau_R}\int_ {\mathbb T^3}\nu |\nabla \tilde\bfu|^2+(\lambda+\nu)|\diver\tilde\bfu|^2+\varepsilon\big(a\gamma\tilde\varrho^{\gamma-2}+\delta\beta\tilde\varrho^{\beta-2}\big)|\nabla\tilde\varrho|^2\,\dif x\,\dif s\bigg]\\
& \leq C\bigg(1+\stred\int_ {\mathbb T^3} \Big(\frac{1}{2} \varrho_0| \bfu_0|^2+\frac{a}{\gamma-1}\varrho_0^\gamma+\frac{\delta}{\beta-1}\varrho_0^\beta\Big)\dif x\bigg).
\end{align*}
Let us now take supremum in time, $p$-th power and expectation. For the stochastic integral $J_8$ we make use of the Burkholder-Davis-Gundy inequality and the assumption \eqref{growth1} to obtain, for all $t\in[0,T]$,
\begin{equation*}
\begin{split}
\E\sup_{0\leq s\leq t\wedge\tau_R}&|J_8|^{p}\leq C\,\E\bigg[\int_0^{t\wedge\tau_R}\sum_{k\geq1}\bigg(\int_ {\mathbb T^3}\tilde \bfu\cdot g^{N}_k\big(\tilde\varrho,\tilde\varrho\tilde \bfu\big)\dx\bigg)^2\dif s\bigg]^{\frac{p}{2}}\\
&\hspace{-.5cm}= C\,\E\bigg[\int_0^{t\wedge\tau_R}\sum_{k\geq1}\bigg(\int_ {\mathbb T^3}\mathcal M^{\frac{1}{2}}[\tilde\varrho]\tilde \bfu\cdot P_N\bigg(\frac{g_k\big(\tilde\varrho,\tilde\varrho \tilde \bfu\big)}{\sqrt{\tilde\varrho}}\bigg)\dx\bigg)^2\dif s\bigg]^{\frac{p}{2}}\\
&\hspace{-.5cm}\leq C\,\stred\bigg[\int_0^{t\wedge\tau_R}\bigg(\int_{\mt}\big|\mathcal M^{\frac{1}{2}}[\tilde\varrho]\tilde\bu\big|^2\,\dif x\bigg)\bigg(\int_{\mt}\Big|\frac{g_k\big(\tilde\varrho,\tilde\varrho \tilde \bfu\big)}{\sqrt{\tilde\varrho}}\Big|^2\,\dif x\bigg)\,\dif s\bigg]^{\frac{p}{2}}\\&\hspace{-.5cm}\leq C\,\stred\bigg[\int_0^{t\wedge\tau_R}\bigg(\int_{\mt}\mathcal M[\tilde\varrho]\tilde\bfu\cdot\tilde\bu\,\dif x\bigg)\bigg(\int_{\mt}\big(\tilde\varrho+\tilde\varrho^\gamma+\tilde\varrho|\tilde\bfu|^2\big)\,\dif x\bigg)\,\dif s\bigg]^{\frac{p}{2}}\\
&\hspace{-.5cm}\leq \kappa \,\stred\bigg(\sup_{t\wedge\tau_R}\int_{\mt}\tilde\varrho|\tilde\bu|^2\,\dif x\bigg)^p\,\dif s+C(\kappa)\,\stred\int_0^{t\wedge\tau_R}\bigg(\int_{\mt}\big(\tilde\varrho+\tilde\varrho^\gamma+\tilde\varrho|\tilde\bfu|^2\big)\,\dif x\bigg)^p\,\dif s.
\end{split}
\end{equation*}
Finally, taking $\kappa$ small enough and using the Gronwall lemma completes the proof.
\end{proof}
\end{proposition}

\begin{corollary}
It holds that
$$\prst\Big(\sup_{R\in\mn}\tau_R=T\Big)=1$$
and as a consequence the process $(\tilde\varrho,\tilde\bfu)$ is the unique solution to \eqref{eq:galerkin} on $[0,T]$.

\begin{proof}
Since
\begin{align}\label{22}
\prst\Big(\sup_{R\in\mn}\tau_R<T\Big)\leq\prst\big(\tau_R<T\big)&\leq\prst\Big(\sup_{0\leq t\leq T}\|\tilde \bfu_R(t)\|_{L^2}\geq R\Big)\\
&\;\; +\prst\bigg(\sup_{0\leq t\leq T}\bigg\|\int_0^t\varPhi^{N}\big(\mathcal{S}(\tilde\bfu_R),\mathcal{S}(\tilde\bfu_R)\tilde\bfu_R\big)\,\dif W\bigg\|_{L^2}\geq R\bigg)\nonumber
\end{align}
for all $R$, it is enough to show that the right hand side converges to zero as $R\rightarrow\infty$.
To this end, we recall the maximum principle for $\tilde\varrho_R$ \eqref{eq:max} and gain
$$\underline{\varrho}\exp\Big(-\int_0^t\|\Div\tilde \bfu_R\|_\infty\,ds\Big)\leq\tilde\varrho_R(t,x)\leq \overline{\varrho}\exp\Big(\int_0^t\|\Div\tilde \bfu_R\|_\infty\,ds\Big).$$
Since $\tilde\bfu_R\in\mathcal{B }=L^2(\Omega;C([0,T];X_N))$ and all the norms on $X_N$ are equivalent, the above left hand side can be further estimated from below by
$$\underline{\varrho}\exp\Big(-T-c\,\int_0^T\|\nabla\tilde \bfu_R\|^2_{L^2}\,ds\Big)\leq\tilde\varrho_R(t,x).$$
Plugging this into \eqref{eq:apriori112} we infer that
\begin{equation}\label{eq:apriori11}
\stred\bigg[\exp\Big(-c\int_0^T\|\nabla\tilde \bfu_R\|^2_{L^2}\,ds\Big)\sup_{0\leq t\leq T}\|\tilde\bfu_R\|^2_{L^2}\bigg]\leq \tilde c.
\end{equation}
Next, let us fix two increasing sequences $(a_R)$ and $(b_R)$ such that $a_R,\,b_R\rightarrow\infty$ and $a_R\,\mathrm{e}^{b_R}=R$ for each $R\in\mn$. As in \cite{franco}, we introduce the following events 
\begin{align*}
A&=\bigg[\exp\Big(-c\int_0^T\|\nabla\tilde \bfu_R\|^2_{L^2}\,ds\Big)\sup_{0\leq t\leq T}\|\tilde\bfu_R\|^2_{L^2}\leq a_R\bigg]\\
B&=\bigg[c\,\int_0^T\|\nabla\tilde\bfu_R\|_{L^2}^2\dif t\leq b_R\bigg]\\
C&=\bigg[\sup_{0\leq t\leq T}\|\tilde\bfu_R\|^2_{L^2}\leq a_R\,\mathrm{e}^{b_R}\bigg].
\end{align*}
Then $A\cap B\subset C$ because on $A\cap B$ there holds that
\begin{align*}
\sup_{0\leq t\leq T}\|\tilde\bfu_R\|_{L^2}^2&=\mathrm{e}^{b_R}\mathrm{e}^{-b_R}\sup_{0\leq t\leq T}\|\tilde\bfu_R\|_{L^2}^2\\
&\leq\mathrm{e}^{b_R}\exp\Big(-c\,\int_0^T\|\nabla\tilde \bfu_R\|^2_{L^2}\,ds\Big)\sup_{0\leq t\leq T}\|\tilde\bfu_R\|_{L^2}^2\leq\mathrm{e}^{b_R} a_R.
\end{align*}
Furthermore, according to \eqref{eq:apriori112}, \eqref{eq:apriori11} and the Chebyshev inequality
\begin{align*}
\prst(A)\geq 1-\frac{C}{a_R},\qquad\prst(B)\geq 1-\frac{C}{b_R}.
\end{align*}
Due to the general inequality for probabilities $\prst(C)\geq \prst(A)+\prst(B)-1$ we deduce that
$$\prst(C)\geq 1-\frac{C}{a_R}-\frac{C}{b_R}\longrightarrow 1,\qquad R\rightarrow\infty.$$
This yields the desired convergence of the first term on the right hand side of \eqref{22}.

For the second term, we have due to equivalence of norms on $X_N$ and Burkholder-Davis-Gundy inequality
\begin{align*}
\stred\sup_{0\leq t\leq T}\bigg\|\int_0^t\varPhi^{N}(\tilde\varrho_R,\tilde\varrho_R\tilde\bfu_R)\,\dif W\bigg\|_{L^2}^2&\leq C\,\stred\sup_{0\leq t\leq T}\bigg\|\int_0^t\varPhi^{N}(\tilde\varrho_R,\tilde\varrho_R\tilde\bfu_R)\,\dif W\bigg\|_{W^{-l,2}}^2\\
&\leq C\,\stred\int_0^T\sum_{k\geq 1}\Big\|\mathcal{M}^\frac{1}{2}[\tilde\varrho_R]P_N\frac{g_k(\tilde\varrho_R,\tilde\varrho_R\tilde\bfu_R)}{\sqrt{\tilde\varrho_R}}\Big\|_{W^{-l,2}}^2\,\dif r.
\end{align*}
Next, there holds
\begin{align}\label{a1}
\begin{aligned}
\sum_{k\geq 1}\Big\|\mathcal{M}^\frac{1}{2}[\tilde\varrho_R]P_N\frac{g_k(\tilde\varrho_R,\tilde\varrho_R\tilde\bfu_R)}{\sqrt{\tilde\varrho_R}}\Big\|_{W^{-l,2}}^2&=\sum_{k\geq 1}\sup_{\substack{\bfpsi\in W^{l,2}\\ \|\bfpsi\|_{W^{l,2}}\leq 1}}\Big|\Big\langle\mathcal{M}^\frac{1}{2}[\tilde\varrho_R]P_N\frac{g_k(\tilde\varrho_R,\tilde\varrho_R\tilde\bfu_R)}{\sqrt{\tilde\varrho_R}},\bfpsi\Big\rangle\Big|^2\\
&=\sum_{k\geq 1}\sup_{\substack{\bfpsi\in W^{l,2}\\ \|\bfpsi\|_{W^{l,2}}\leq 1}}\Big|\Big\langle P_N\frac{g_k(\tilde\varrho_R,\tilde\varrho_R\tilde\bfu_R)}{\sqrt{\tilde\varrho_R}},\mathcal{M}^\frac{1}{2}[\varrho_N]\bfpsi\Big\rangle\Big|^2\\
&\leq \sum_{k\geq 1}\Big\|\frac{g_k(\tilde\varrho_R,\tilde\varrho_R\tilde\bfu_R)}{\sqrt{\tilde\varrho_R}}\Big\|_{L^2}^2\sup_{\substack{\bfpsi\in W^{l,2}\\ \|\bfpsi\|_{W^{l,2}}\leq 1}}\big\|\mathcal{M}^\frac{1}{2}[\tilde\varrho_R]\bfpsi\big\|_{L^2}^2.
\end{aligned}
\end{align}
We further estimate using \eqref{growth1}
\begin{align}\label{a2}
\begin{aligned}
\sum_{k\geq 1}\Big\|\frac{g_k(\tilde\varrho_R,\tilde\varrho_R\tilde\bfu_R)}{\sqrt{\tilde\varrho_R}}\Big\|_{L^2}^2&\leq C\big(1+\|\tilde\varrho_R\|_{L^\gamma}^\gamma+\|\sqrt{\tilde\varrho_R}\tilde\bfu_R\|_{L^2}^2\big)
\end{aligned}
\end{align}
and
\begin{align}\label{a3}
\begin{aligned}
\sup_{\substack{\bfpsi\in W^{l,2}\\ \|\bfpsi\|_{W^{l,2}}\leq 1}}&\big\|\mathcal{M}^\frac{1}{2}[\tilde\varrho_R]\bfpsi\big\|_{L^2}^2=\sup_{\substack{\bfpsi\in W^{l,2}\\ \|\bfpsi\|_{W^{l,2}}\leq 1}}\big\langle\mathcal{M}[\tilde\varrho_R]\bfpsi,\bfpsi\big\rangle\\
&\leq \sup_{\substack{\bfpsi\in W^{l,2}\\ \|\bfpsi\|_{W^{l,2}}\leq 1}}\big\|\mathcal{M}[\tilde\varrho_R]\bfpsi\big\|_{L^2}\leq \sup_{\substack{\bfpsi\in W^{l,2}\\ \|\bfpsi\|_{W^{l,2}}\leq 1}}\big\|\tilde\varrho_R P_N\bfpsi\big\|_{L^2}\\
&\leq \sup_{\substack{\bfpsi\in W^{l,2}\\ \|\bfpsi\|_{W^{l,2}}\leq 1}}\big\|\tilde\varrho_R\big\|_{L^2}\| P_N\bfpsi\|_{L^\infty}\leq C\sup_{\substack{\bfpsi\in W^{l,2}\\ \|\bfpsi\|_{W^{l,2}}\leq 1}}\big\|\tilde\varrho_R\big\|_{L^2}\| P_N\bfpsi\|_{W^{l,2}}\\
&\leq C\sup_{\substack{\bfpsi\in W^{l,2}\\ \|\bfpsi\|_{W^{l,2}}\leq 1}}\big\|\tilde\varrho_R\big\|_{L^2}\|\bfpsi\|_{W^{l,2}}\leq C\|\tilde\varrho_R\|_{L^2}.
\end{aligned}
\end{align}
Altogether we deduce
\begin{align*}
\stred\sup_{0\leq t\leq T}&\bigg\|\int_0^t\varPhi^{N}(\tilde\varrho_R,\tilde\varrho_R\tilde\bfu_R)\,\dif W\bigg\|_{L^2}^2\leq C\,\stred\int_0^T\|\tilde\varrho_R\|_{L^2}\big(1+\|\tilde\varrho_R\|_{L^\gamma}^\gamma+\|\sqrt{\tilde\varrho_R}\tilde\bfu_N\|_{L^2}^2\big)\,\dif r\\
&\leq C\bigg(1+\stred\sup_{0\leq t\leq T}\|\tilde\varrho_R\|_{L^2}^{2}+\stred\sup_{0\leq t\leq T}\|\tilde\varrho_R\|_{L^\gamma}^{2\gamma}+\stred\sup_{0\leq t\leq T}\|\sqrt{\tilde\varrho_R}\tilde\bfu_N\|_{L^2}^{4}\bigg)\leq C,
\end{align*}
where we used \eqref{eq:apriori112} with $p=2$. Finally, the convergence of the second term on the right hand side of \eqref{22} follows from Chebyshev's inequality and the proof is complete.
\end{proof}
\end{corollary}

\section{The viscous approximation}
\label{sec:viscous}

In this section, we continue with our proof of Theorem \ref{thm:main} and prove existence of a martingale solution to the viscous approximation \eqref{eq:approx} with the initial law $\Gamma$ (see the beginning of Section \ref{sec:galerkin} for its definition), where $\varepsilon,\delta$ are fixed. In particular, we justify the passage to the limit in \eqref{eq:galerkin} as $N\rightarrow\infty$. Let $(\varrho_N,\bfu_N)$ denote the solution to \eqref{eq:galerkin} and observe that by the same approach as in Proposition \ref{prop:apriori1} it can be shown that it satisfies the corresponding a priori estimate uniformly in $N$. In fact, there holds for any $p<\infty$
\begin{align}
&\stred\bigg[\sup_{0\leq t\leq T}\int_{\mt}\Big(\frac{1}{2}\varrho_N|\bu_N|^2+\frac{a}{\gamma-1}\varrho^\gamma_N+\frac{\delta}{\beta-1}\varrho^\beta_N\Big)\,\dif x\label{eq:aprioriN}\\
&+\int_0^T\int_{\mt}\nu|\bu_N|^2+(\lambda+\nu)|\diver\bu_N|^2\,\dif x\,\dif s+\varepsilon\int_0^T\int_{\mt}\big(a\gamma\varrho_N^{\gamma-2}+\delta\beta\varrho_N^{\beta-2}\big)|\nabla\varrho_N|^2\,\dif x\,\dif s\bigg]^p\nonumber\\
& \leq C_p\bigg(1+\stred\bigg[\int_{\mt}\Big(\frac{1}{2}\varrho_0|\bu_0|^2+\frac{a}{\gamma-1}\varrho_0^\gamma+\frac{\delta}{\beta-1}\varrho_0^\beta\Big)\,\dif x\bigg]^p \bigg)\nonumber
\end{align}
uniformly in $N$, $\varepsilon$ and $\varrho$. 
Thus we obtain uniform bounds in the following spaces
\begin{align}
\bfu_N&\in L^p(\Omega;L^2(0,T;W^{1,2}(\mt))),\label{apv1}\\
\sqrt{\varrho_N}\bfu_N&\in L^p(\Omega;L^\infty(0,T;L^2(\mt))),\label{aprhov1}\\
\varrho_N&\in L^p(\Omega;L^\infty(0,T;L^\beta(\mt))),\label{aprho1}\\
\sqrt{\varepsilon\delta}\,(\varrho_N)^{\beta/2}&\in L^p(\Omega;L^2(0,T;W^{1,2}(\mt))).\label{aprhoo1}
\end{align}
We recall that $\beta>\max\set{\tfrac{9}{2},\gamma}$. Here $p\in[1,\infty)$ is arbitrary due to \eqref{initial2} and the estimate of $\bfu_N\in L^p(\Omega;L^2(0,T;L^2(\mt)))$ is obtained as in \cite[Remark 5.1, page 4]{Li2}.
Besides, testing \eqref{eq:galerkin11} by $\varrho_N$ yields
\begin{align*}
\int_{\mathbb T^3}|\varrho_N|^2\dx+2\varepsilon\int_0^t\int_{\mathbb T^3}|\nabla\varrho_N|^2\dxs=\int_{\mathbb T^3}|\varrho_0|^2\dx-\int_{0}^t\int_{\mathbb T^3}\Div\bfu_N\,|\varrho_N|^2\dxs.
\end{align*}
And therefore since $\beta>\max\{\frac{9}{2},\gamma\}$, (\ref{apv1}) and (\ref{aprho1}) imply for any $p\in[1,\infty)$
\begin{align*}
\E\bigg[\int_0^T\int_{\mathbb T^3}\varepsilon\,|\nabla\varrho_N|^2\dxs\bigg]^p\leq C\,\E\bigg[1+\int_{0}^T\int_{\mathbb T^3}|\nabla\bfu_N|^2\dxs+\int_{0}^T\int_{\mathbb T^3}|\varrho_N|^4\dxs\bigg]^p\leq C.
\end{align*}
This yields the uniform bound
\begin{equation}\label{est:nablarho1}
\sqrt{\varepsilon}\varrho_N\in L^p(\Omega;L^2(0,T;W^{1,2}(\mt))).
\end{equation}
Moreover, from \eqref{aprho1} and \eqref{aprhoo1} we obtain by interpolation that
$$\stred\bigg[\int_0^T\|\varrho_N^\beta\|_{L^2(\mt)}^{2}\,\dif t\bigg]^p\leq \,\stred\bigg[\sup_{0\leq t\leq T}\|\varrho_N^\beta\|_{L^1(\mt)}\bigg]^{p}+\,\stred\bigg[\int_0^T\|\varrho_N^\beta\|^{\frac{3}{2}}_{L^3(\mt)}\,\dif t\bigg]^{2p}\leq C.$$
In particular we obtain a uniform bound
\begin{equation}\label{aprhobeta}
\varrho_N\in L^p(\Omega;L^{\beta+1}(Q)),\quad Q=(0,T)\times\mt,
\end{equation}
for all $p\in[1, \infty)$ as $\beta>\max\{\frac{9}{2},\gamma\}$.

\subsection{Compactness and identification of the limit}

Let us now prepare the setup for our compactness method. We define the path space $\mathcal{X}=\mathcal{X}_\varrho\times\mathcal{X}_\bu\times\mathcal{X}_{\varrho\bfu}\times\mathcal{X}_{\varrho_0}\times\mathcal{X}_W$ where\footnote{If a topological space $X$ is equipped with the weak topology we write $(X,w)$.}
\begin{align*}
\mathcal{X}_\varrho&=C_w([0,T];L^\beta(\mt))\cap L^4(0,T;L^4(\mt))\cap (L^2(0,T;W^{1,2}(\mt)),w),\\
\mathcal{X}_\bu&=\big(L^2(0,T;W^{1,2}(\mt)),w\big),\quad
\mathcal{X}_{\varrho\bfu}=C_w([0,T];W^{-\frac{1}{2},2}(\mt)),\\
\mathcal{X}_{\varrho_0}&=L^2(\mt),\quad\qquad\qquad\quad\qquad\,\mathcal{X}_W=C([0,T];\mathfrak{U}_0).
\end{align*}
Let us denote by $\mu_{\varrho_N}$, $\mu_{\bu_N}$, $\mu_{P_N(\varrho_N\bfu_N)}$ and $\mu_{\varrho_0}$, respectively, the law of $\varrho_N$, $\bu_N$, $P_N(\varrho_N\bfu_N)$ and $\varrho_N(0)=\varrho_0$ on the corresponding path space. By $\mu_W$ we denote the law of $W$ on $\mathcal{X}_W$. The joint law of all variables on $\mathcal{X}$ is denoted by $\mu^N$.

\begin{proposition}\label{prop:bfutightness}
The set $\{\mu_{\bu_N};\,N\in\mn\}$ is tight on $\mathcal{X}_\bu$.

\begin{proof} 
The proof follows directly from \eqref{apv1}. Indeed, for any $R>0$ the set
$$B_R=\big\{\bu\in L^2(0,T;W^{1,2}(\mt));\, \|\bu\|_{L^2(0,T;W^{1,2}(\mt))}\leq R\big\}$$
is relatively compact in $\mathcal{X}_\bu$ and
\begin{equation*}
 \begin{split}
  \mu_{\bu_N}(B_R^c)=\prst\big(\|\bu_N\|_{L^2(0,T;W^{1,2}(\mt))}\geq R\big)\leq\frac{1}{R}\stred\|\bu_N\|_{L^2(0,T;W^{1,2}(\mt))}\leq \frac{C}{R}
 \end{split}
\end{equation*}
which yields the claim.
\end{proof}
\end{proposition}

\begin{proposition}\label{prop:rhotight}
The set $\{\mu_{\varrho_N};\,N\in\mn\}$ is tight on $\mathcal{X}_\varrho$. 
\begin{proof}
Due to \eqref{aprhov1} and \eqref{aprho1} we obtain that
\begin{equation}\label{estrhou}
\{\varrho_N\bfu_N\}\quad\text{is bounded in}\quad L^p(\Omega;L^\infty(0,T;L^{\frac{2\beta}{\beta+1}}(\mt)))
\end{equation}
hence $\{\diver(\varrho_N\bu_N)\}$ is bounded in $L^{p}(\Omega;L^\infty(0,T;W^{-1,\frac{2\beta}{\beta+1}}(\mt)))$ and similarly $\{\varepsilon\Delta\varrho_N\}$ is bounded in $L^p(\Omega;L^\infty(0,T;W^{-2,2}(\mt)))$. As a consequence,
$$\stred\|\varrho_N\|^{p}_{C^{0,1}([0,T];W^{-2,\frac{2\beta}{\beta+1}}(\mt))}\leq C$$
due the continuity equation \eqref{eq:galerkin11}.
Now, the required tightness in $C_w([0,T];L^\beta(\mt))$ follows by a similar reasoning as in Proposition \ref{prop:bfutightness} together with the compact embedding (see \cite[Corollary B.2]{on2})
$$L^\infty(0,T;L^\beta(\mt))\cap C^{0,1}([0,T];W^{-2,\frac{2\beta}{\beta+1}}(\mt))\overset{c}{\hookrightarrow} C_w([0,T];L^\beta(\mt)).$$

Next, observe that by applying interpolation to \eqref{aprho1} and \eqref{est:nablarho1} we obtain
\begin{align*}
\stred\int_0^T\|\varrho_N\|_{W^{\frac{1}{2},\frac{4\beta}{\beta+2}}}^4\dif t\leq \stred\sup_{0\leq t\leq T}\|\varrho_N\|_{L^\beta}^{4}+\,\stred\bigg[\int_0^T\|\varrho_N\|_{W^{1,2}}^2\dif t\bigg]^{2}\leq C.
\end{align*}
Since $W^{\kappa,q}$ is compactly embedded into $L^4$ we make use of the Aubin-Lions compact embedding
$$L^4(0,T;W^{\kappa,q}(\mt))\cap C^{0,1}([0,T];W^{-2,\frac{2\beta}{\beta+1}}(\mt))\overset{c}{\hookrightarrow} L^4(Q)$$
and conclude as in Proposition \ref{prop:bfutightness}.

Tightness in $(L^2(0,T;W^{1,2}(\mt)),w)$ follows directly from \eqref{est:nablarho1} which completes the proof.
\end{proof}
\end{proposition}

\begin{proposition}\label{prop:rhoutight}
The set $\{\mu_{P_N(\varrho_N\bu_N)};\,N\in\mn\}$ is tight on $\mathcal{X}_{\varrho\bu}$.

\begin{proof}
First, we shall study time regularity of $P_N(\varrho_N\bu_N)$. Towards this end, let us decompose $P_N(\varrho_N\bu_N)$ into two parts, namely, $P_N(\varrho_N\bu_N)(t)=Y^N(t)+Z^N(t)$, where
\begin{equation*}
 \begin{split}
Y^N(t)&=P_N\bfq(0)-\int_0^tP_N\big[\diver(\varrho_N\bu_N\otimes\bu_N)+\nu\Delta\bu_N+(\lambda+\nu)\nabla\diver\bfu_N\\
&\hspace{3.8cm}-a\nabla \varrho_N^\gamma-\delta\nabla \varrho_N^\beta\big]\dif s+\int_0^t\,\varPhi^N(\varrho_N,\varrho_N\bu_N) \,\dif W(s),\\
Z^N(t)&=\varepsilon\int_0^tP_N\big[\nabla\bu_N\nabla\varrho_N\big]\,\dif s,
 \end{split}
\end{equation*}
and consider them separately.

{\em H\"older continuity of $(Z^N)$.}
We show that there exists $\kappa\in(0,1)$ such that
\begin{equation}\label{eq:holderZ}
\stred\|Z^N\|_{C^\kappa([0,T];W^{-l,2}(\mt))}\leq C.
\end{equation}
To this end, we observe that according to \eqref{apv1}, \eqref{aprho1} and the embedding $W^{1,2}(\mt)\hookrightarrow L^6(\mt)$ there holds
$$\stred\|\varrho_N\bu_N\|^p_{L^2_tL_x^{\frac{6\beta}{\beta+6}}}\leq C\,\stred\sup_{0\leq t\leq T}\|\varrho_N\|_{L^\beta_x}^{2p}+C\,\stred\|\bu_N\|_{L^2_tL^6_x}^{2p}\leq C.$$
By interpolation with \eqref{estrhou} (and noticing that $\beta>4$) there exists $r>2$ such that we have a uniform bound in
$$\varrho_N\bu_N\in L^p(\Omega;L^r(0,T;L^2(\mt))).$$
Now we have all in hand to apply maximal regularity estimates to \eqref{eq:galerkin11} with
$$\diver(\varrho_N\bfu_N)\in L^p(\Omega;L^r(0,T;W^{-1,2}(\mt)))$$
as a right hand side and deduce a uniform estimate in
\begin{equation}\label{estnablarho}
\varrho_N\in L^p(\Omega;L^r(0,T;W^{1,2}(\mt))).
\end{equation}
Finally, we combine this with \eqref{apv1} and the continuity of $P_N$ on $W^{-l,2}(\mt)$ and \eqref{eq:holderZ} follows.

{\em H\"older continuity of $(Y^N)$.}
As the next step, we prove that there exist $\vartheta>0$ and $m>5/2$ such that
\begin{equation}\label{eq:holderY1}
\stred\big\|Y^N\|_{C^\vartheta([0,T];W^{-m,2}(\mt))}\leq C.
\end{equation}

Let us now estimate the stochastic integral. Due to Burkholder-Davis-Gundy inequality we obtain for any $\theta\geq 2$
\begin{align*}
\stred&\bigg\|\int_s^t\varPhi^N(\varrho_N,\varrho_N\bfu_N)\,\dif W\bigg\|_{W^{-l,2}}^\theta\leq C\,\stred\bigg(\int_s^t\sum_{k\geq 1}\Big\|\mathcal{M}^\frac{1}{2}[\varrho_N]P_N\frac{g_k(\varrho_N,\varrho_N\bfu_N)}{\sqrt{\varrho_N}}\Big\|_{W^{-l,2}}^2\,\dif r\bigg)^{\theta/2}.
\end{align*}
Here, we can apply the estimates established in \eqref{a1} - \eqref{a3}
and deduce
\begin{align*}
\stred\bigg\|\int_s^t&\varPhi^N(\varrho_N,\varrho_N\bfu_N)\,\dif W\bigg\|_{W^{-l,2}}^\theta\leq C\,\stred\bigg(\int_s^t\|\varrho_N\|_{L^2}\big(1+\|\varrho_N\|_{L^\gamma}^\gamma+\|\sqrt{\varrho_N}\bfu_N\|_{L^2}^2\big)\,\dif r\bigg)^{\theta/2}\\
&\leq C|t-s|^{\theta/2}\,\stred\bigg[\sup_{0\leq t\leq T}\|\varrho_N\|_{L^2}\big(1+\|\varrho_N\|_{L^\gamma}^\gamma+\|\sqrt{\varrho_N}\bfu_N\|_{L^2}^2\big)\bigg]^{\theta/2}\\
&\leq C|t-s|^{\theta/2}\bigg(1+\stred\sup_{0\leq t\leq T}\|\varrho_N\|_{L^2}^{\theta}+\stred\sup_{0\leq t\leq T}\|\varrho_N\|_{L^\gamma}^{\theta\gamma}+\stred\sup_{0\leq t\leq T}\|\sqrt{\varrho_N}\bfu_N\|_{L^2}^{2\theta}\bigg)\\
&\leq C|t-s|^{\theta/2}.
\end{align*}
%
By the Kolmogorov continuity criterion we conclude that for any $\sigma\in[0,1/2)$
$$\stred\,\bigg\|\int_0^t\varPhi^N(\varrho_N,\varrho_N\bu_N)\,\dif W\bigg\|_{C^\sigma([0,T];W^{-l,2}(\mt))}\leq C.$$
Besides, from \eqref{apv1} and \eqref{estrhou} we get
\begin{equation}\label{est:rhobfu2}
\varrho_N\bfu_N\otimes\bfu_N\in L^q(\Omega;L^2(0,T;L^\frac{6\beta}{4\beta+3}(\mt)))
\end{equation}
uniformly in $N$ and therefore
\begin{equation}\label{b1}
\{\diver(\varrho_N\bu_N\otimes\bu_N)\}\quad\text{is bounded in}\quad L^p(\Omega;L^2(0,T;W^{-1,\frac{6\beta}{4\beta+3}}(\mt))).
\end{equation}
As a consequence of \eqref{apv1} and \eqref{aprhobeta}
\begin{align}
\{\nu\Delta \bu_N+(\lambda+\nu)\nabla\diver\bfu_N\}\quad&\text{is bounded in}\quad L^p(\Omega;L^2(0,T;W^{-1,2}(\mt))),\label{b2}\\
\{a\nabla \varrho_N^\gamma+\delta\nabla\varrho_N^\beta\}\quad&\text{is bounded in}\quad L^p(\Omega;L^{\frac{\beta+1}{\beta}}(0,T;W^{-1,{\frac{\beta+1}{\beta}}}(\mt))).\label{b3}
\end{align}
Since $l>\frac{5}{2}$ there holds
$$W^{-1,\frac{6\beta}{4\beta+3}}(\mt)\hookrightarrow W^{-l,2}(\mt),\qquad W^{-1,\frac{\beta+1}{\beta}}(\mt)\hookrightarrow W^{-l,2}(\mt)$$
and thanks to uniform boundedness of $P_N$ on $W^{-l,2}(\mt)$ and $W^{-1,2}(\mt)$ it follows that
\begin{equation*}
\begin{split}
\big\{P_N\diver(\varrho_N\bu_N\otimes\bu_N)\big\}\quad&\text{is bounded in}\quad L^p(\Omega;L^{2}(0,T;W^{-l,{2}}(\mt))),\\
\big\{P_N\big(\nu\Delta \bu_N+(\lambda+\nu)\nabla\diver\bfu_N\big)\big\}\quad&\text{is bounded in}\quad L^p(\Omega;L^{2}(0,T;W^{-1,{2}}(\mt))),\\
\big\{P_N\big(a\nabla \varrho_N^\gamma+\delta\nabla\varrho_N^\beta\big)\big\}\quad&\text{is bounded in}\quad L^p(\Omega;L^{\frac{\beta+1}{\beta}}(0,T;W^{-l,{2}}(\mt))).
\end{split}
\end{equation*}
Finally, \eqref{eq:holderY1} follows for some $m>l$.

{\em Conclusion.} Collecting the above results we obtain that
$$\stred\|P_N(\varrho_N\bfu_N)\|_{C^\tau([0,T];W^{-m,2}(\mt)}\leq C$$
for some $\tau\in(0,1)$ and $m>\frac{5}{2}$. This implies the desired tightness by making use of \eqref{estrhou}, uniform boundedness of $P_N$ on $W^{-\frac{1}{2},2}(\mt)$, the embedding $L^\frac{2\beta}{\beta+1}(\mt)\hookrightarrow W^{-\frac{1}{2},2}(\mt)$ together with the compact embedding (see \cite[Corollary B.2]{on2})
$$L^\infty(0,T;L^\frac{2\beta}{\beta+1}(\mt))\cap C^{\tau}([0,T];W^{-m,2}(\mt))\overset{c}{\hookrightarrow} C_w([0,T];L^\frac{2\beta}{\beta+1}(\mt)).$$
\end{proof}
\end{proposition}

Since also the laws $\mu_{\varrho_0}$ and $\mu_W$, respectively, are tight as being Radon measures on the Polish spaces $\mathcal{X}_{\varrho_0}$ and $\mathcal{X}_W$, respectively, we can deduce tightness of the joint laws $\mu^N$.

\begin{corollary}\label{cor:tight}
The set $\{\mu^N;\,N\in\mn\}$ is tight on $\mathcal{X}$. 
\end{corollary}

The path space $\mathcal{X}$ is not a Polish space and so our compactness argument is based on the Jakubowski-Skorokhod representation theorem instead of the classical Skorokhod representation theorem, see \cite{jakubow}. To be more precise, passing to a weakly convergent subsequence $\mu^N$ (and denoting by $\mu$ the limit law) we infer the following result.

\begin{proposition}\label{prop:skorokhod}
There exists a probability space $(\tilde\Omega,\tilde\mf,\tilde\prst)$ with $\mathcal{X}$-valued Borel measurable random variables $(\tilde\varrho_N,\tilde\bu_N,\tilde\bfq_N,\tilde\varrho_{0,N},\tilde W_N)$, $N\in\mn$, and $(\tilde\varrho,\tilde\bu,\tilde\bfq,\tilde\varrho_0,\tilde W)$ such that
\begin{enumerate}
 \item the law of $(\tilde\varrho_N,\tilde\bu_N,\tilde\bfq_N,\tilde\varrho_{0,N},\tilde W_N)$ is given by $\mu^N$, $n\in\mn$,
\item the law of $(\tilde\varrho,\tilde\bu,\tilde \bfq,\tilde\varrho_0,\tilde W)$, denoted by $\mu$, is a Radon measure,
 \item $(\tilde\varrho_N,\tilde\bu_N,\tilde\bfq_N,\tilde{\varrho}_{0,N},\tilde W_N)$ converges $\,\tilde{\prst}$-almost surely to $(\tilde\varrho,\tilde{\bu},\tilde\bfq,\tilde\varrho_0,\tilde{W})$ in the topology of $\mathcal{X}$.
\end{enumerate}
\end{proposition}

We are immediately able to identify $(\tilde\varrho_{0,N},\tilde\bfq_N),\,N\in\mn,$ and $(\tilde\varrho_0,\tilde\bfq)$.

\begin{lemma}\label{lemma:identif}
There holds $\tilde\prst$-a.s. that
\begin{align*}
(\tilde\varrho_{0,N},\tilde\bfq_N)&=(\tilde\varrho_N(0),P_N(\tilde\varrho_N\tilde\bfu_N)),
\qquad(\tilde\varrho_0,\tilde\bfq)=(\tilde\varrho(0),\tilde\varrho\tilde\bfu).
\end{align*}

\begin{proof}
The first statement follows from the equality of joint laws of $(\varrho_N,\bfu_N,P_N(\varrho_N\bfu_N),\varrho_N(0))$ and $(\tilde\varrho_N,\tilde\bfu_N,\tilde\bfq_N,\tilde\varrho_{0,N})$. Identification of $\tilde\varrho_0$ follows from the a.s. convergence
$$\tilde\varrho_N\rightarrow\tilde\varrho\qquad\text{in}\qquad C_w([0,T];L^\beta(\mt))$$ and in order to identify the limit $\tilde\bfq$, note that
$$\tilde\varrho_N\tilde\bfu_N\rightharpoonup \tilde\varrho\tilde\bfu\quad\text{ in }\quad L^1(0,T;L^1(\mt))\quad\tilde\prst\text{-a.s.}$$
as a consequence of the convergence of $\tilde\varrho_N$ and $\tilde\bfu_N$ in $\mathcal{X}_\varrho$ and $\mathcal{X}_\bfu$, respectively. Clearly, this also identifies the limit of $P_N(\tilde\varrho_N\tilde\bfu_N)$ with $\tilde\varrho\tilde\bfu$.
\end{proof}
\end{lemma}

From \eqref{eq:aprioriN} and equality of joint laws we deduce
\begin{align}
&\tilde\stred\bigg[\sup_{0\leq t\leq T}\int_{\mt}\Big(\frac{1}{2}\tilde\varrho_N|\tilde\bu_N|^2+\frac{a}{\gamma-1}\tilde\varrho^\gamma_N+\frac{\delta}{\beta-1}\tilde\varrho^\beta_N\Big)\,\dif x\label{eq:aprioriNtilde}\\
&+\int_0^T\int_{\mt}\nu|\tilde\bu_N|^2+(\lambda+\nu)|\diver\tilde\bu_N|^2\,\dif x\,\dif s\bigg]\nonumber\\
& \leq C_p\bigg(1+\tilde\stred\bigg[\int_{\mt}\Big(\frac{1}{2}\frac{|\tilde\bfq_N(0)|^2}{\tilde\varrho_{0,N}}+\frac{a}{\gamma-1}\tilde\varrho_{0,N}^\gamma+\frac{\delta}{\beta-1}\tilde\varrho_{0,N}^\beta\Big)\,\dif x\bigg]^p \bigg)\nonumber\\
&=C_p\bigg(1+\int_{L^\beta_x\times L^\frac{2\beta}{\beta+1}_x}\bigg\|\frac{1}{2}\frac{|\bfq|^2}{\rho}+\frac{a}{\gamma-1}\rho^\gamma+\frac{\delta}{\beta-1}\rho^\beta\bigg\|_{L^1_x}^p\,\dif\Gamma(\rho,\bfq)\bigg)\leq C(p,\Gamma)\nonumber
\end{align}
uniformly in $N$, $\varepsilon$ and $\varrho$. 

Based on Proposition \ref{prop:skorokhod} and \eqref{eq:aprioriNtilde} we are going to achieve a series of further
convergences after taking not relabelled subsequences.

\begin{corollary}\label{cor:con}
The following convergence holds true $\tilde\prst$-a.s.
\begin{align}
\tilde{\varrho}_N  \tilde{\bfu}_N\otimes  \tilde{\bfu}_N&\rightharpoonup  \tilde{\varrho}  \tilde{\bfu}\otimes  \tilde{\bfu}\qquad\text{in}\qquad L^1(0,T;L^{1}( \mathbb T^3)).\label{conv:rhovv1}
\end{align}
\end{corollary}

\begin{proof}
From Proposition \ref{prop:skorokhod} and Lemma \ref{lemma:identif}, we gain
\begin{align}\label{eq:tight2''}
\begin{aligned}
\Big\|\sqrt{\tilde\varrho_N}\tilde\bfu_N\Big\|_{L^2_tL^2_x}^2&=\int_0^T\int_{\mt}P_N(\tilde{\varrho}_N  \tilde{\bfu}_N)\cdot\tilde\bfu_N\dxt\\
&\longrightarrow \int_0^T\int_{\mt}\tilde{\varrho}  \tilde{\bfu}\cdot\tilde\bfu\dxt=\Big\|\sqrt{\tilde\varrho}\tilde\bfu\,\Big\|_{L^2_tL^2_x}^2\quad\tilde\prst\text{-a.s.}
\end{aligned}
\end{align}
In the last step we used the compact embedding
$$L^{\frac{2\beta}{\beta+1}}(\mt)\overset{c}{\hookrightarrow}W^{-1,2}(\mt)$$
which implies together with Proposition \ref{prop:skorokhod} and Lemma \ref{lemma:identif} that $$P_N(\tilde{\varrho}_N  \tilde{\bfu}_N)\rightarrow \tilde\varrho\tilde\bfu\qquad\text{in}\qquad L^2(0,T;W^{-1,2}(\mt))\quad\tilde\p\text{-a.s.}$$
According to \eqref{eq:tight2''}, we infer that for almost every $\omega$, the sequence $\big(\sqrt{\tilde\varrho}_N\tilde\bfu_N(\omega)\big)$ is bounded in $L^2(0,T;L^2(\mt))$. Hence combining weak and strong convergence from Proposition \ref{prop:skorokhod} implies
\begin{align}\label{eq:tight1''}
\sqrt{\tilde\varrho_N}\tilde\bfu_N\rightharpoonup\sqrt{\tilde\varrho}\tilde\bfu\qquad\text{in}\qquad L^2(0,T;L^2(\mt))\quad\tilde\prst\text{-a.s.}
\end{align}
So \eqref{conv:rhovv1} follows by combining \eqref{eq:tight2''} and \eqref{eq:tight1''}.
\end{proof}

Let us now fix some notation that will be used in the sequel. We denote by $\bfr_t$ the operator of restriction to the interval $[0,t]$ acting on various path spaces. In particular, if $X$ stands for one of the path spaces $\mathcal{X}_\varrho,\,\mathcal{X}_{\bfu},\,\mathcal{X}_{\varrho\bfu}$ or $\mathcal{X}_{W}$ and $t\in[0,T]$, we define
\begin{align}\label{restr}
\bfr_t:X\rightarrow X|_{[0,t]},\quad f\mapsto f|_{[0,t]}.
\end{align}
Clearly, $ \bfr_t$ is a continuous mapping.
Let $(\tilde{\mf}_t)$ be the $\tilde{\prst}$-augmented canonical filtration of the process $(\tilde\varrho,\tilde{\bu},\tilde{W})$, respectively, that is
\begin{equation}\label{filtr}
\begin{split}
\tilde{\mf}_t&=\sigma\big(\sigma\big(\bfr_t\tilde\varrho, \,\bfr_t\tilde{\bu},\,\bfr_t\tilde{W}\big)\cup\big\{N\in\tilde{\mf};\;\tilde{\prst}(N)=0\big\}\big),\quad t\in[0,T].
\end{split}
\end{equation}

%
%
%
%
%
%
%
%
%
%
%

Finally, we have all in hand to conclude this Section by the following existence result.

\begin{proposition}\label{prop:martsol111}
$\big((\tilde{\Omega},\tilde{\mf},(\tilde{\mf}_t),\tilde{\prst}),\tilde\varrho,\tilde{\bu},\tilde{W}\big)$
is a weak martingale solution to \eqref{eq:approx} with the initial law $\Gamma$. 
\end{proposition}

We divide the proof into two parts. First, we prove that the equation \eqref{eq:approx1} holds true and establish strong convergence of $\nabla\tilde\varrho_N$ in $L^2(0,T;L^2(\mt))$ a.s. Second, we focus on the momentum equation \eqref{eq:approx2} and employ a new general method of constructing martingale solutions to SPDEs, that does not rely on any kind of martingale representation theorem and therefore holds independent interest especially in situations where these representation theorems are no longer available.

\begin{lemma}\label{eq:mass}
$\big(\tilde\varrho,\tilde{\bu}\big)$
is a weak solution to \eqref{eq:approx1}, i.e. for all $\psi\in C^\infty(\mt)$ and
and all $t\in[0,T]$ there holds $\tilde\prst$-a.s.
\begin{align*}
\big\langle\tilde\varrho(t),\psi\big\rangle&=\big\langle\tilde\varrho(0),\psi\big\rangle+\int_0^t\big\langle\tilde\varrho\tilde\bfu,\nabla\psi\big\rangle\,\dif s-\varepsilon\int_0^t\big\langle\nabla\tilde\varrho,\nabla\psi\big\rangle\,\dif s.
\end{align*}
 Furthermore, $\tilde\prst$-a.s.
$$\nabla\tilde\varrho_N\rightarrow\nabla\tilde\varrho\qquad\text{in}\qquad L^2(0,T;L^2(\mt)).$$

\begin{proof}
Let us we define, for all $t\in[0,T]$ and $\psi\in C^\infty(\mt)$, the functional
$$L(\rho,\bfq)_t(\psi)=\langle\rho(t),\psi\rangle-\langle\rho(0),\psi\rangle-\int_0^t\langle\bfq,\nabla\psi\rangle\,\dif s+\varepsilon\int_0^t\langle\nabla\rho,\nabla\psi\rangle\,\dif s.$$
For notational simplicity we neglect the $\psi$-dependence in the following. The mapping $(\rho,\bfq)\mapsto L(\rho,\bfq)_t$ is continuous on $\mathcal{X}_\varrho\times\mathcal{X}_{\varrho\bu}$. Hence the laws of $L(\varrho_N,\varrho_N\bfu_N)_t$ and $L(\tilde\varrho_N,\tilde\varrho_N\tilde\bfu_N)_t$ coincide and since $(\varrho_N,\varrho_N\bfu_N)$ solves \eqref{eq:galerkin11} we deduce that
\begin{align*}
\tilde\stred\big|L(\tilde\varrho_N,\tilde\varrho_N\tilde\bfu_N)_t\big|^2=\stred\big|L(\varrho_N,\varrho_N\bfu_N)_t\big|^2=0.
\end{align*}
Next, we pass to the limit on the left hand side by \eqref{aprho1}, \eqref{estrhou} and the Vitali convergence theorem which verifies \eqref{eq:approx1}.

In order to prove the strong convergence of $\nabla \tilde \varrho_N$, we recall that due to Proposition \ref{prop:skorokhod} there holds $\tilde\p$-a.s.
$$\nabla\tilde\varrho_N\rightharpoonup\nabla\tilde\varrho\qquad\text{in}\qquad L^2(0,T;L^2(\mt)).$$
Hence in order to prove strong convergence it is sufficient to establish convergence of the norms in $L^2(0,T;L^2(\mt)).$ Since both $(\tilde\varrho_N,\tilde\bfu_N)$ and $(\tilde\varrho,\tilde\bfu)$ solve \eqref{eq:approx1}, we shall test by $\tilde\varrho_N$ and $\tilde\varrho$, respectively, to obtain $\tilde\p$-a.s.
\begin{align*}
\|\tilde\varrho_N(t)\|_{L^2}^2+2\varepsilon\int_0^t\|\nabla\tilde\varrho_N\|_{L^2}^2\,\dif s&=\|\tilde\varrho_N(0)\|^2_{L^2}-\int_{0}^t\int_{\mathbb T^3}\Div\tilde\bfu_N\,|\tilde\varrho_N|^2\,\dif x\,\dif s,\\
\|\tilde\varrho(t)\|_{L^2}^2+2\varepsilon\int_0^t\|\nabla\tilde\varrho\|_{L^2}^2\,\dif s&=\|\tilde\varrho(0)\|^2_{L^2}-\int_{0}^t\int_{\mathbb T^3}\Div\tilde\bfu\,|\tilde\varrho|^2\,\dif x\,\dif s.
\end{align*}
Due to Proposition \ref{prop:skorokhod} we pass to the limit in the first term on the left hand side (after taking a subsequence) as well as in both terms on the right hand side. This implies $\tilde\p$-a.s.
$$\|\nabla\tilde\varrho_N\|_{L^2_{t,x}}\rightarrow\|\nabla\tilde\varrho\|_{L^2_{t,x}}$$
and completes the proof.
\end{proof}
\end{lemma}

\begin{lemma}\label{lem:strongq}
We have for all $1\leq q<\tfrac{2\beta}{\beta+1}$ where $\beta>\max\set{\tfrac{9}{2},\gamma}$
$$\tilde\varrho_N\tilde{\bfu}_N\rightarrow\tilde\varrho\tilde{\bfu}\qquad\text{in}\qquad L^q(\tilde\Omega\times Q).$$
\end{lemma}
\begin{proof}
Similar to the proof of \eqref{conv:rhovv1} we have $\tilde \p$-a.s.
\begin{align*}
\int_Q \varrho_N|\tilde\bfu_N|^2\dxt=\int_Q \varrho_N\tilde\bfu_N\cdot\tilde\bfu_N\dxt\longrightarrow \int_Q \varrho\tilde\bfu\cdot\tilde\bfu\dxt,
\end{align*}
so $\sqrt{\tilde\varrho_N}\tilde\bfu_N\rightarrow \sqrt{\tilde\varrho}\tilde\bfu$ in $L^2(Q)$. Combining this with Proposition \ref{prop:skorokhod} (and taking a subsequence) yields $\tilde \p$-a.s.
$$\tilde\varrho_N\tilde{\bfu}_N\rightarrow\tilde\varrho\tilde{\bfu}\qquad\text{in}\qquad L^1(Q).$$
The higher integrability from \eqref{estrhou} implies the claim.
\end{proof}

\begin{proposition}\label{prop:limit1}
The process $\tilde{W}$ is a $(\tilde{\mf}_t)$-cylindrical Wiener process, where the filtration $(\tilde\mf_t)$ was defined in \eqref{filtr}. Besides,
$$\big((\tilde{\Omega},\tilde{\mf},(\tilde{\mf}_t),\tilde{\prst}),\tilde\varrho,\tilde{\bu},\tilde{W}\big)$$
is a finite energy weak martingale solution to \eqref{eq:approx}.\footnote{This has to be understood in the sense of Definition \ref{def:sol} via an obvious modification by adding the artificial viscosity and artificial pressure terms.}

\begin{proof}


The first part of the claim follows immediately from the fact that $\tilde W_N$ has the same law as $W$. As a consequence, there exists a collection of mutually independent real-valued $(\tilde{\mf}_t)$-Wiener processes $(\tilde{\beta}^{N}_k)_{k\geq1}$ such that $\tilde{W}_N=\sum_{k\geq1}\tilde{\beta}^{N}_k e_k$ , i.e. there exists a collection of mutually independent real-valued $(\tilde{\mf}_t)$-Wiener processes $(\tilde{\beta}_k)_{k\geq1}$ such that $\tilde{W}=\sum_{k\geq1}\tilde{\beta}_k e_k$.

Let us now define for all $t\in[0,T]$ and $\bfvarphi\in \cup_{N\in\mn}X_N$ the functionals (we neglect the dependence on the fixed test-function in the notation)
\begin{equation*}
\begin{split}
M(\rho,\bfv,\bfq)_t&=\big\langle\bfq(t),\bfvarphi\big\rangle-\big\langle \bfq(0),\bfvarphi\big\rangle+\int_0^t\big\langle\bfq\otimes\bfv,\nabla\bfvarphi\big\rangle\,\dif r-\nu\int_0^t\big\langle\nabla\bfv,\nabla\bfvarphi\big\rangle\,\dif r\\
&\quad-(\lambda+\nu)\int_0^t\big\langle\diver\bfv,\diver\bfvarphi\big\rangle\,\dif r+a\int_0^t\big\langle\rho^\gamma,\diver\bfvarphi\big\rangle\,\dif r+\delta\int_0^t\big\langle\rho^\beta,\diver\bfvarphi\big\rangle\,\dif r\\
&\quad-\varepsilon\int_0^t\big\langle\nabla \bfv\nabla\rho,\bfvarphi\big\rangle\,\dif r,\\
\end{split}
\end{equation*}
\begin{align*}
N^N(\rho,\bfq)_t&=\sum_{k\geq1}\int_0^t\big\langle g_k^N(\rho, \bfq ),\bfvarphi\big\rangle^2\,\dif r,&\qquad N(\rho,\bfq)_t&=\sum_{k\geq1}\int_0^t\big\langle g_k(\rho, \bfq ),\bfvarphi\big\rangle^2\,\dif r\\
N_k^N(\rho,\bfq)_t&=\int_0^t\big\langle g_k^N(\rho,\bfq),\bfvarphi\big\rangle\,\dif r,&\qquad N_k(\rho,\bfq)_t&=\int_0^t\big\langle g_k(\rho,\bfq),\bfvarphi\big\rangle\,\dif r.
\end{align*}
Let $M(\rho,\bfv,\bfq)_{s,t}$ denote the increment $M(\rho,\bfv,\bfq)_{t}-M(\rho,\bfv,\bfq)_{s}$ and similarly for the other processes. 
Note that the proof will be complete once we show that the process $M(\tilde\varrho,\tilde\bfu,\tilde\varrho\tilde\bfu)$ is a $(\tilde\mf_t)$-martingale and its quadratic and cross variations satisfy, respectively,
\begin{equation}\label{mart}
\begin{split}
\langle\!\langle M(\tilde\varrho,\tilde\bfu,\tilde\varrho\tilde\bfu)\rangle\!\rangle&=N(\tilde\varrho,\tilde\varrho\tilde\bfu),\quad\qquad\langle\!\langle M(\tilde\varrho,\tilde\bfu,\tilde\varrho\tilde\bfu),\tilde \beta_k\rangle\!\rangle=N_k(\tilde\varrho,\tilde\varrho\tilde\bfu).
\end{split}
\end{equation}
Indeed, in that case we have due to bilinearity of the cross-variation
\begin{align*}
&\bigg\langle\!\!\!\bigg\langle M(\tilde\varrho,\tilde\bu,\tilde\varrho\tilde\bfu)-\int_0^\tec \big\langle\varPhi(\tilde\varrho,\tilde\varrho\tilde\bu)\,\dif \tilde W,\bfvarphi\big\rangle\bigg\rangle\!\!\!\bigg\rangle\\
&=\langle\!\langle M(\tilde\varrho,\tilde\bfu,\tilde\varrho\tilde\bfu)\rangle\!\rangle-2\sum_{k\geq 1}\int_0^\tec\langle g_k(\tilde\varrho,\tilde\varrho\tilde\bfu),\bfvarphi\rangle\,\dif \langle\!\langle M(\tilde\varrho,\tilde\bfu,\tilde\varrho\tilde\bfu),\tilde\beta_k\rangle\!\rangle+\bigg\langle\!\!\!\bigg\langle\int_0^\tec \big\langle\varPhi(\tilde\varrho,\tilde\varrho\tilde\bu)\,\dif \tilde W,\bfvarphi\big\rangle\bigg\rangle\!\!\!\bigg\rangle\\
&=0
\end{align*}
and \eqref{eq:approx2} is satisfied.

Let us verify \eqref{mart}. To this end, we claim that with the above uniform estimates in hand, the mappings
$$(\rho,\bfv,\bfq)\mapsto M(\rho,\bfv,\bfq)_t,\quad\,(\rho,\bfv,\bfq)\mapsto N^N(\rho,\bfq)_t,\,\quad (\rho,\bfv,\bfq)\mapsto N^N_k(\rho,\bfq)_t$$
and
$$(\rho,\bfv,\bfq)\mapsto N(\rho,\bfq)_t,\,\quad (\rho,\bfv,\bfq)\mapsto N_k(\rho,\bfq)_t$$
are well-defined and measurable on a subspace of $\mathcal{X}_\varrho\times\mathcal{X}_\bu\times\mathcal{X}_{\varrho\bfu}$ where the joint law of $(\tilde\varrho,\tilde\bfu,\tilde\bfq)$ is supported, i.e. where all the uniform estimates hold true and $(\varrho(t))_{\mt}$ is bounded in $t$ and $\omega$.
Indeed, in the case of $N^N(\rho,\bfq)_t$ we have similarly to \eqref{a1} - \eqref{a3}
\begin{align}\label{34}
\sum_{k\geq1}\int_0^t\big\langle g_k^N(\rho, \bfq ),\bfvarphi\big\rangle^2\,\dif s&\leq C\int_0^t\|\rho\|_{L^2}\int_{\mt}\Big(1+\rho^\gamma+\frac{|\bfq|^2}{\rho}\Big)\,\dif x\,\dif s,
\end{align}
for $N(\rho,\bfq)_t$ by \eqref{growth1} similarly to \eqref{stochest}
\begin{align*}
\sum_{k\geq 1}\int_0^t\big\langle g_k(\rho,\bfq),\bfvarphi\big\rangle^2\,\dif s&\leq C\sum_{k\geq 1}\int_0^t\|\, g_k(\rho,\bfq)\|_{L^1}^2\,\dif s\\
&\leq C\int_0^t\int_{\mt}\Big(\rho +\rho^{\gamma}+\frac{|\bfq|^2}{\rho}\Big)\,\dif x\,\dif s
\end{align*}
Both are finite due to \eqref{aprho1} and \eqref{est:rhobfu2}.
$M(\rho,\bfv,\bfq)$, $N^N_k(\rho,\bfv)_t\,$ and $N_k(\rho,\bfv)_t\,$ can be handled similarly and therefore, the following random variables have the same laws
\begin{align*}
M(\varrho_N,\bu_N,\varrho_N\bfu_N)&\distr M(\tilde\varrho_N,\tilde\bu_N,\tilde\varrho_N\tilde\bfu_N),\\
N^N(\varrho_N,\varrho_N\bfu_N)&\distr N^N(\tilde\varrho_N,\tilde\varrho_N\tilde\bfu_N),\\
N^N_k(\varrho_N,\varrho_N\bfu_N)&\distr N^N_k(\tilde\varrho_N,\tilde\varrho_N\tilde\bfu_N).
\end{align*}

Let us now fix times $s,t\in[0,T]$ such that $s<t$ and let
$$h:\mathcal{X}_\varrho|_{[0,s]}\times\mathcal{X}_\bfu|_{[0,s]}\times\mathcal{X}_W|_{[0,s]}\rightarrow [0,1]$$
be a continuous function.
Since
$$M(\varrho_N,\bu_N,\varrho_N\bfu_N)_t=\int_0^t\big\langle\varPhi^N(\varrho_N,\varrho_N\bu_N)\,\dif W,\bfvarphi\big\rangle=\sum_{k\geq1}\int_0^t\big\langle g_k^N(\varrho_N,\varrho_N\bu_N),\bfvarphi\big\rangle\,\dif\beta_k$$
is a square integrable $(\mf_t)$-martingale, we infer that
$$\big[M(\varrho_N,\bu_N,\varrho_N\bfu_N)\big]^2-N^N(\varrho_N,\varrho_N\bfu_N),\quad M(\varrho_N,\bu_N,\varrho_N\bfu_N)\beta_k-N_k^N(\varrho_N,\varrho_N\bfu_N)$$
are $(\mf_t)$-martingales.
Besides, it follows from the equality of laws that (recall that the restriction operator $\bfr_s$ was defined in \eqref{restr})
\begin{equation}\label{exp11}
\begin{split}
&\tilde{\stred}\,h\big(\bfr_s\tilde\varrho_N, \bfr_s\tilde{\bu}_N,\bfr_s\tilde{W}_N\big)\big[M(\tilde\varrho_N,\tilde\bu_N,\tilde\varrho_N\tilde\bfu_N)_{s,t}\big]\\
&=\stred\,h\big(\bfr_s\varrho_N, \bfr_s\bu_N, \bfr_s W_N\big)\big[M(\varrho_N,\bu_N,\varrho_N\bfu_N)_{s,t}\big]=0,
\end{split}
\end{equation}
\begin{equation}\label{exp21}
\begin{split}
&\tilde{\stred}\,h\big(\bfr_s\tilde\varrho_N, \bfr_s\tilde{\bu}_N,\bfr_s\tilde{W}_N\big)\bigg[[M(\tilde\varrho_N,\tilde\bu_N,\tilde\varrho_N\tilde\bfu_N)^2]_{s,t}-N^N(\tilde\varrho_N,\tilde\varrho_N\tilde\bfu_N)_{s,t}\bigg]\\
&=\stred\,h\big(\bfr_s\varrho_N, \bfr_s\bu_N, \bfr_s W_N\big)\bigg[[M(\varrho_N,\bu_N,\varrho_N\bfu_N)^2]_{s,t}-N^N(\varrho_N,\varrho_N\bfu_N)_{s,t}\bigg]=0,
\end{split}
\end{equation}
\begin{equation}\label{exp31}
\begin{split}
&\tilde{\stred}\,h\big(\bfr_s\tilde\varrho_N, \bfr_s\tilde{\bu}_N,\bfr_s\tilde{W}_N\big)\bigg[[M(\tilde\varrho_N,\tilde\bu_N,\tilde\varrho_N\tilde\bfu_N)\tilde{\beta}_k^N]_{s,t}-N_k^N(\tilde\varrho_N,\tilde\varrho_N\tilde\bfu_N)_{s,t}\bigg]\\
&=\stred\,h\big(\bfr_s\varrho_N, \bfr_s\bu_N, \bfr_s W_N\big)\bigg[[M(\varrho_N,\bu_N,\varrho_N\bfu_N)\beta_k]_{s,t}-N_k^N(\varrho_N,\varrho_N\bfu_N)_{s,t}\bigg]=0.
\end{split}
\end{equation}

As the next step, we employ the assumptions \eqref{growth1} and \eqref{growth2} and the estimates \eqref{apv1}, \eqref{aprho1}, \eqref{estrhou}, \eqref{estnablarho}, \eqref{est:rhobfu2} together with Proposition \ref{prop:skorokhod}, Corollary \ref{cor:con}, Lemma \ref{eq:mass}, Lemma \ref{lem:strongq} and the Vitali convergence theorem, pass to the limit in \eqref{exp11}, \eqref{exp21} and \eqref{exp31} and establish the following identities that justify \eqref{mart}
\begin{align*}
\tilde{\stred}\,h\big(\bfr_s\tilde\varrho, \bfr_s\tilde{\bu},\bfr_s\tilde{W}\big)\big[M(\tilde\varrho,\tilde\bu,\tilde\varrho\tilde\bfu)_{s,t}\big]=0,\\
\tilde{\stred}\,h\big(\bfr_s\tilde\varrho, \bfr_s\tilde{\bu},\bfr_s\tilde{W}\big)\bigg[[M(\tilde\varrho,\tilde\bu,\tilde\varrho\tilde\bfu)^2]_{s,t}-N(\tilde\varrho,\tilde\varrho\tilde\bfu)_{s,t}\bigg]=0,\\
\tilde{\stred}\,h\big(\bfr_s\tilde\varrho, \bfr_s\tilde{\bu},\bfr_s\tilde{W}\big)\bigg[[M(\tilde\varrho,\tilde\bu,\tilde\varrho\tilde\bfu)\tilde{\beta}_k]_{s,t}-N_k(\tilde\varrho,\tilde\varrho\tilde\bfu)_{s,t}\bigg]=0.
\end{align*}
Let us comment on the passage to the limit in the terms coming from the stochastic integral, i.e. $N^N(\tilde\varrho_N,\tilde\varrho_N\tilde\bu_N)$ and $N_k^N(\tilde\varrho_N,\tilde\varrho_N\tilde\bu_N)$. The convergence in \eqref{exp31} being easier, let us only focus on \eqref{exp21} in detail.
As a first step we aim to show for all $ k\in\mn$ that
\begin{align}\label{eq:gkN}
\big\langle g_k^N(\tilde\varrho_N,\tilde\varrho_N\tilde\bfu_N),\bfvarphi\big\rangle\rightarrow\big\langle g_k(\tilde\varrho,\tilde\varrho\tilde\bfu),\bfvarphi\big\rangle\qquad\tilde\prst\otimes\mathcal{L}\text{-a.e.}
\end{align}
We first remark that by definition and the symmetry of $\mathcal M_N[\varrho]$ we have
\begin{align*}
\big\langle g_k^N(\tilde\varrho_N,\tilde\varrho_N\tilde\bfu_N),\bfvarphi\big\rangle=\bigg\langle\mathcal M_N^{\frac{1}{2}}[\tilde\varrho_N]P_N \Big(\frac{g_k(\tilde\varrho_N,\tilde\varrho_N\tilde\bfu_N)}{\sqrt{\tilde\varrho_N}}\Big),\bfvarphi\bigg\rangle=\bigg\langle\frac{g_k(\tilde\varrho_N,\tilde\varrho_N\tilde\bfu_N)}{\sqrt{\tilde\varrho_N}},\mathcal M_N^{\frac{1}{2}}[\tilde\varrho_N]\bfvarphi\bigg\rangle.
\end{align*}
As a consequence of the strong convergences in Proposition \ref{prop:skorokhod} and Lemma \ref{lem:strongq} we have (at least after taking a subsequence)
\begin{align}\label{eq:gkN1}
\frac{g_k(\tilde\varrho_N,\tilde\varrho_N\tilde\bfu_N)}{\sqrt{\tilde\varrho_N}}\longrightarrow \frac{g_k(\tilde\varrho,\tilde\varrho\tilde\bfu)}{\sqrt{\tilde\varrho}}\quad\text{in}\quad L^2(\mt)\quad\tilde\p\otimes\mathcal{L}\text{-a.e.}
\end{align}
where we also used \eqref{growth1}, \eqref{growth2} and the a priori estimates.
Moreover, for every $\bfv\in W^{l,2}(\mt)$, using again strong convergence of $\tilde\varrho_N$, the embedding $W^{l,2}(\mt)\hookrightarrow L^\infty(\mt)$ and continuity of $P_N$ on $L^2(\mt)$ and $W^{l,2}(\mt)$, we have
\begin{align*}
\big\|\mathcal{M}_N[\tilde\varrho_N]\bfv-\tilde\varrho\bfv\big\|_{L^2}&\leq\big\|P_N\big((\tilde\varrho_N-\tilde\varrho)\,P_N \bfv\big)\big\|_{L^2}+\big\|P_N(\tilde\varrho \,P_N \bfv)-P_N(\tilde\varrho\bfv)\big\|_{L^2}\\
&\quad+\big\|P_N(\tilde\varrho\bfv)-\tilde\varrho\bfv\big\|_{L^2}\\
&\leq c\|\tilde\varrho_N-\tilde\varrho\|_{L^2}\|\bfv\|_{W^{l,2}}+c\|\tilde\varrho\|_{L^2}\big\|P_N\bfv-\bfv\big\|_{W^{l,2}}+\big\|P_N(\tilde\varrho\bfv)-\tilde\varrho\bfv\big\|_{L^2}\\
&\longrightarrow0\qquad\tilde\p\otimes\mathcal{L}\text{-a.e}.
\end{align*}
Hence $\mathcal{M}_N[\tilde\varrho_N]\,\cdot\rightarrow\tilde\varrho\,\cdot$ pointwise as an operator from $W^{l,2}(\mt)\rightarrow L^2(\mt)$. From
a formal point of view it should follow that
$$\mathcal M_N^{ \frac{1}{2}}[\tilde\varrho_N]\,\cdot\rightarrow \sqrt{\tilde\varrho}\,\cdot\qquad\tilde\p\otimes\mathcal{L}\text{-a.e}.$$
in the same sense (recalling that the square root of a positive semidefinite operator is unique). In order to make this argument rigorous we extend $\mathcal M_N[\rho]$ to an operator $W^{-l,2}(\mt)\rightarrow W^{-l,2}(\mt)$ (thus we stay in the same space). So we set
\begin{align*}
&\overline{\mathcal M}_N[\tilde\varrho_N]:W^{-l,2}(\mt)\rightarrow W^{-l,2}(\mt),\\
&\overline{\mathcal M}_N[\tilde\varrho_N]\bfPsi(\bfv)=\langle\mathcal M_N[\tilde\varrho_N]\bfPhi,\bfv\rangle_2,\quad 
\bfPsi(\bfw)=\langle\bfPhi,\bfw\rangle_{l,2},
\end{align*}
where $\bfw,\bfv,\bfPhi\in W^{l,2}(\mt)$. Now we have
\begin{align*}
\overline{\mathcal M}_N[\tilde\varrho_N]\bfPsi\rightarrow \langle\tilde\varrho\,\bfPhi,\cdot\rangle_2\qquad\tilde\p\otimes\mathcal{L}\text{-a.e}.
\end{align*}
pointwise as an operator from $W^{-l,2}(\mt)\rightarrow W^{-l,2}(\mt)$ and hence \begin{align*}
\overline{\mathcal M}^{\frac{1}{2}}_N[\tilde\varrho_N]\bfPsi\rightarrow \big\langle\sqrt{\tilde\varrho}\,\bfPhi,\cdot\big\rangle_2\qquad\tilde\p\otimes\mathcal{L}\text{-a.e}.
\end{align*}
The latter can be easily justified by using the series expansion
\begin{align*}
A^{\frac{1}{2}}=\sum_{k=0}^\infty c_k(\mathrm{I}-A)^k,\quad c_k\in\R,\quad\sum_{k=0}^\infty|c_k|<\infty,
\end{align*}
which holds for every symmetric positive semidefinite operator $A$ on some real Hilbert space $\mathscr H$ with $\sup_{\|z\|_{\mathscr H}\leq1}\langle Az,z\rangle_{\mathscr H}\leq 1$. Finally, we gain
\begin{align}\label{eq:gkN2}
\mathcal M^{\frac{1}{2}}_N[\tilde\varrho_N]\,\cdot\rightarrow \sqrt{\tilde\varrho}\,\cdot\qquad\tilde\p\otimes\mathcal L\text{-a.e.}
\end{align}
pointwise as an operator from $W^{l,2}(\mt)\rightarrow L^2(\mt)$. Plugging \eqref{eq:gkN1} and \eqref{eq:gkN2} together we have shown \eqref{eq:gkN}.

The convergence
\begin{align*}
\sum_{k\geq1}\big\langle g^N_k(\tilde\varrho_N,\tilde\varrho_N\tilde\bfu_N),\bfvarphi\big\rangle^2\rightarrow \sum_{k\geq 1}\big\langle g_k(\tilde\varrho,\tilde\varrho\tilde\bfu),\bfvarphi\big\rangle^2\qquad\tilde\prst\otimes\mathcal{L}\text{-a.e}.
\end{align*}
follows once we show that
\begin{equation}\label{convL2}
\big\langle\varPhi^N(\tilde\varrho_N,\tilde\varrho_N\tilde\bfu_N)\,\cdot\,,\bfvarphi\big\rangle\rightarrow\big\langle\varPhi(\tilde\varrho,\tilde\varrho\tilde\bfu)\,\cdot\,,\bfvarphi\big\rangle\qquad\text{in}\qquad L_2(\mathfrak{U};\mr)\qquad\tilde\prst\otimes\mathcal{L}\text{-a.e.}
\end{equation}
To this end, we estimate
\begin{align*}
I&=\big\|\big\langle\varPhi^N(\tilde\varrho_N,\tilde\varrho_N\tilde\bfu_N)\,\cdot\,,\bfvarphi\big\rangle-\big\langle\varPhi(\tilde\varrho,\tilde\varrho\tilde\bfu)\,\cdot\,,\bfvarphi\big\rangle\big\|_{L_2(\mathfrak{U};\mr)}\\
&\leq \big\|\big\langle\varPhi^N(\tilde\varrho_N,\tilde\varrho_N\tilde\bfu_N)\,\cdot\,,\bfvarphi\big\rangle-\big\langle\varPhi(\tilde\varrho_N,\tilde\varrho_N\tilde\bfu_N)\,\cdot\,,\bfvarphi\big\rangle\big\|_{L_2(\mathfrak{U};\mr)}\\
&\qquad+\big\|\big\langle\varPhi(\tilde\varrho_N,\tilde\varrho_N\tilde\bfu_N)\,\cdot\,,\bfvarphi\big\rangle-\big\langle\varPhi(\tilde\varrho,\tilde\varrho\tilde\bfu)\,\cdot\,,\bfvarphi\big\rangle\big\|_{L_2(\mathfrak{U};\mr)}=I_1+I_2.
\end{align*}
The first term can be estimated as follows
\begin{align*}
I_1&\leq \bigg(\sum_k\Big\langle\frac{g_k(\tilde\varrho_N,\tilde\varrho_N\tilde\bfu_N)}{\sqrt{\tilde\varrho_N}},P_N\mathcal M_N^{\frac{1}{2}}[\tilde\varrho_N]\bfvarphi-\sqrt{\tilde\varrho_N}\bfvarphi\Big\rangle^2\bigg)^{\frac{1}{2}}\\
&\leq C\bigg(\,\sum_{k\geq1}\int_{\mt}\tilde\varrho_N^{-1}| g_{k}(\cdot,\tilde \varrho_N,\tilde\varrho_N\tilde\bfu_N)|^2\dx\bigg)^{\frac{1}{2}}\Big\|P_N\mathcal M_N^{\frac{1}{2}}[\tilde\varrho_N]\bfvarphi-\sqrt{\tilde\varrho_N}\bfvarphi\Big\|_{L_x^2}\\
&\leq C\bigg(\int_{\mt}\big(1+\tilde \varrho_N^\gamma+\tilde\varrho_N|\tilde\bfu_N|^2\big)\dx\bigg)^{\frac{1}{2}}\bigg(\Big\|P_N\mathcal M_N^{\frac{1}{2}}[\tilde\varrho_N]\bfvarphi-\sqrt{\tilde\varrho}\bfvarphi\Big\|_{L_x^2}+\Big\|(\sqrt{\tilde\varrho}-\sqrt{\tilde\varrho_N})\bfvarphi\Big\|_{L_x^2}\bigg).
\end{align*}
Therefore, using \eqref{aprhobeta}, Proposition \ref{prop:skorokhod} and \eqref{eq:gkN2} (and taking a subsequence) we deduce that 
\begin{align*}
\tilde\E \int_0^TI_1\dt\longrightarrow 0,\quad N\rightarrow\infty.
\end{align*}
Hence we gain $I_1\rightarrow 0$ for a.e. $(\omega, t)$ after taking a subsequence.
Moreover, we have using the Minkowski integral inequality, the mean value theorem, \eqref{growth1} and \eqref{growth2}
\begin{align*}
I_2&\leq C\bigg(\,\sum_{k\geq1}\big\| g_{k}(\tilde \varrho_N,\tilde\varrho_N\tilde\bfu_N)-g_{k}(\tilde \varrho,\tilde \varrho \tilde\bfu)\big\|_{L^1_x}^2\bigg)^{\frac{1}{2}}\\
&\leq C\int_{\mt}\!\bigg(\sum_{k\geq1}\big|g_{k}(\tilde \varrho_N,\tilde\varrho_N\tilde\bfu_N)-g_{k}(\tilde \varrho,\tilde \varrho \tilde\bfu)\big|^2\bigg)^{\frac{1}{2}}\dif x\\
&\leq C\int_{\mt}\Big(1+\tilde\varrho_N^{\frac{\gamma-1}{2}}+\tilde\varrho^{\frac{\gamma-1}{2}}\Big)\Big(|\tilde\varrho_N-\tilde\varrho|+|\tilde\varrho_N\tilde\bfu_N-\tilde\varrho\tilde\bfu|\Big)\,\dif x\\
&\leq C\bigg[\int_{\mt}\Big(1+\tilde\varrho_N^{\frac{\gamma-1}{2}}+\tilde\varrho^{\frac{\gamma-1}{2}}\Big)^{p}\,\dif x\bigg]^{\frac{1}{{p}}}\bigg[\int_{\mt}|\tilde\varrho_N-\tilde\varrho|^{q}+|\tilde\varrho_N\tilde\bfu_N-\tilde\varrho\tilde\bfu|^{q}\,\dif x\bigg]^{\frac{1}{q}}
\end{align*}
where the conjugate exponents $p,q\in(1,\infty)$ are chosen in such a way that
$$p\frac{\gamma-1}{2}<\gamma+1\qquad\text{and}\qquad q<\frac{2\gamma}{\gamma+1}.$$
Therefore, using \eqref{aprhobeta}, Proposition \ref{prop:skorokhod} and Lemma \ref{lem:strongq} (and taking a subsequence) we deduce 
\begin{align*}
\tilde\E \int_0^TI_2\dt\longrightarrow 0,\quad N\rightarrow\infty.
\end{align*}
and so for a subsequence $I_2\rightarrow 0$ for a.e. $(\omega, t)$ and \eqref{convL2} follows. Besides, since similarly to \eqref{a1} - \eqref{a3}, for all $p\geq 2$,
\begin{align*}
\tilde\stred\int_s^t&\big\|\big\langle\varPhi^N(\tilde \varrho_N,\tilde\varrho_N\tilde\bfu_N)\,\cdot,\bfvarphi\big\rangle\big\|_{L_2(\mathfrak{U};\mr)}^p\,\dif r\\
&\leq C\,\tilde\stred\int_s^t\|\tilde\varrho_N\|_{L^2}^{\frac{p}{2}}\Big(1+\|\tilde\varrho_N\|_{L^\gamma}^\gamma+\|\sqrt{\tilde\varrho_N}\tilde\bfu_N\|^2_{L^2}\Big)^{\frac{p}{2}}\dif r\\
&\leq C\bigg(1+\tilde\stred\sup_{0\leq t\leq T}\|\tilde\varrho_N\|_{L^\gamma}^{\gamma p}+\tilde\stred\sup_{0\leq t\leq T}\|\sqrt{\tilde\varrho_N}\tilde\bfu_N\|_{L^{2}}^{2p}\bigg)\leq C
\end{align*}
due to \eqref{aprhov1}, \eqref{aprho1}, we obtain the convergence in \eqref{exp21}.\\
By standard regularity theory for parabolic equations it can be shown that $\tilde\varrho$ satisfies the continuity equation a.e. and hence is also a solution in the renormalized sense. The energy inequality is a consequence of \eqref{eq:aprioriNtilde} together with the lower semi-continuity of the left-hand-side.
\end{proof}
\end{proposition}

\section{The vanishing viscosity limit}
\label{sec:vanishingviscosity}

The aim of this Section is to study the limit $\varepsilon\rightarrow0$ in the approximate system \eqref{eq:approx} and to establish existence of a weak martingale solution with the initial law $\Gamma$ to
\begin{subequations}\label{eq:approximation2}
 \begin{align}
  \dif \varrho+\diver(\varrho\bu)\dif t&=0,\label{eq:approximation21}\\
  \dif(\varrho\bu)+\big[\diver(\varrho\bu\otimes\bu)-\nu\Delta\bu-(\lambda+\nu)\nabla\diver\bfu
+a\nabla \varrho^\gamma+\delta\nabla\varrho^\beta\big]\dif t&=\varPhi(\varrho,\varrho\bu) \,\dif W,\label{eq:approximation22}
 \end{align}
\end{subequations}
where $\delta>0$ and $\beta> \max\{\frac{9}{2},\gamma\}$. To this end, we recall that it was proved in Section \ref{sec:viscous} that for every $\varepsilon\in (0,1)$ there exists
$$\big((\tilde\Omega^\varepsilon,\tilde\mf^\varepsilon,(\tilde\mf^\varepsilon_t),\tilde\prst^\varepsilon),\tilde\varrho_\varepsilon,\tilde\bfu_\varepsilon,\tilde W_\varepsilon\big)$$
which is a weak martingale solution to \eqref{eq:approx}. It was shown in \cite{jakubow} that it is enough to consider only one probability space, namely,
$$(\tilde\Omega^\varepsilon,\tilde\mf^\varepsilon,\tilde\prst^\varepsilon)=\big([0,1],\mathcal{B}([0,1]),\mathcal{L}\big)\qquad\forall \varepsilon\in (0,1)$$
where $\mathcal{L}$ denotes the Lebesgue measure on $[0,1]$.
Moreover, we can assume without loss of generality that there exists one common Wiener process $W$ for all $\varepsilon$. Indeed, one could perform the compactness argument of the previous section for all the parameters from any chosen subsequence $\varepsilon_n$ at once by redefining
$$\mathcal{X}=\Big(\prod_{n\in\mn}\mathcal{X}_\varrho\times\mathcal{X}_\bfu\times\mathcal{X}_{\varrho\bfu}\Big)\times  \mathcal X_{\varrho_0}\times\mathcal{X}_W$$
and proving tightness for the following set of $\mathcal{X}$-valued random variables
$$\big\{\big((\varrho_{N,\varepsilon_1},\bfu_{N,\varepsilon_1},\varrho_{N,\varepsilon_1}\bfu_{N,\varepsilon_1}),(\varrho_{N,\varepsilon_2},\bfu_{N,\varepsilon_2},\varrho_{N,\varepsilon_2}\bfu_{N,\varepsilon_2}),\dots,\varrho_0,W\big);\,N\in\mn\big\}.$$
In order to further simplify the notation we also omit the tildas and denote the weak martingale solution found in Section \ref{sec:viscous} by
$$\big((\Omega,\mf,(\mf^\varepsilon_t),\prst),\varrho_\varepsilon,\bfu_\varepsilon,W\big).$$

The functions $\bfu_\varepsilon$ and $\varrho_\varepsilon$ satisfy the energy inequality, i.e. for any $p<\infty$ we have
\begin{align}
&\stred\bigg[\sup_{0\leq t\leq T}\int_{\mt}\Big(\frac{1}{2}\varrho_\varepsilon|\bu_\varepsilon|^2+\frac{a}{\gamma-1}\varrho^\gamma_\varepsilon+\frac{\delta}{\beta-1}\varrho^\beta_\varepsilon\Big)\,\dif x\label{eq:apriorivarepsilon}\\
&+\int_0^T\int_{\mt}\nu|\bu_\varepsilon|^2+(\lambda+\nu)|\diver\bu_\varepsilon|^2\,\dif x\,\dif s\bigg]^p\nonumber\\
&\leq \,C_p\bigg(1+\int_{L^\beta_x\times L^\frac{2\beta}{\beta+1}_x}\bigg\|\frac{1}{2}\frac{|\bfq|^2}{\rho}+\frac{a}{\gamma-1}\rho^\gamma+\frac{\delta}{\beta-1}\rho^\beta\bigg\|_{L^1_x}^p\,\dif\Gamma(\rho,\bfq)\bigg)\leq C(p,\Gamma)\nonumber
\end{align}
This means we have following uniform bounds
\begin{align}
 \bfu_\varepsilon&\in L^{p}(\Omega;L^2(0,T;W^{1,2}( \mathbb T^3))),\label{apv}\\
\sqrt{ \varrho_\varepsilon} \bfu_\varepsilon&\in L^{p}(\Omega;L^\infty(0,T;L^2( \mathbb T^3))),\label{aprhov}\\
 \varrho_\varepsilon&\in L^{p}(\Omega;L^\infty(0,T;L^\beta( \mathbb T^3))),\label{aprho}\\
  \varrho_\varepsilon\bfu_\varepsilon&\in L^{p}(\Omega;L^\infty(0,T;L^\frac{2\beta}{\beta+1}( \mathbb T^3))),\label{estrhou2}\\
\varrho_\varepsilon\bfu_\varepsilon\otimes\bfu_\varepsilon&\in L^p(\Omega;L^2(0,T;L^\frac{6\beta}{4\beta+3}(\mt))).\label{est:rhobfu22}
\end{align}
Besides, testing \eqref{eq:approx1} by $\varrho_\varepsilon$ gives
\begin{align}\label{est:nablarho}
\sqrt{\varepsilon}\nabla\varrho_\varepsilon\in L^p(\Omega;L^2(0,T;L^2(\mt)))
\end{align}
and consequently
\begin{align}
\varepsilon\nabla\varrho_\varepsilon\rightarrow 0\quad\text{in}\quad L^2(\Omega\times Q),\label{eq:van1}\\
\varepsilon\nabla\bfu_\varepsilon\nabla\varrho_\varepsilon\rightarrow 0\quad\text{in}\quad L^1(\Omega\times Q).\label{eq:van2}
\end{align}

As the next step, we improve the space integrability of the density.

\begin{proposition}\label{prop:higher}
There holds
\begin{equation}\label{eq:gamma+1}
\stred\int_0^T\int_{\mt}\big(a\varrho_\varepsilon^{\gamma+1}+\delta\varrho_\varepsilon^{\beta+1}\big)\,\dif x\,\dif t\leq C.
\end{equation}

\begin{proof}
In the deterministic case, this is achieved by testing \eqref{eq:approx2} with $$\Delta^{-1}\nabla\varrho_\varepsilon=\nabla\Delta^{-1}(\varrho_\varepsilon-(\varrho_\varepsilon)_{\mt}).$$ Here $\Delta^{-1}$ is the solution operator to the Laplace equation on the torus (the mean value of right hand side needs to vanish) which commutes with derivatives.
In the stochastic setting, we apply the It\^{o} formula to the functional $f(\rho, \bfq)=\int_ {\mathbb T^3} \bfq\cdot\Delta^{-1}\nabla \rho\dx$. Note that since $f$ is linear in $\bfq= \varrho_\varepsilon \bfu_\varepsilon$
and the quadratic variation of $\varrho_\varepsilon$ is zero, no correction terms appear in our calculation. We gain 
\begin{align}\label{eq:test}
\begin{aligned}
\int_ {\mathbb T^3} & \varrho_\varepsilon \bfu_\varepsilon\cdot \Delta^{-1}\nabla \varrho_\varepsilon\dx=\int_ {\mathbb T^3}  \varrho(0) \bfu(0)\cdot \Delta^{-1}\nabla \varrho(0)\dx\\
&\quad-\nu\int_0^t\int_ {\mathbb T^3} \nabla \bfu_\varepsilon:\nabla\Delta^{-1}\nabla \varrho_\varepsilon\dxs-(\lambda+\nu)\int_0^t\int_ {\mathbb T^3} \diver\bfu_\varepsilon\,\varrho_\varepsilon\dxs\\
&\quad+\int_0^t\int_ {\mathbb T^3}  \varrho_\varepsilon \bfu_\varepsilon\otimes \bfu_\varepsilon:\nabla\Delta^{-1} \nabla\varrho_\varepsilon\dxs\\
&\quad-\varepsilon\int_0^t\int_{\mathbb T^3}\nabla\bfu_\varepsilon\nabla\varrho_\varepsilon\cdot\Delta^{-1}\nabla\varrho_\varepsilon\dxs\\
&\quad+\int_0^t\int_ {\mathbb T^3} \big(a\varrho_\varepsilon^{\gamma+1}+\delta\varrho_\varepsilon^{\beta+1}\big)\dxs-\int_0^t(\varrho_\varepsilon)_{\mt}\int_ {\mathbb T^3} \big(a\varrho_\varepsilon^{\gamma}+\delta\varrho_\varepsilon^{\beta}\big)\dxs\\
&\quad+\sum_{k\geq 1}\int_0^t \int_ {\mathbb T^3}\Delta^{-1}\nabla\varrho_\varepsilon\cdot g_k(\varrho_\varepsilon, \varrho_\varepsilon \bfu_\varepsilon)\dx\,\dif \beta_k(\sigma)\\
&\quad+\varepsilon\int_0^t\int_ {\mathbb T^3} \varrho_\varepsilon\bfu_\varepsilon\nabla\varrho_\varepsilon\dxs-\int_0^t\int_{\mt}\varrho_\varepsilon\bfu_\varepsilon\Delta^{-1}\nabla\diver(\varrho_\varepsilon\bfu_\varepsilon)\dxs\\
&=J_1+\cdots +J_{10}.
\end{aligned}
\end{align}
Our goal is to find an estimate for the expectation of $J_6$ which means that we have to find suitable bounds for all the other terms. Let the term on the left hand side be denoted by $J_0$. There holds that
\begin{equation*}
\begin{split}
\stred|J_0|&\leq C\,\stred\|\Delta^{-1}\nabla\varrho_\varepsilon\|^2_{L^\infty(\mt)}+C\,\stred\int_{\mt}\varrho_\varepsilon|\bu_\varepsilon|^2\,\dif x.
\end{split}
\end{equation*}
Using the continuity of the operator $\Delta^{-1}\nabla$ and Sobolev's embedding theorem, we obtain for any $p\in(3,\beta)$ that
\begin{align}\label{eq:Linfty}
\begin{aligned}
\|\Delta^{-1}\nabla\varrho_\varepsilon\|_{L^\infty(\mt)}&\leq C\,\|\nabla^2\Delta^{-1}(\varrho_\varepsilon-(\varrho_\varepsilon)_{\mt})\|_{L^p(\mt)}\leq C\,\| \varrho_\varepsilon\|_{L^p(\mt)}.
\end{aligned}
\end{align}
Hence $\stred|J_0|\leq C$ due to \eqref{aprho}. Note that in particular we have shown that $\Delta^{-1}\nabla\varrho_\varepsilon\in L^p(\Omega;L^\infty(Q))$ uniformly in $\varepsilon$. Besides, $J_1$ can be estimated by the same argument. As $\varrho_\varepsilon\in L^2(\Omega\times Q)$ uniformly due to (\ref{aprho}) and $\beta\geq2$ we deduce that
 $\E|J_2|\leq C$ as a consequence of (\ref{apv}) and the continuity of the operator $\nabla\Delta^{-1}\nabla$. Similar arguments lead to the bound for $J_3$. The most critical term, $J_4$, can be estimated using the continuity of $\nabla\Delta^{-1}\nabla$, the Sobolev imbedding theorem, the H\"older inequality, (\ref{apv}) and (\ref{aprho})
\begin{align*}
\E|J_4|&\leq C\,\E\int_0^t\| \varrho_\varepsilon\|_3\| \bfu_\varepsilon\|_6^2\| \varrho_\varepsilon\|_3\,\dif s\leq C\,\E\bigg[\sup_{0\leq s\leq t}\| \varrho_\varepsilon\|^2_3\int_0^t\int_ {\mathbb T^3}|\bfu_\varepsilon|^2+|\nabla \bfu_\varepsilon|^2\,\dif x\,\dif s\bigg]\\
&\leq C\,\bigg(\E\sup_{0\leq s\leq t}\| \varrho_\varepsilon\|_\beta^{4}\bigg)^{\frac{1}{2}}\bigg(\E\bigg[\int_0^t\int_ {\mathbb T^3}|\bfu_\varepsilon|^2+|\nabla \bfu_\varepsilon|^2\dxt\bigg]^{2}\bigg)^{\frac{1}{2}}\leq C.
\end{align*}
For $J_5$ we have on account of \eqref{eq:Linfty}, \eqref{est:nablarho}, \eqref{apv} and \eqref{aprho}
\begin{align*}
\stred|J_5|&\leq\,\stred\sup_{0\leq s\leq t}\|\Delta^{-1}\nabla\varrho_\varepsilon\|_{L^\infty(\mt)}^2+\,\stred\bigg[\int_0^t\int_{\mt}|\nabla\bfu_\varepsilon|^2\dxt\bigg]^2\\
&\qquad+\,\stred\bigg[\int_0^t\int_{\mt}\varepsilon^2|\nabla\varrho_\varepsilon|^2\dxt\bigg]^2\leq C.
\end{align*}
By \eqref{aprho} we can easily bound the expectation of $J_7$. Let us now justify that the stochastic integral $J_8$ is a square integrable martingale and hence has zero expected value. Towards this end, we make use of the It\^o isometry and the assumption \eqref{growth1} as well as \eqref{eq:Linfty}, \eqref{aprhov} and \eqref{aprho} to obtain (recall \eqref{stochest} and $(\varrho_\varepsilon)_{\mt}=(\varrho_\varepsilon(0))_{\mt}\leq \overline{\varrho}$)
\begin{equation*}
\begin{split}
\stred&\bigg|\sum_{k\geq1}\int_0^t\int_{\mt}\Delta^{-1}\nabla\varrho_\varepsilon\cdot g_k (\varrho_\varepsilon,\varrho_\varepsilon\bu_\varepsilon)\,\dif x\,\dif\beta_k(s)\bigg|^2\\
&\qquad=\stred\int_0^t\sum_{k\geq 1}\bigg(\int_{\mt}\Delta^{-1}\nabla\varrho_\varepsilon \cdot g_k(\varrho_\varepsilon,\varrho_\varepsilon\bu_\varepsilon)\,\dif x\bigg)^2\dif s\\
&\qquad\leq \,\stred\bigg[\|\Delta^{-1}\nabla\varrho_\varepsilon\|_{L^\infty(Q)}^2\int_0^t\sum_{k\geq 1}\bigg(\int_{\mt}|g_k(\varrho_\varepsilon,\varrho_\varepsilon\bu_\varepsilon)|\,\dif x\bigg)^2\,\dif s\bigg]\\
&\qquad\leq \,C(\overline\varrho)\,\stred\bigg[\|\Delta^{-1}\nabla\varrho_\varepsilon\|_{L^\infty(Q)}^2\int_0^t\bigg(\sum_{k\geq 1}\int_{\mt}\varrho_\varepsilon^{-1}|g_k(\varrho_\varepsilon,\varrho_\varepsilon\bu_\varepsilon)|^2\,\dif x\bigg)\,\dif s\bigg]\\
&\qquad\leq \,C(\overline\varrho)\,\stred\|\Delta^{-1}\nabla\varrho_\varepsilon\|_{L^\infty(Q)}^4+C(\overline\varrho)\,\stred\bigg[\int_0^t\int_{\mt}\big(1+\varrho_\varepsilon^\gamma+\varrho_\varepsilon|\bfu_\varepsilon|^2\big)\,\dxs\bigg]^2\leq C(\overline\varrho).
\end{split}
\end{equation*}
We conclude that $\stred J_8=0$.
So the only remaining terms are $J_{9}$ and $J_{10}$ that can be estimated together. Indeed, due to the properties of the operator $\Delta^{-1}\nabla$
\begin{equation*}
\begin{split}
\stred J_9+\stred J_{10}&\leq\sqrt{\varepsilon}C\bigg(\stred\int_0^t\int_{\mt}|\varrho_\varepsilon\bu_\varepsilon|^2\,\dif x\,\dif s\bigg)^\frac{1}{2}\bigg(\stred\int_0^t\int_{\mt}|\sqrt{\varepsilon}\nabla\varrho_\varepsilon|^2\,\dif x\,\dif s\bigg)^\frac{1}{2}\\
&\qquad +C\,\stred\int_0^t\int_{\mt}|\varrho_\varepsilon\bu_\varepsilon|^2\,\dif x\,\dif s
\end{split}
\end{equation*}
which is finite since for any $p\in [1,\infty)$ and uniformly in $\varepsilon$
\begin{equation}\label{eq:L2}
\begin{split}
\varrho_\varepsilon\bu_\varepsilon\in L^p(\Omega;L^2(0,T;L^2(\mt)))
\end{split}
\end{equation}
which is a consequence of the fact that
$$\varrho_\varepsilon\in L^q(\Omega;L^\infty(0,T;L^3(\mt))),\qquad \bu_\varepsilon\in L^p(\Omega;L^2(0,T;L^6(\mt)))$$
uniformly in $\varepsilon.$
Plugging all together we obtain \eqref{eq:gamma+1} uniformly in $\varepsilon$.
\end{proof}
\end{proposition}

\subsection{Compactness}
\label{subsec:compactness}

Let us define the path space $\mathcal{X}=\mathcal{X}_\varrho\times\mathcal{X}_\bu\times\mathcal{X}_{\varrho\bu}\times\mathcal{X}_W$ where
\begin{align*}
\mathcal{X}_\varrho&=C_w([0,T];L^\beta(\mt))\cap \big(L^{\beta+1}(Q),w\big),&\mathcal{X}_\bu&=\big(L^2(0,T;W^{1,2}(\mt)),w\big),\\
\mathcal{X}_{\varrho\bu}&=C_w([0,T];L^\frac{2\beta}{\beta+1}(\mt)),&\mathcal{X}_W&=C([0,T];\mathfrak{U}_0).
\end{align*}
Let us denote by $\mu_{\varrho_\varepsilon}$, $\mu_{\bu_\varepsilon}$ and $\mu_{\varrho_\varepsilon\bu_\varepsilon}$, respectively, the law of $\varrho_\varepsilon$, $\bu_\varepsilon$ and $\varrho_\varepsilon\bu_\varepsilon$ on the corresponding path space. By $\mu_W$ we denote the law of $W$ on $\mathcal{X}_W$ and their joint law on $\mathcal{X}$ is denoted by $\mu^\varepsilon$.

To proceed, it is necessary to establish tightness of $\{\mu^\varepsilon;\,\varepsilon\in(0,1)\}$. To this end, we observe that tightness of $\{\mu_{\bu_\varepsilon};\,\varepsilon\in(0,1)\}$ follows as in Proposition \ref{prop:bfutightness} using \eqref{apv}, tightness of $\{\mu_{\varrho_\varepsilon};\,\varepsilon\in(0,1)\}$ is as in Proposition \ref{prop:rhotight} using \eqref{aprho} and \eqref{eq:gamma+1} and tightness of $\mu_{W}$ is immediate and was discussed just before Corollary \ref{cor:tight}. So it only remains to show tightness for $\{\mu_{\varrho_\varepsilon\bfu_\varepsilon};\,\varepsilon\in(0,1)\}$ where the proof of Proposition \ref{prop:rhoutight} does not apply and requires some modifications.

\begin{proposition}\label{rhoutight1}
The set $\{\mu_{\varrho_\varepsilon\bu_\varepsilon};\,\varepsilon\in(0,1)\}$ is tight on $\mathcal{X}_{\varrho\bu}$.

\begin{proof}
We proceed similarly as in Proposition \ref{prop:rhoutight} and decompose $\varrho_\varepsilon\bu_\varepsilon$ into two parts, namely, $\varrho_\varepsilon\bu_\varepsilon(t)=Y^\varepsilon(t)+Z^\varepsilon(t)$, where
\begin{equation*}
 \begin{split}
Y^\varepsilon(t)&=\bfq(0)-\int_0^t\big[\diver(\varrho_\varepsilon\bu_\varepsilon\otimes\bu_\varepsilon)+\nu\Delta\bu_\varepsilon+(\lambda+\nu)\nabla\diver\bfu_\varepsilon\\
&\hspace{3.8cm}-a\nabla \varrho_\varepsilon^\gamma-\delta\nabla \varrho_\varepsilon^\beta\big]\dif s+\int_0^t\,\varPhi(\varrho_\varepsilon,\varrho_\varepsilon\bu_\varepsilon) \,\dif W(s),\\
Z^\varepsilon(t)&=\varepsilon\int_0^t\nabla\bu_\varepsilon\nabla\varrho_\varepsilon\,\dif s.
 \end{split}
\end{equation*}
By a similar approach as in Proposition \ref{prop:rhoutight}, we obtain H\"older continuity of $Y^\varepsilon$, namely, there exist $\vartheta>0$ and $m>3/2$ such that
\begin{equation*}
\stred\big\|Y^\varepsilon\|_{C^\vartheta([0,T];W^{-m,2}(\mt))}\leq C.
\end{equation*}
Indeed, concerning the stochastic integral, we obtain due to \eqref{growth1} (similarly to \eqref{stochest}) that
\begin{equation*}
\begin{split}
\stred&\,\bigg\|\int_s^t\varPhi(\varrho_\varepsilon,\varrho_\varepsilon\bu_\varepsilon)\,\dif W\bigg\|^\theta_{W^{-b,2}(\mt)}\leq C\,\stred\bigg(\int_s^t\sum_{k\geq1}\big\| g_k(\varrho_\varepsilon,\varrho_\varepsilon\bu_\varepsilon)\big\|_{W^
{-b,2}}^2\,\dif r\bigg)^\frac{\theta}{2}\\
&\leq C\,\stred\bigg(\int_s^t\sum_{k\geq1}\big\|g_k(\varrho_\varepsilon,\varrho_\varepsilon\bu_\varepsilon)\big\|_{L^
{1}}^2\,\dif r\bigg)^\frac{\theta}{2}\leq C\,\stred\bigg(\int_s^t\int_{\mt}(\varrho_\varepsilon+\varrho_\varepsilon|\bu_\varepsilon|^2+\varrho_\varepsilon^\gamma)\,\dif x\,\dif r\bigg)^{\theta/2}\\
&\leq C|t-s|^{\theta/2}\Big(1+\stred\sup_{0\leq t\leq T}\|\sqrt\varrho_\varepsilon\bu_\varepsilon\|_{L^{2}}^{\theta}+\stred\sup_{0\leq t\leq T}\|\varrho_\varepsilon\|_{L^\gamma}^{\theta\gamma/2}\Big)\leq C|t-s|^{\theta/2}
\end{split}
\end{equation*}
and the Kolmogorov continuity criterion applies. For the deterministic part, we make use of estimates \eqref{b1} - \eqref{b3} that are also valid uniformly in $\varepsilon$ (only employing \eqref{eq:gamma+1} instead of \eqref{aprhobeta}).

{\em Tightness of $(Z^\varepsilon)$.}
Next, we show that the set of laws $\{\prst\circ[Z^\varepsilon]^{-1};\,\varepsilon\in(0,1)\}$ is tight on $C([0,T];W^{-m,2}(\mt))$ for every $m>3/2$.
It follows immediately from \eqref{eq:van2} that (up to a subsequence)
$$\varepsilon\nabla\bu_\varepsilon\nabla \varrho_\varepsilon\rightarrow 0\quad\text{ in }\quad L^1(0,T;L^1(\mt))\quad\text{a.s.}$$
hence
$$Z^\varepsilon\rightarrow0\quad\text{ in }\quad C([0,T];L^1(\mt))\quad\text{a.s.}$$
This leads to convergence in law
$$Z^\varepsilon\overset{d\hspace{2.5mm}}{\rightarrow0}\quad\text{ on }\quad C([0,T];L^1(\mt))$$
and the claim follows as $L^1(\mt)\hookrightarrow W^{-m,2}(\mt)$ for $m>3/2$.

{\em Conclusion.} Let $\eta>0$ be given. According to tightness of $\{\prst\circ[Z^\varepsilon]^{-1}\}$ on $C([0,T];W^{-m,2}(\mt))$ there exists $A\subset C([0,T];W^{-m,2}(\mt))$ compact such that
$$\prst\big(Z^\varepsilon\notin A\big)<\eta/2.$$
Next, let use define the sets
\begin{equation*}
\begin{split}
B_{R}=&\big\{h\in L^\infty(0,T;L^\frac{2\beta}{\beta+1}(\mt));\,\|h\|_{L^\infty(0,T;L^\frac{2\beta}{\beta+1}(\mt))}\leq R\big\}\\
C_{R}=&\big\{h\in C^\vartheta([0,T];W^{-m,2}(\mt));\,\|h\|_{C^\vartheta([0,T];W^{-m,2}(\mt))}\leq R\big\}
\end{split}
\end{equation*}
and
$$K_R=B_{R}\cap \big(C_R+A\big).$$
Then it can be shown that $K_R$ is relatively compact in $\mathcal{X}_{\varrho\bu}$. The proof is based on the Arzel\`a-Ascoli theorem and follows closely the lines of the proof of \cite[Corollary B.2]{on2}.
As a consequence, we obtain
\begin{equation*}
\begin{split}
&\mu_{\varrho_\varepsilon\bu_\varepsilon}\big(K_R^c)=\prst\big([\varrho_\varepsilon\bu_\varepsilon\notin B_R]\cup[Y^\varepsilon+Z^\varepsilon\notin C_R+A]\big)\\
&\quad\leq\prst\Big(\|\varrho_\varepsilon\bu_\varepsilon\|_{L^\infty(0,T;L^\frac{2\beta}{\beta+1}(\mt))}>R\Big)+\prst\Big(\|Y^\varepsilon\|_{C^\vartheta([0,T];W^{-m,2}(\mt))}>R\Big)+\prst\big(Z^\varepsilon\notin A\big)\\
&\quad\leq \frac{C}{R}+\eta/2.
\end{split}
\end{equation*}
A suitable choice of $R$ completes the proof.
\end{proof}
\end{proposition}

\begin{corollary}
The set $\{\mu^\varepsilon;\,\varepsilon\in(0,1)\}$ is tight on $\mathcal{X}$. 
\end{corollary}

Now we have all in hand to apply the Jakubowski-Skorokhod representation theorem. It yields the following.

\begin{proposition}\label{prop:skorokhod1}
There exists a subsequence $\mu^\varepsilon$, a probability space $(\tilde\Omega,\tilde\mf,\tilde\prst)$ with $\mathcal{X}$-valued Borel measurable random variables $(\tilde\varrho_n,\tilde\bu_\varepsilon,\tilde\bq_\varepsilon,\tilde W_\varepsilon)$, $n\in\mn$, and $(\tilde\varrho,\tilde\bu,\tilde\bq,\tilde W)$ such that
\begin{enumerate}
 \item the law of $(\tilde\varrho_\varepsilon,\tilde\bu_\varepsilon,\tilde\bq_\varepsilon,\tilde W_\varepsilon)$ is given by $\mu^\varepsilon$, $\varepsilon\in(0,1)$,
\item the law of $(\tilde\varrho,\tilde\bu,\tilde\bq,\tilde W)$, denoted by $\mu$, is a Radon measure,
 \item $(\tilde\varrho_\varepsilon,\tilde\bu_\varepsilon,\tilde\bq_\varepsilon,\tilde W_\varepsilon)$ converges $\,\tilde{\prst}$-almost surely to $(\tilde\varrho,\tilde{\bu},\tilde\bq,\tilde{W})$ in the topology of $\mathcal{X}$.
\end{enumerate}
\end{proposition}
Although the passage to the limit argument follows the same scheme as the one presented in Section \ref{sec:viscous}, the lack of strong convergence of the density does not allow us to identify the limit of the terms where the dependence on $\varrho$ and $\varrho\bfu$ is nonlinear, namely, the pressure and the stochastic integral. Therefore, the identification of the limit is split into two steps: the aim of the remainder of this subsection is to apply the convergence established by the Skorokhod representation theorem and pass to the limit in \eqref{eq:approx}. In the next subsection, we introduce a stochastic generalization of the technique based on regularity of the effective viscous flux, which is originally due to Lions \cite{Li2}. By this we establish strong convergence of the approximate densities and identify the pressure terms as well as the stochastic integral.

In order to not repeat ourselves we will often refer the reader to Section \ref{sec:viscous} in the sequel and present detailed proofs only when new arguments are necessary.
We remark that the energy inequality \eqref{eq:apriorivarepsilon} continues to hold on the new probability space. Moreover, Proposition \ref{prop:higher} continues to hold on the new probability space.

\begin{lemma}\label{lemma:identif2}
The following convergences hold true $\tilde\prst$-a.s.
\begin{align}
\tilde\varrho_\varepsilon\tilde\bu_\varepsilon&\rightarrow  \tilde{\varrho}  \tilde{\bfu}\qquad\text{in}\qquad L^2(0,T;W^{-1,2}( \mathbb T^3))\label{conv:rhov2}\\
\tilde{\varrho}_\varepsilon  \tilde{\bfu}_\varepsilon\otimes  \tilde{\bfu}_\varepsilon&\rightharpoonup  \tilde{\varrho}  \tilde{\bfu}\otimes  \tilde{\bfu}\qquad\text{in}\qquad L^1(0,T;L^{1}( \mathbb T^3))\label{conv:rhovv2}
\end{align}

\begin{proof}
See Lemma \ref{lemma:identif} and Corollary \ref{cor:con}.
\end{proof}
\end{lemma}

Let $(\tilde{\mf}_t^\varepsilon)$ and $(\tilde{\mf}_t)$, respectively, be the $\tilde{\prst}$-augmented canonical filtration of the process $(\tilde\varrho_\varepsilon,\tilde{\bu}_\varepsilon,\tilde{W}_\varepsilon)$ and $(\tilde\varrho,\tilde{\bu},\tilde{W})$, respectively, that is
\begin{equation*}
\begin{split}
\tilde{\mf}_t^\varepsilon&=\sigma\big(\sigma\big(\bfr_t\tilde\varrho_\varepsilon,\,\bfr_t\tilde{\bu}_\varepsilon,\,\bfr_t \tilde{W}_\varepsilon\big)\cup\big\{N\in\tilde{\mf};\;\tilde{\prst}(N)=0\big\}\big),\quad t\in[0,T],\\
\tilde{\mf}_t&=\sigma\big(\sigma\big(\bfr_t\tilde\varrho, \,\bfr_t\tilde{\bu},\,\bfr_t\tilde{W}\big)\cup\big\{N\in\tilde{\mf};\;\tilde{\prst}(N)=0\big\}\big),\quad t\in[0,T].
\end{split}
\end{equation*}

We obtain the following result.

\begin{proposition}\label{prop:martsol}
For every $\varepsilon\in(0,1)$, $\big((\tilde{\Omega},\tilde{\mf},(\tilde{\mf}^\varepsilon_t),\tilde{\prst}),\tilde\varrho_\varepsilon,\tilde{\bu}_\varepsilon,\tilde{W}_\varepsilon\big)$ is a weak martingale solution to \eqref{eq:approx} with the initial law $\Gamma$. Furthermore, there exists $b> \frac{3}{2}$ together with a $W^{-b,2}(\mt)$-valued continuous square integrable $(\tilde\mf_t)$-martingale $\tilde M$ and
$$\tilde p\in L^\frac{\beta+1}{\beta}(\tilde\Omega\times Q)$$
such that 
$\big((\tilde{\Omega},\tilde{\mf},(\tilde{\mf}_t),\tilde{\prst}),\tilde\varrho,\tilde{\bu},\tilde p,\tilde{M}\big)$ is a weak martingale solution to
\begin{subequations}\label{eq:approximat22}
 \begin{align}
  \dif \tilde\varrho+\diver(\tilde\varrho\tilde\bu)\dif t&=0,\label{eq:approximat221}\\
  \dif(\tilde\varrho\tilde\bu)+\big[\diver(\tilde\varrho\tilde\bu\otimes\tilde\bu)-\nu\Delta\tilde\bu-(\lambda+\nu)\nabla\diver\tilde\bfu+\nabla\tilde p\, \big]\dif t&=\dif\tilde M\label{eq:approximat222}
 \end{align}
\end{subequations}
with the initial law $\Gamma$. Besides, \eqref{eq:approximat221} holds true in the renormalized sense.
\end{proposition}

\begin{proof}
The passage to the limit in \eqref{eq:approx1} employs \eqref{conv:rhov2} together with the arguments of Lemma \ref{eq:mass}. Concerning the passage to the limit in \eqref{eq:approx2}, we follow the approach of Proposition \ref{prop:limit1} and define for all $t\in[0,T]$ and $\bfvarphi\in C^\infty(\mt)$ the functionals
\begin{equation*}
\begin{split}
M_\varepsilon(\rho,\bfv,\bfq)_t&=\big\langle\bfq(t),\bfvarphi\big\rangle-\big\langle \bfq(0),\bfvarphi\big\rangle+\int_0^t\big\langle\bfq\otimes\bfv,\nabla\bfvarphi\big\rangle\,\dif r-\nu\int_0^t\big\langle\nabla\bfv,\nabla\bfvarphi\big\rangle\,\dif r\\
&\quad-(\lambda+\nu)\int_0^t\big\langle\diver\bfv,\diver\bfvarphi\big\rangle\,\dif r+a\int_0^t\big\langle\rho^\gamma,\diver\bfvarphi\big\rangle\,\dif r+\delta\int_0^t\big\langle\rho^\beta,\diver\bfvarphi\big\rangle\,\dif r\\
&\quad-\varepsilon\int_0^t\big\langle\nabla \bfv\nabla\rho,\bfvarphi\big\rangle\,\dif r,\\
N(\rho,\bfq)_t&=\sum_{k\geq1}\int_0^t\big\langle g_k(\rho, \bfq ),\bfvarphi\big\rangle^2\,\dif r,\\
N_k(\rho,\bfq)_t&=\int_0^t\big\langle g_k(\rho,\bfq),\bfvarphi\big\rangle\,\dif r,
\end{split}
\end{equation*}
and deduce that
\begin{equation}\label{exp111}
\begin{split}
&\tilde{\stred}\,h\big(\bfr_s\tilde\varrho_\varepsilon, \bfr_s\tilde{\bu}_\varepsilon,\bfr_s\tilde{W}_\varepsilon\big)\big[M_{\varepsilon}(\tilde\varrho_\varepsilon,\tilde\bu_\varepsilon,\tilde\varrho_\varepsilon\tilde\bfu_\varepsilon)_{s,t}\big]=0,
\end{split}
\end{equation}
\begin{equation}\label{exp211}
\begin{split}
&\tilde{\stred}\,h\big(\bfr_s\tilde\varrho_\varepsilon, \bfr_s\tilde{\bu}_\varepsilon,\bfr_s\tilde{W}_\varepsilon\big)\bigg[[M_{\varepsilon}(\tilde\varrho_\varepsilon,\tilde\bu_\varepsilon,\tilde\varrho_\varepsilon\tilde\bfu_\varepsilon)^2]_{s,t}-N(\tilde\varrho_\varepsilon,\tilde\varrho_\varepsilon\tilde\bfu_\varepsilon)_{s,t}\bigg]=0,
\end{split}
\end{equation}
\begin{equation}\label{exp311}
\begin{split}
&\tilde{\stred}\,h\big(\bfr_s\tilde\varrho_\varepsilon, \bfr_s\tilde{\bu}_\varepsilon,\bfr_s\tilde{W}_\varepsilon\big)\bigg[[M_{\varepsilon}(\tilde\varrho_\varepsilon,\tilde\bu_\varepsilon,\tilde\varrho_\varepsilon\tilde\bfu_\varepsilon)\tilde{\beta}_k^\varepsilon]_{s,t}-N_k(\tilde\varrho_\varepsilon,\tilde\varrho_\varepsilon\tilde\bfu_\varepsilon)_{s,t}\bigg]=0,
\end{split}
\end{equation}
which implies the first part of the statement.

As the next step, we will pass to the limit in \eqref{exp111}. We apply \eqref{est:rhobfu22} and \eqref{conv:rhovv2} for the convective term, \eqref{apv}, \eqref{est:nablarho} and \eqref{eq:van2} for the term involving the artificial viscosity $\varepsilon$. In the case of the pressure, we see that according to \eqref{eq:gamma+1} there exists $\tilde p\in L^\frac{\beta+1}{\beta}(\tilde\Omega\times Q)$ such that
$$a\tilde\varrho_\varepsilon^\gamma+\delta\tilde\varrho_\varepsilon^\beta\rightharpoonup\tilde p\qquad\text{in}\qquad L^\frac{\beta+1}{\beta}(\tilde\Omega\times Q)$$
hence in view of \eqref{aprho} we deduce
\begin{align*}
\tilde{\stred}&\,h\big(\bfr_s\tilde\varrho_\varepsilon, \bfr_s\tilde{\bu}_\varepsilon,\bfr_s\tilde{W}_\varepsilon\big)\bigg[a\int_0^t\big\langle\tilde\varrho_\varepsilon^\gamma,\diver\bfvarphi\big\rangle\,\dif r+\delta\int_0^t\big\langle\tilde\varrho_\varepsilon^\beta,\diver\bfvarphi\big\rangle\,\dif r\bigg]\\
&\qquad\rightarrow\tilde{\stred}\,h\big(\bfr_s\tilde\varrho, \bfr_s\tilde{\bu},\bfr_s\tilde{W}\big)\bigg[\int_0^t\big\langle\tilde p,\diver\bfvarphi\big\rangle\,\dif r\bigg].
\end{align*}
Convergence of the remaining terms is obvious and therefore we have proved that
\begin{equation}\label{exp111a}
\begin{split}
&\tilde{\stred}\,h\big(\bfr_s\tilde\varrho, \bfr_s\tilde{\bu},\bfr_s\tilde{W}\big)\big[\big\langle\tilde M,\bfvarphi\big\rangle_{s,t}\big]=0,
\end{split}
\end{equation}
where
\begin{align*}
\tilde M_{t}&=\tilde\varrho\tilde\bfu(t)-\tilde\varrho\tilde\bfu(0)-\int_0^t\diver(\tilde\varrho\tilde\bfu\otimes\tilde\bfu)\,\dif r+\nu\int_0^t\Delta\tilde\bfu\,\dif r\\
&\qquad+(\lambda+\nu)\int_0^t\nabla\diver\tilde\bfu\,\dif r-\int_0^t\nabla\tilde p\,\dif r.
\end{align*}
Hence $\tilde M$ is a continuous $(\tilde\mf_t)$-martingale and possesses moments of any order due to our uniform estimates.

To conclude the proof, we will show that $(\tilde\varrho,\tilde \bfu)$ solves the continuity equation in the renormalized sense. We apply to \eqref{eq:approximation21} a standard smoothing operator $S^m$ (which is the convolution with an approximation to the identity in space) such that $\tilde{\p}\otimes\mathcal L^{4}$-a.e. on $\tilde{\Omega}\times Q$
\begin{align}\label{eq:ren1}
\partial_t S^m[\tilde{\varrho}]+\Div\big(S^m[\tilde{\varrho}]\tilde{\bfu}\big)
=\Div\big(S^m[\tilde{\varrho}]\tilde{\bfu}-S^m[\tilde{\varrho}\tilde{\bfu}]\big).
\end{align}
Setting $\tilde{r}_m:=\Div\big(S^m[\tilde{\varrho}]\tilde{\bfu}-S_m[\tilde{\varrho}\tilde{\bfu}]\big)$ we infer from the commutation lemma (see e.g. \cite[Lemma 2.3]{Li1}) that $\tilde\p\otimes\mathcal L^1$-a.e.
\begin{align*}
\|\tilde{r}_m\|_{L^q_x}\leq \|\tilde{\bfu}\|_{W^{1,2}_x}\|\tilde{\varrho}\|_{L^{\beta+1}_x},\qquad \tfrac{1}{q}=\tfrac{1}{2}+\tfrac{1}{\beta+1},
\end{align*}
as well as $\tilde{r}_m\rightarrow0$ in $L^1(\mathbb T^3)$.
Both together imply $\tilde{r}_m\rightarrow0$ in $L^1(\tilde{\Omega}\times Q)$. Let $b:\R\rightarrow\R$ be a $C^1$-function with compact support.
We multiply
(\ref{eq:ren1}) by $b'(S^m[\tilde{\varrho}])$ to obtain
\begin{align*}
\partial_t b(S^m[\tilde{\varrho}])&+\Div\big(b(S^m[\tilde{\varrho}])\tilde{\bfu}\big)
+\big(b'(S^m[\tilde{\varrho}])S^m[\tilde{\varrho}]-b(S^m[\tilde{\varrho}])\big)\Div\tilde{\bfu}
=\tilde{r}_mb'(S^m[\tilde{\varrho}]).
\end{align*}
As $b'$ is bounded the right hand side vanishes for $m\rightarrow\infty$ (in the $L^1(\tilde{\Omega}\times Q)$-sense) and we gain
\begin{align}\label{eq:rendelta}
\partial_t b(\tilde{\varrho})&+\Div\big(b(\tilde{\varrho})\tilde{\bfu}\big)
+\big(b'(\tilde{\varrho})\tilde{\varrho}-b(\tilde{\varrho})\big)\Div\tilde{\bfu}
=0
\end{align}
in the sense of distributions, i.e.
\begin{align*}
\int_Q b(\tilde{\varrho})\,\partial_t\phi\dxt &=-\int_Q\big(b(\tilde{\varrho})\tilde{\bfu}\big)\cdot\nabla\phi\dxt
+\int_Q\big(b'(\tilde{\varrho})\tilde{\varrho}-b(\tilde{\varrho})\big)\Div\tilde{\bfu}\,\phi\dxt\\
&-\int_{\mt} b(\tilde{\varrho}(0))\phi(0)\dx
\end{align*}
for all $\phi\in C^\infty([0,T]\times \mt)$ with $\phi(T)=0$ which is equivalent to 
\begin{align*}
\int_{\mt} b(\tilde{\varrho})\,\psi\dx &=\int_{\mt} b(\tilde{\varrho}(0))\psi(0)\dx+\int_0^t\int_{\mt}\big(b(\tilde{\varrho})\tilde{\bfu}\big)\cdot\nabla\psi\dxs\\
&-\int_0^t\int_{\mt}\big(b'(\tilde{\varrho})\tilde{\varrho}-b(\tilde{\varrho})\big)\Div\tilde{\bfu}\,\psi\dxs
\end{align*}
for all $\psi\in C^\infty(\mt)$.
\end{proof}

\subsection{Strong convergence of density}
\label{subsec:strongconvdensity}

In the first step, we proceed as in Proposition \ref{prop:higher} and test \eqref{eq:approx2} by $\Delta^{-1}\nabla \tilde\varrho_\varepsilon$, that is, we apply It\^{o}'s formula to the function $f(\rho,\bfq)=\int_{\mt}\bfq\cdot\Delta^{-1}\nabla\rho\dx$ which yields the corresponding version of \eqref{eq:test}. Let us also keep the same notation, i.e. we denote by $J_0$ the term on the left hand side and by $J_1,\dots J_{10}$ the terms on the right hand side. Taking the expectation we observe that the stochastic integral $J_8$ is a martingale. Similarly for the limit equation we obtain 
\begin{align}\label{eq:hjhj}
\begin{aligned}
\tilde\stred\int_ {\mathbb T^3} & \tilde\varrho\tilde \bfu\cdot \Delta^{-1}\nabla \tilde\varrho\dx=\tilde\stred\int_ {\mathbb T^3}  \tilde\varrho\tilde \bfu(0)\cdot \Delta^{-1}\nabla \tilde\varrho(0)\dx\\
&\quad-\nu\,\tilde\stred\int_0^t\int_ {\mathbb T^3} \nabla \tilde\bfu:\nabla\Delta^{-1} \nabla\tilde\varrho\dxs-(\lambda+\nu)\,\tilde\stred\int_0^t\int_ {\mathbb T^3} \diver\tilde\bfu\,\tilde\varrho\dxs\\
&\quad+\tilde\stred\int_0^t\int_ {\mathbb T^3}  \tilde\varrho\tilde \bfu\otimes \tilde\bfu:\nabla\Delta^{-1}\nabla \tilde\varrho\dxs+\tilde\stred\int_0^t\int_ {\mathbb T^3}\tilde \varrho\,\tilde p\dxs\\
&\quad-\tilde\stred\int_0^t(\tilde\varrho)_{\mt}\int_ {\mathbb T^3}\tilde p\dxs-\tilde\stred\int_0^t\int_{\mt}\tilde\varrho\tilde\bfu\nabla\Delta^{-1}\diver(\tilde\varrho\tilde\bfu)\dxs\\
&=\tilde\stred K_1+\cdots +\tilde\stred K_{7}.
\end{aligned}
\end{align}
To see why expectation of the stochastic integral vanishes, let us recall that, at this level, the It\^o formula can only be applied after a preliminary step of mollification. That is, mollification of \eqref{eq:approximat22} and application of the 1-dimensional It\^o formula to the product (where $x\in\mt$ is fixed)
$$(\tilde\varrho\tilde\bfu)^\kappa(x) \big(\Delta^{-1}\nabla\tilde\varrho\big)^\kappa(x).$$
This yields a stochastic integral of the form
$$\int_0^t \big(\Delta^{-1}\nabla\tilde\varrho\big)^\kappa(s,x)\,\dif\tilde M^\kappa(s,x).$$
Now, we observe that by \eqref{eq:approximat221}
$$\tilde \varrho\in L^q(\Omega;C^{0,1}([0,T];W^{-1,\frac{2\beta}{\beta+1}}(\mt))).$$
Hence $\big(\Delta^{-1}\nabla\tilde\varrho\big)^\kappa$ is a process with Lipschitz continuous trajectories and values in $C^\infty(\mt)$.
Consequently, we may use the integration by parts formula which follows easily from the It\^o formula applied to the product
$$\big(\Delta^{-1}\nabla\tilde\varrho\big)^\kappa(x)\tilde M^\kappa(x)$$
and infer that
\begin{align*}
&\int_0^t\big(\Delta^{-1}\nabla\tilde\varrho\big)^\kappa(s,x)\,\dif\tilde M^\kappa(s,x)\\
&\quad=\big(\Delta^{-1}\nabla\tilde\varrho\big)^\kappa(t,x)\,\tilde M^\kappa(t,x)-\int_0^t\tilde M^\kappa(s,x)\,\dif\big(\Delta^{-1}\nabla\tilde\varrho\big)^\kappa(s,x).
\end{align*}
But this necessarily implies that
\begin{equation}\label{exp}
\tilde\stred\int_0^t\big(\Delta^{-1}\nabla\tilde\varrho\big)^\kappa(s,x)\,\dif\tilde M^\kappa(s,x)=0.
\end{equation}
Indeed, let $A$ be a square integrable adapted process of bounded variation, let $N$ be a square integrable continuous martingale with $N_0=0$ and let $0=t_0\leq t_1\leq \cdots\leq t_n=t$ be a partition of $[0,t]$. Define
$$N^\Pi_s=\sum_{k=1}^n N_{t_k}\ind_{(t_{k-1},t_k]}(s).$$
Then it holds
\begin{align*}
\stred\int_0^tN^\Pi_s\,\dif A_s&=\stred\sum_{k=1}^n N_{t_k}(A_{t_k}-A_{t_{k-1}})=\stred\bigg[\sum_{k=1}^n N_{t_k}A_{t_k}-\sum_{k=0}^{n-1} N_{t_{k+1}}A_{t_k}\bigg]\\
&=\stred [N_tA_t]-\stred\sum_{k=0}^{n-1}A_{t_k}(N_{t_{k+1}}-N_{t_k})=\stred [N_tA_t]
\end{align*}
and letting the mesh size of the partition vanish we obtain by dominated convergence theorem
\begin{align*}
\stred\int_0^tN_s\,\dif A_s&=\stred [N_tA_t].
\end{align*}
Accordingly \eqref{exp} follows and \eqref{eq:hjhj} is justified.

Therefore, we obtain 
\begin{equation}\label{eq:kkkn}
 \begin{split}
 \tilde\stred\int_0^t\int_{\mt}&\big(a\tilde\varrho_\varepsilon^\gamma+\delta\tilde\varrho^\beta_\varepsilon
-(\lambda+2\nu)\diver\tilde\bu_\varepsilon\big)\tilde\varrho_\varepsilon\,\dif x\,\dif t=\tilde\stred\big[J_0-J_1-J_5+J_7-J_9\big]\\
&\qquad\quad+\tilde\stred\int_0^t\int_{\mt}\tilde u^i_\varepsilon\big(\tilde\varrho_\varepsilon\mathcal{R}_{ij}[\tilde\varrho_\varepsilon\tilde u^{j}_\varepsilon]-\tilde\varrho_\varepsilon\tilde u^{j}_n\mathcal{R}_{ij}[\tilde\varrho_\varepsilon]\big)\,\dif x\,\dif \sigma,
 \end{split}
\end{equation}
and
\begin{equation}\label{kkk}
\begin{split}
\tilde\stred\int_0^t\int_{\mt}&\big(\tilde p-(\lambda+2\nu)\diver\tilde\bu\big)\tilde\varrho\,\dif x\,\dif t=\tilde\stred\big[K_0-K_1+K_6\big]\\
&\qquad\quad+\tilde\stred\int_0^t\int_{\mt}\tilde u^i\big(\tilde\varrho\mathcal{R}_{ij}[\tilde\varrho\tilde u^{j}]-\tilde\varrho\tilde u^{j}\mathcal{R}_{ij}[\tilde\varrho]\big)\,\dif x\,\dif \sigma,
\end{split}
\end{equation}
where we used the Einstein summation convention. The operator $\mathcal{R}$ is defined by $\mathcal{R}_{ij}=\partial_j\Delta^{-1}\partial_i$. Now, by definition of $\tilde p$ and the convergence
$$(\tilde\varrho_\varepsilon(t))_{\mt}\rightarrow (\tilde\varrho(t))_{\mt}\quad\text{for a.e.}\quad(\omega,t)$$
together with \eqref{aprho}
it follows that $\tilde\stred J_7\rightarrow\tilde\stred K_6$. Moreover, it can be shown that $\tilde\stred J_5\rightarrow0$ and $\tilde\stred J_9\rightarrow0.$ Indeed,
\begin{equation*}
 \begin{split}
\tilde\stred |J_5|&\leq\,C\sqrt{\varepsilon}\,\Big(\tilde\stred\|\nabla\Delta^{-1}\tilde\varrho_\varepsilon\|^3_{L^\infty(Q)} +\tilde\stred\|\nabla\tilde\bu_\varepsilon\|^3_{L^2(0,T;L^2(\mt))}+\tilde\stred\|\sqrt{\varepsilon}\nabla\tilde\varrho_\varepsilon\|^3_{L^2(0,T;L^2(\mt))}\Big)\\
&\leq \,C\sqrt{\varepsilon}
 \end{split}
\end{equation*}
and
\begin{equation*}
 \begin{split}
\tilde\stred| J_9|&\leq\,C\sqrt{\varepsilon}\,\Big(\tilde\stred\|\tilde\varrho_\varepsilon\|^3_{L^\infty(0,T;L^3(\mt))} +\tilde\stred\|\tilde\bu_\varepsilon\|^3_{L^2(0,T;L^6(\mt))}+\tilde\stred\|\sqrt{\varepsilon}\nabla\tilde\varrho_\varepsilon\|^3_{L^2(0,T;L^2(\mt))}\Big)\\
&\leq \,C\sqrt{\varepsilon}.
 \end{split}
\end{equation*}
Next, we prove that $\tilde\stred J_0\rightarrow\tilde\stred K_0$ and similarly $\tilde\stred J_1\rightarrow\tilde\stred K_1$. Due to Proposition \ref{prop:skorokhod1}, \eqref{conv:rhov2} and the compactness of the operator $\Delta^{-1}\nabla $ on $L^\beta(\mt)$ we have for any fixed $t\in[0,T]$,
$$\Delta^{-1}\nabla\tilde\varrho_\varepsilon(t)\rightarrow\Delta^{-1}\nabla\tilde\varrho(t)\quad\text{in}\quad L^\beta(\mt)\quad\tilde\prst\text{-a.s.},$$
$$\tilde\varrho_\varepsilon\tilde\bu_\varepsilon(t)\rightharpoonup\tilde\varrho\tilde\bu(t)\quad\text{in}\quad L^\frac{2\beta}{\beta+1}(\mt)\quad\tilde\prst\text{-a.s.}$$
Hence due to the assumption $\beta> 4$
$$\int_{\mt}\tilde\varrho_\varepsilon\tilde\bu_\varepsilon(t)\cdot\Delta^{-1}\nabla\tilde\varrho_\varepsilon(t)\,\dif x\rightarrow\int_{\mt}\tilde\varrho\tilde\bu(t)\cdot \Delta^{-1}\nabla\tilde\varrho(t)\,\dif x\quad\tilde\prst\text{-a.s.}$$
This, together with the following bound, for all $p\geq1$,
\begin{equation*}
\begin{split}
\tilde\stred\,&\bigg|\int_{\mt}\tilde\varrho_\varepsilon\tilde\bu_\varepsilon(t)\cdot\Delta^{-1}\nabla\tilde\varrho_\varepsilon(t)\,\dif x\bigg|^p\\
&\quad\leq C\,\tilde\stred\|\Delta^{-1}\nabla\tilde\varrho_\varepsilon\|_{L^\infty(\mt)}^{2p}+C\,\tilde\stred\bigg[\int_{\mt}|\tilde\varrho_\varepsilon\tilde\bu_\varepsilon|\dif x\bigg]^p\leq C
\end{split}
\end{equation*}
yields the claim.

Now we come to the crucial point. In order to establish convergence of the left hand side of \eqref{eq:kkkn} to the left hand side of \eqref{kkk}, we need to verify convergence of the remaining term on the right hand side of \eqref{eq:kkkn} to the corresponding one in \eqref{kkk}. Since $ \tilde\bfu_\varepsilon$ is weakly convergent in $L^2(\Omega;L^2(0,T;W^{1,2}(\mt)))$, we have to show that
$ \tilde\varrho_\varepsilon\mathcal R[ \Tilde\varrho_\varepsilon \tilde\bfu_\varepsilon]- \tilde\varrho_\varepsilon \tilde\bfu_\varepsilon\mathcal R[ \tilde\varrho_\varepsilon]$ converges strongly in $L^2(\Omega;L^2(0,T;W^{-1,2}(\mt)))$. For the identification of the limit we make use of the div-curl lemma.

From Proposition \ref{prop:skorokhod1} we obtain that
$$\tilde\varrho_\varepsilon\rightharpoonup \tilde\varrho\quad\text{in}\quad L^{\beta}(\mt)\quad \tilde\prst\otimes\mathcal{L}\text{-a.e.},$$
$$\tilde\varrho_\varepsilon\tilde\bu_\varepsilon\rightharpoonup\tilde\varrho\tilde\bu\quad\text{in}\quad L^\frac{2\beta}{\beta+1}(\mt)\quad\tilde\prst\otimes\mathcal{L}\text{-a.e.}$$
Hence we can apply \cite[Lemma 3.4]{feireisl1} to conclude that
$$\tilde\varrho_\varepsilon\mathcal{R}_{ij}[\tilde\varrho_\varepsilon\tilde\bu_\varepsilon]-\tilde\varrho_\varepsilon\tilde\bu_\varepsilon\mathcal{R}_{ij}[\tilde\varrho_\varepsilon]\rightharpoonup\tilde\varrho\mathcal{R}_{ij}[\tilde\varrho\tilde\bu]-\tilde\varrho\tilde\bu\mathcal{R}_{ij}[\tilde\varrho]\quad\text{in}\quad L^r(\mt)\quad \tilde\prst\otimes\mathcal{L}\text{-a.e.},$$
where
$$\frac{1}{r}=\frac{1}{\beta}+\frac{\beta+1}{2\beta}<\frac{5}{6}$$
provided $\beta>\frac{9}{2}$. Therefore $L^r(\mt)$ is compactly embedded into $W^{-1,2}(\mt)$ and as a consequence,
$$\tilde\varrho_\varepsilon\mathcal{R}_{ij}[\tilde\varrho_\varepsilon\tilde\bu_\varepsilon]-\tilde\varrho_\varepsilon\tilde\bu_\varepsilon\mathcal{R}_{ij}[\tilde\varrho_\varepsilon]\rightarrow\tilde\varrho\mathcal{R}_{ij}[\tilde\varrho\tilde\bu]-\tilde\varrho\tilde\bu\mathcal{R}_{ij}[\tilde\varrho]\quad\text{in}\quad W^{-1,2}(\mt)\quad \tilde\prst\otimes\mathcal{L}\text{-a.e.}$$
Moreover, it is possible to show that for any $p\in\big(2,\frac{\beta}{2}\big)$ (using continuity of $\mathcal R_{ij}$, H\"older's inequality as well as Proposition \ref{prop:skorokhod1})
\begin{equation*}
\begin{split}
\tilde\stred\int_0^T&\big\|\tilde\varrho_\varepsilon \mathcal{R}_{ij}[\tilde\varrho_\varepsilon\tilde\bu_\varepsilon]-\tilde\varrho_\varepsilon\tilde\bu_\varepsilon \mathcal{R}_{ij}[\tilde\varrho_\varepsilon]\big\|_{W^{-1,2}(\mt)}^p\\
&\leq C\,\tilde\stred\int_0^T\|\tilde\varrho_\varepsilon\|_{L^{\beta}(\mt)}^{2p}\dif t+C\,\tilde\stred\sup_{0\leq t\leq T}\|\tilde\varrho_\varepsilon\tilde\bu_\varepsilon\|_{L^\frac{2\beta}{\beta+1}(\mt)}^{2p}\leq C
\end{split}
\end{equation*}
which gives the desired convergence
$$\tilde\varrho_\varepsilon\mathcal{R}_{ij}[\tilde\varrho_\varepsilon\tilde\bu_\varepsilon]-\tilde\varrho_\varepsilon\tilde\bu_\varepsilon\mathcal{R}_{ij}[\tilde\varrho_\varepsilon]\rightarrow\tilde\varrho\mathcal{R}_{ij}[\tilde\varrho\tilde\bu]-\tilde\varrho\tilde\bu\mathcal{R}_{ij}[\tilde\varrho]\quad\text{in}\quad L^2(\Omega;L^2(0,T;W^{-1,2}(\mt))).$$
Thus we conclude that
\begin{align}\label{eq:eq}
\begin{aligned}
\tilde\stred\int_{Q}&\tilde u^i_\varepsilon\big(\tilde\varrho_\varepsilon\mathcal{R}_{ij}[\tilde\varrho_\varepsilon\tilde u^{j}_\varepsilon]-\tilde\varrho_\varepsilon\tilde u^{j}_\varepsilon\mathcal{R}_{ij}[\tilde\varrho_\varepsilon]\big)\,\dif x\,\dif t\\
&\qquad\rightarrow \tilde\stred\int_{Q}\tilde u^i\big(\tilde\varrho\mathcal{R}_{ij}[\tilde\varrho\tilde u^{j}]-\tilde\varrho\tilde u^{j}\mathcal{R}_{ij}[\tilde\varrho]\big)\,\dif x\,\dif t
\end{aligned}
\end{align}
and accordingly
\begin{align}\label{eq:flux}
\tilde\E\int_{Q} \big( a\tilde{\varrho}_\varepsilon^\gamma+\delta\tilde\varrho_\varepsilon^\beta-(\lambda+2\nu)\Div \tilde{\bfu}_\varepsilon\big)\, \tilde{\varrho}_\varepsilon\dxt\rightarrow\tilde\E\int_Q \big( \tilde{p}-(\lambda+2\nu)\Div \tilde{\bfu}\big)\, \tilde{\varrho}\dxt.
\end{align}
As the next step, we intend to prove the following
\begin{align}\label{eq:fluxfinal}
\limsup_{\varepsilon\rightarrow \infty} \tilde\E\int_Q \big(a\tilde{\varrho}_\varepsilon^{\gamma}+\delta\tilde\varrho_\varepsilon^\beta\big)\tilde\varrho_\varepsilon\dxt\leq \tilde\E\int_Q \tilde{\varrho}\, \tilde{p}\dxt.
\end{align}
Towards this end, we make use of the continuity equation \eqref{eq:approx1} and its limit equation in the renormalized form. We consider function $b:[0,\infty)\rightarrow\R$ which is convex and globally Lipschitz continuous. As $\tilde{\varrho}_\varepsilon$ solves \eqref{eq:approx1} a.e. we gain
\begin{align*}
\partial_t b (\tilde{\varrho}_\varepsilon)+\Div(b(\tilde{\varrho}_\varepsilon)\tilde{\bfu}_\varepsilon)+(b'(\tilde{\varrho}_\varepsilon)\tilde{\varrho}_\varepsilon-b(\tilde{\varrho}_\varepsilon))\Div \tilde{\bfu}_\varepsilon-\varepsilon\Delta b(\tilde{\varrho}_\varepsilon)\leq0
\end{align*}
$\p\otimes\mathcal L^4$-a.e.
and hence
\begin{align*}
\int_0^T\int_{\mathbb T^3}\big(b'(\tilde{\varrho}_\varepsilon)\tilde{\varrho}_\varepsilon-b(\tilde{\varrho}_\varepsilon)\big)\Div \tilde{\bfu}_\varepsilon\dxt\leq \int_{\mathbb T^3}b\big(\tilde{\varrho}_\varepsilon(0)\big)\dx-\int_{\mathbb T^3}b\big(\tilde{\varrho}_\varepsilon(T)\big)\dx.
\end{align*}
For $b(z)=z\ln z$ we have
\begin{align}\label{eq:renorm}
\int_0^T\int_{\mathbb T^3}\tilde{\varrho}_\varepsilon\Div \tilde{\bfu}_\varepsilon\dxt\leq \int_{\mathbb T^3}\tilde{\varrho}_\varepsilon(0)\ln \tilde{\varrho}_\varepsilon(0)\dx-\int_{\mathbb T^3}\tilde{\varrho}_\varepsilon(T)\ln \tilde{\varrho}_\varepsilon(T)\dx.
\end{align}
Since the limit functions $(\tilde{\varrho},\tilde{\bfu})$ solve \eqref{eq:approximat221} in the renormalized sense as shown in Proposition \ref{prop:martsol}, it follows that
\begin{align}\label{eq:renormlimit}
\int_0^T\int_{\mathbb T^3}\tilde{\varrho}\Div \tilde{\bfu}\dxt= \int_{\mathbb T^3}\tilde{\varrho}(0)\ln \tilde{\varrho}(0)\dx-\int_{\mathbb T^3}\tilde{\varrho}(T)\ln \tilde{\varrho}(T)\dx.
\end{align}
If we combine (\ref{eq:renorm}) and (\ref{eq:renormlimit}) with the weak lower semicontinuity of $\rho\mapsto \int_{\mathbb T^3}\rho\ln\rho\dx$, the fact that for every $\varepsilon$
the law of $\tilde\varrho_\varepsilon(0)$ coincides $\tilde\varrho(0)$ and is given by the projection of $\Gamma$ to the first coordinate, we deduce that
\begin{align}\label{eq:renormlimit2}
\limsup_{\varepsilon\rightarrow 0}\tilde\stred\int_0^T\int_{\mathbb T^3}\tilde{\varrho}_\varepsilon\Div \tilde{\bfu}_\varepsilon\dxt\leq \tilde\stred\int_0^T\int_{\mathbb T^3}\tilde{\varrho}\Div \tilde{\bfu}\dxt.
\end{align}
Using (\ref{eq:flux}) and (\ref{eq:renormlimit2}) we compute
\begin{align*}
&\limsup_{\varepsilon\rightarrow0} \tilde\E\int_Q \big(a\tilde{\varrho}_\varepsilon^{\gamma}+\delta\tilde\varrho_\varepsilon^\beta\big)\tilde\varrho_\varepsilon\dxt\bigg]\\
&\leq \lim_{\varepsilon\rightarrow0}\tilde\E\int_Q ( a\tilde{\varrho}_\varepsilon^\gamma+\delta\tilde{\varrho}_\varepsilon^\beta-(\lambda+2\nu)\Div \tilde{\bfu}_\varepsilon\big)\, \tilde{\varrho}_\varepsilon\dxt
+(\lambda+2\nu)\limsup_{\varepsilon\rightarrow0}\tilde\E\int_Q \Div \tilde{\bfu}_\varepsilon\, \tilde{\varrho}_\varepsilon\dxt\\
&\leq \tilde\E\int_Q \big( \tilde{p}-(\lambda+2\nu)\Div \tilde{\bfu}\big)\, \tilde{\varrho}\dxt
+(\lambda+2\nu)\tilde\E\int_Q \Div \tilde{\bfu}\, \tilde{\varrho}\dxt=\tilde\E\int_Q \tilde{\varrho}\, \tilde{p}\dxt
\end{align*}
which completes the proof of (\ref{eq:fluxfinal}).
The rest of the proof uses monotonicity of the mapping $t\mapsto t^\gamma$ and the Minty trick similarly to \cite[Section 3.5]{feireisl1}. We deduce that $\tilde p=a\tilde\varrho^\gamma+\delta\tilde\varrho^\beta$ and consequently the following strong convergence holds true
\begin{equation}\label{strongconv}
\tilde\varrho_\varepsilon\rightarrow\tilde\varrho\qquad\tilde\prst\otimes\mathcal{L}^4\text{-a.e.}
\end{equation}
Following the ideas of Lemma \ref{lem:strongq} we also have
\begin{equation}\label{strongconvq}
\tilde\varrho_\varepsilon\tilde{\bfu}_\varepsilon\rightarrow
\tilde\varrho\tilde\bfu\qquad\text{in}\quad L^q(\tilde\Omega\times Q)
\end{equation}
for all $q<\frac{2\beta}{\beta+1}$.
With this in hand, we can finally identify the limit in the stochastic term.

\begin{proposition}\label{prop:martsol1}
$\big((\tilde{\Omega},\tilde{\mf},(\tilde{\mf}_t),\tilde{\prst}),\tilde\varrho,\tilde{\bu},\tilde{W}\big)$
is a finite energy weak martingale solution to \eqref{eq:approximation2} with the initial law $\Gamma$.\footnote{This has to be understood in the sense of Definition \ref{def:sol} via an obvious modification by adding the artificial pressure.}
\end{proposition}
\begin{proof}
According to Proposition \ref{prop:martsol}, it remains to show that
$$\tilde M=\int_0^\tec\varPhi(\tilde\varrho,\tilde\varrho\tilde\bfu)\,\dif\tilde W.$$
Towards this end, it is enough to pass to the limit in \eqref{exp211}, \eqref{exp311} and establish
\begin{equation}\label{exp211a}
\begin{split}
&\tilde{\stred}\,h\big(\bfr_s\tilde\varrho, \bfr_s\tilde{\bu},\bfr_s\tilde{W}\big)\bigg[\big[\langle\tilde M,\bfvarphi\rangle^2\big]_{s,t}-\sum_{k\geq 1}\int_s^t\big\langle g_k(\tilde\varrho,\tilde\varrho\tilde\bfu),\bfvarphi\big\rangle^2\,\dif r\bigg]=0,
\end{split}
\end{equation}
\begin{equation}\label{exp311a}
\begin{split}
&\tilde{\stred}\,h\big(\bfr_s\tilde\varrho, \bfr_s\tilde{\bu},\bfr_s\tilde{W}\big)\bigg[\big[\langle\tilde M,\bfvarphi\rangle\tilde{\beta}_k\big]_{s,t}-\int_s^t\big\langle g_k(\tilde\varrho,\tilde\varrho\tilde\bfu),\bfvarphi\big\rangle\,\dif r\bigg]=0.
\end{split}
\end{equation}
The convergence in the terms that involve $M_{\varepsilon}(\tilde\varrho_\varepsilon,\tilde\bu_\varepsilon,\tilde\varrho_\varepsilon\tilde\bfu_\varepsilon)$ follows from a similar reasoning as in Proposition \ref{prop:martsol} together with the fact that, due to our estimates, $M_{\varepsilon}(\tilde\varrho_\varepsilon,\tilde\bu_\varepsilon,\tilde\varrho_\varepsilon\tilde\bfu_\varepsilon)$ possesses moments of any order (uniformly in $\varepsilon$). The convergence in terms coming from the stochastic integral can be justified similarly to Proposition \ref{prop:limit1} and therefore we omit the details.
Again the energy inequality follows from lower semi-continuity.
\end{proof}

\section{The limit in the artificial pressure}

In this final section we let $\delta\rightarrow0$ in the approximate system \eqref{eq:approximation2} and complete the proof of Theorem \ref{thm:main}.

As the first step, let us construct initial laws $\Lambda^\delta$ that satisfy the assumptions of Section \ref{sec:galerkin} and that approximate the given law $\Lambda$ in a suitable sense. To this end, let $(\rho,\bfq)$ be random variables having the law $\Lambda$ defined on some probability space $(\Omega,\mf,\p)$. Then one can find random variables $\rho_\delta$ with values in $C^{2+\kappa}(\mt)$, for some $\kappa>0$, such that
$$0<\delta\leq \rho_\delta\leq\delta^{-\frac{1}{2\beta}},\quad (\rho_\delta)_{\mt}\leq 2M\quad\text{a.s.},\quad\rho_\delta\to\rho\quad\text{in}\quad L^p(\Omega;L^\gamma(\mt)) \quad\forall p\in[1,\infty).$$
Next, setting
$$\tilde\bfq_\delta=\begin{cases}
				\bfq\sqrt{\frac{\rho_\delta}{\rho}},&\text{if }\rho>0,\\
				0,&\text{if }\rho=0,
				\end{cases}
$$
it follows from the assumptions on $\Lambda$ that
$$\frac{|\tilde\bfq_\delta|^2}{\rho_\delta}\in L^p(\Omega;L^1(\mt)) \quad\forall p\in[1,\infty)$$
uniformly in $\delta,$ and we can find random variables $h_\delta$ with values in $C^2(\mt)$ such that
$$\frac{\tilde\bfq_\delta}{\sqrt{\rho_\delta}}-h_\delta\to 0\quad\text{in}\quad L^p(\Omega;L^2(\mt))\quad\forall p\in[1,\infty).$$
Let $\bfq_\delta=h_\delta\sqrt{\rho_\delta}$. Then
$$\frac{|\bfq_\delta|^2}{\rho_\delta}\in L^p(\Omega;L^1(\mt)) \quad\forall p\in[1,\infty)$$
uniformly in $\delta$
and
\begin{align*}
\bfq_\delta&\rightarrow \bfq\quad\text{in}\quad L^p(\Omega; L^1(\mt))\quad\forall p\in[1,\infty)\\
\frac{\bfq_\delta}{\sqrt{\rho_\delta}}&\rightarrow \frac{\bfq}{\sqrt{\rho}}\quad\text{in}\quad L^p(\Omega; L^2(\mt))\quad\forall p\in[1,\infty).
\end{align*}
Finally, we define $\Lambda^\delta=\p\circ (\rho_\delta,\bfq_\delta)^{-1}$.

As discussed at the beginning of Section \ref{sec:vanishingviscosity}, without any loss of generality one can suppose that for every $\delta\in(0,1)$ there exists
$$\big((\Omega,\mf,(\mf^\delta_t),\prst),\varrho_\delta,\bfu_\delta,W\big)$$
which is a finite energy weak martingale solution to \eqref{eq:approximation2} with the initial law $\Lambda^\delta$.
Moreover, due to the construction of the approximate initial laws $\Lambda^\delta$, the term appearing on the right hand side of the corresponding energy inequality (cf. \eqref{eq:apriorivarepsilon}) satisfies
$$\int_{C^{2+\kappa}_x\times C^2_x}\bigg\|\frac{1}{2}\frac{|\bfq|^2}{\rho}+\frac{a}{\gamma-1}\rho^\gamma+\frac{\delta}{\beta-1}\rho^\beta\bigg\|_{L^1_x}^p\,\dif\Lambda^\delta(\rho,\bfq)\rightarrow \int_{L^\gamma_x\times L^{\frac{2\gamma}{\gamma+1}}_x}\bigg\|\frac{1}{2}\frac{|\bfq|^2}{\rho}+\frac{a}{\gamma-1}\rho^\gamma\bigg\|_{L^1_x}^p\,\dif\Lambda(\rho,\bfq) $$
for all $p\in[1,\infty)$ and, in addition,
$\Lambda^\delta\overset{*}{\rightharpoonup}\Lambda$ weakly in the sense of measures on $L^\gamma(\mt)\times L^{q}(\mt)$ where $q<\frac{2\gamma}{\gamma+1}$.
Furthermore, we obtain the following uniform bounds
\begin{align}
 \bfu_\delta&\in L^{p}(\Omega;L^2(0,T;W^{1,2}( \mathbb T^3))),\label{apvdelta}\\
\sqrt{ \varrho_\delta} \bfu_\delta&\in L^{p}(\Omega;L^\infty(0,T;L^2( \mathbb T^3))),\label{aprhovdelta}\\
 \varrho_\delta&\in L^{p}(\Omega;L^\infty(0,T;L^\gamma( \mathbb T^3))),\label{aprhodelta}\\
 \delta \varrho_\delta^\beta&\in L^p(\Omega;L^\infty(0,T;L^1( \mathbb T^3))),\label{aprho2delta}\\
  \varrho_\delta\bfu_\delta&\in L^{p}(\Omega;L^\infty(0,T;L^\frac{2\gamma}{\gamma+1}( \mathbb T^3))),\label{estrhou2delta}\\
\varrho_\delta\bfu_\delta\otimes\bfu_\delta&\in L^p(\Omega;L^2(0,T;L^\frac{6\gamma}{4\gamma+3}(\mt))).\label{est:rhobfu22delta}
\end{align}

Let us now improve integrability of the density.

\begin{proposition}
There holds for all $\Theta\leq \frac{2}{3}\gamma-1$
\begin{align}\label{eq:gamma+theta}
\E\int_0^T\int_{\mt}\big(a\varrho_\delta^{\gamma+\Theta}+\delta\varrho_\delta^{\beta+\Theta}\big)\dxt\leq C.
\end{align}

\begin{proof}
In the deterministic case one has to test with 
$$\nabla\Delta^{-1}\big( \varrho^\Theta-( \varrho^\Theta)_ {\mathbb T^3}\big)=\Delta^{-1}\nabla \varrho^\Theta,$$
where $\Theta>0$. In order to do this rigorously we have to replace the map $z\mapsto z^\Theta$ by some function $b\in C^{1}(\R)$ with compact support in order to use the renormalized continuity equation.
So we apply It\^{o}'s formula to the functional $f(\bfq,g)=\int_ {\mathbb T^3} \bfq\cdot\Delta^{-1}\nabla g\dx$. Note that $f$ is linear in $\bfq= \varrho \bfu$
and the quadratic variation of $g= b(\varrho)$ is zero. Hence we do not need a correction term. We gain
\begin{align*}
\E J_0&=\E\int_ {\mathbb T^3}  \varrho_\delta \bfu_\delta\cdot \Delta^{-1}\nabla b(\varrho_\delta)\dx\\&=\E\int_ {\mathbb T^3}  \varrho_\delta \bfu_\delta(0)\cdot \Delta^{-1}\nabla b(\varrho_\delta(0))\dx+\nu\E\int_0^t\int_ {\mathbb T^3} \nabla \bfu_\delta:\nabla\Delta^{-1}\nabla b(\varrho_\delta)\dx\\
&\qquad+(\lambda+\nu)\E\int_0^t\int_ {\mathbb T^3} \Div \bfu_\delta\,b(\varrho_\delta)\dx+\E\int_0^t\int_ {\mathbb T^3}  \varrho \bfu_\delta\otimes \bfu_\delta:\nabla\Delta^{-1}\nabla b(\varrho_\delta)\dxs\\
&\qquad+\E\int_0^t\int_ {\mathbb T^3} \big(\varrho_\delta^{\gamma}+\delta\varrho_\delta^{\beta}\big)b(\varrho_\delta)\dxs
-\E\int_0^t(b(\varrho_\delta))_{\mt}\int_ {\mathbb T^3} \big(\varrho_\delta^{\gamma}+\delta\varrho_\delta^{\beta}\big)\dxs\\
&\qquad+\E\int_0^t\int_ {\mathbb T^3}\Delta^{-1}\Div( \varrho_\delta\bfu_\delta)\,\dd( b(\varrho_\delta))\dx
=\E J_1+\cdots+\E J_7.
\end{align*}
This can be justified as done in \eqref{eq:hjhj}. For $J_7$
we use the renormalized continuity equation which reads as
\begin{align*}
\partial_t b(\varrho_\delta)+\Div\big(b(\varrho_\delta)\bfu_\delta\big)+\big(b'(\varrho_\delta)\varrho_\delta-b(\varrho_\delta)\big)\Div\bfu_\delta=0
\end{align*}
such that
\begin{align*}
\int_ {\mathbb T^3}  b(\varrho_\delta(t))\,\varphi\dx&=\int_0^t\int_ {\mathbb T^3} b(\varrho_\delta) \bfu_\delta\cdot\nabla\varphi\dxs-\int_0^t\int_ {\mathbb T^3} \big(b'(\varrho_\delta)\varrho_\delta-b(\varrho_\delta)\big) \Div \bfu_\delta\,\varphi\dxs
\end{align*}
and 
\begin{align*}
J_7&=\int_0^t\int_ {\mathbb T^3}\Delta^{-1}\Div( \varrho_\delta\bfu_\delta)\,\dd( b(\varrho_\delta))\dx\\
&=\int_0^t\int_ {\mathbb T^3} b(\varrho_\delta) \bfu_\delta\cdot\nabla\Delta^{-1}\Div( \varrho_\delta\bfu_\delta)\dxs\\
&\qquad-\int_0^t\int_ {\mathbb T^3} \big(b'(\varrho_\delta)\varrho_\delta-b(\varrho_\delta)\big) \Div \bfu\,\Delta^{-1}\Div( \varrho_\delta\bfu_\delta)\dxs\\
&=J_7^1+J_7^2
\end{align*}
Now we use a sequence of compactly supported smooth functions $b_m$ to approximate $z\mapsto z^\Theta$ and gain
\begin{align*}
\E J_0&=\E\int_ {\mathbb T^3}  \varrho_\delta \bfu_\delta\cdot \Delta^{-1}\nabla\varrho^\Theta_\delta\dx\\&=\E\int_ {\mathbb T^3}  \varrho_\delta \bfu_\delta(0)\cdot \Delta^{-1}\nabla\varrho_\delta^\Theta(0)\dx
+\nu\E\int_0^t\int_ {\mathbb T^3} \nabla \bfu_\delta:\nabla\Delta^{-1}\nabla\varrho^\Theta_\delta\dx\\
&\;+(\lambda+\nu)\E\int_0^t\int_ {\mathbb T^3} \Div \bfu_\delta\,\varrho^\Theta_\delta\dx
+\E\int_0^t\int_ {\mathbb T^3}  \varrho \bfu_\delta\otimes \bfu_\delta:\nabla\Delta^{-1}\nabla\varrho_\delta^\Theta\dxs\\
&\;+\E\int_0^t\int_ {\mathbb T^3} \big(\varrho_\delta^{\gamma+\Theta}+\delta\varrho_\delta^{\beta+\Theta}\big)\dxs-\E\int_0^t(\varrho_\delta^\Theta)_{\mt}\int_ {\mathbb T^3} \big(\varrho_\delta^{\gamma}+\delta\varrho_\delta^{\beta}\big)\dxs\\
&\;+\E\int_0^t\int_ {\mathbb T^3} \varrho_\delta^\Theta \bfu\cdot\nabla\Delta^{-1}\Div( \varrho_\delta\bfu_\delta)\dxs+(1-\Theta)\E\int_0^t\int_ {\mathbb T^3} \Delta^{-1}\nabla(\varrho_\delta^\Theta \Div \bfu_\delta) \varrho_\delta\bfu_\delta\dxs\\
&=\E J_1+\cdots+\E J_6+\E J_7^1+\E J_7^2.
\end{align*}
We want to bound the term $J_5$, so we have to estimate all the others.
We have $$(\varrho_\delta^\Theta)_ {\mathbb T^3}\leq (1+\varrho_\delta)_ {\mathbb T^3}=(1+\varrho_\delta(0))_ {\mathbb T^3}\leq C$$ provided $\Theta\leq1$. So (\ref{aprho2delta}) yields $\E J_6\leq C$. 
The most critical term is $J_4$ which we estimate by
\begin{align*}
\E J_4&\leq\,\E\int_0^t\| \varrho_\delta\|_\gamma\| \bfu_\delta\|_6^2\| \varrho^{\Theta}_\delta\|_r\dt,
\end{align*}
where $r:=\frac{3\gamma}{2\gamma-3}$. We proceed, using continuity of $\nabla\Delta^{-1}\nabla$, by
\begin{align*}
\E J_4&\leq\,C\,\E\Big(\sup_{0\leq s\leq t}\| \varrho_\delta\|_\gamma\Big)\Big(\sup_{0\leq s\leq t}\| \varrho_\delta^{\Theta}\|_r\Big)\int_0^t\|\nabla \bfu_\delta\|_2^2+\| \bfu_\delta\|_2^2\ds\\
&\leq\,C\,\bigg(\E\sup_{0\leq s\leq t}\| \varrho_\delta\|_\gamma^{q_1}\bigg)^{\frac{1}{q_1}}\bigg(\E\sup_{0\leq s\leq t}\| \varrho_\delta^{\Theta}\|_r^{q_2}\bigg)^{\frac{1}{q_2}}\bigg(\E\bigg[\int_0^t\|\nabla \bfu_\delta\|_2^2+\| \bfu_\delta\|_2^2\ds\bigg]^{q_3}\bigg)^{\frac{1}{q_3}}
\end{align*}
as a consequence of H\"older's inequality ($\frac{1}{q_1}+\frac{1}{q_2}+\frac{1}{q_2}=1$, for instance $q_1=q_2=q_3=3$). We need to choose $r$ such that $\Theta r\leq \gamma$ which is equivalent to $\Theta\leq \frac{2}{3}\gamma-1$. Then we conclude from (\ref{apvdelta}) and (\ref{aprhodelta}) that $\E J_4\leq C$. In order to estimate $J_0$ we use the following estimate which follows from the continuity of $\nabla\Delta^{-1}\nabla$ and Sobolev's Theorem for  $q=\frac{6\gamma}{5\gamma-3}\in(1,3)$
\begin{align*}
\|\Delta^{-1}\nabla \varrho_\delta^\Theta\|_{L^{\frac{3q}{3-q}}( \mathbb T^3)}&\leq \,C\,\|\nabla\Delta^{-1}\nabla \varrho_\delta^\Theta\|_{L^{q}( \mathbb T^3)}\leq\,C\,\| \varrho_\delta^\Theta\|_{L^{q}( \mathbb T^3)}.
\end{align*}
We gain $|\E J_0|\leq C$ as a consequence of \eqref{aprhodelta} and \eqref{estrhou2delta} by choosing $\Theta\leq \frac{5}{6}\gamma-\frac{1}{2}$.
We have due to the continuity of $\nabla\Delta^{-1}\nabla$
\begin{align*}
\E J_2&\leq \,\E\bigg[\int_0^t\int_ {\mathbb T^3} |\nabla \bfu_\delta|^2\dxs\bigg]+\E\bigg[\int_0^t\int_ {\mathbb T^3} | \varrho_\delta|^{2\Theta}\dxs\bigg]\leq C
\end{align*}
provided $\Theta\leq\gamma/2$. Similarly for $J_3$.
Choosing $p=\frac{6\gamma}{5\gamma-6}$ and $q=\frac{6\gamma}{7\gamma-6}$ there holds
\begin{align*}
\E J_7^2&\leq\,\E\int_0^t\| \varrho_\delta\|_\gamma\| \bfu_\delta\|_6\| \Delta^{-1}\nabla(\varrho^{\Theta}_\delta\Div\bfu_\delta)\|_p\ds\\
&\leq\,\E\int_0^t\| \varrho_\delta\|_\gamma\| \bfu_\delta\|_6\| \varrho^{\Theta}_\delta\Div\bfu_\delta\|_q\ds\\
&\leq\,C\,\E\bigg[\Big(\sup_{0\leq s\leq t}\| \varrho_\delta\|_\gamma\Big)\Big(\int_0^t\|\nabla \bfu_\delta\|_2^2+\| \bfu_\delta\|_2^2\ds\Big)^{\frac{1}{2}}\Big(\int_0^t\|\varrho_\delta^\Theta\Div\bfu_\delta\|_q^2\ds\Big)^\frac{1}{2}\bigg]\\
&\leq\,C\,\bigg(\E\sup_{0\leq s\leq t}\| \varrho_\delta\|_\gamma^{q_1}\bigg)^{\frac{1}{q_1}}\bigg(\E\bigg[\int_0^t\|\nabla \bfu_\delta\|_2^2+\| \bfu_\delta\|_2^2\ds\bigg]^{\frac{q_2}{2}}\bigg)^{\frac{1}{q_2}}\bigg(\E\bigg[\int_0^t\|\varrho_\delta^\Theta\Div\bfu_\delta\|_q^2\ds\bigg]^{\frac{q_3}{2}}\bigg)^{\frac{1}{q_3}}.
\end{align*}
The first two terms are uniformly bounded on account of (\ref{apvdelta}) and (\ref{aprhodelta}). For the third one we estimate (note that $q<2$ as $\gamma>\frac{3}{2}$)
\begin{align*}
\E\bigg[\int_0^t&\|\varrho_\delta^\Theta\Div\bfu_\delta\|_q^2\ds\bigg]\\&\leq \E\bigg[\int_0^t\Big(\int_ {\mathbb T^3}|\Div\bfu_\delta|^2\dx\Big)\Big(\int_ {\mathbb T^3}\varrho_\delta^{\Theta\frac{2q}{2-q}}\dx\Big)^{\frac{2-q}{q}}\ds\bigg]\\
&\leq \E\bigg[\Big(\sup_{0\leq s\leq t}\int_ {\mathbb T^3}\varrho_\delta^{\Theta\frac{2q}{2-q}}\dx\Big)^{\frac{2-q}{q}}\int_0^t\int_ {\mathbb T^3}|\Div\bfu_\delta|^2\dxs\bigg]\\
&\leq \bigg(\E\bigg[\Big(\sup_{0\leq s\leq t}\int_ {\mathbb T^3}\varrho_\delta^{\Theta\frac{2q}{2-q}}\dx\Big)^{\frac{2-q}{q}}\bigg]^{q_1}\bigg)^{\frac{1}{q_1}}\bigg(\E\bigg[\int_0^t\int_ {\mathbb T^3}|\Div\bfu_\delta|^2\dxs\bigg]^{q_2}\bigg)^{\frac{1}{q_2}}.
\end{align*}
By (\ref{apvdelta}) and (\ref{aprhodelta}) it is bounded provided $\Theta\frac{2q}{2-q}\leq\gamma$
which is aquivalent to $\Theta\leq\frac{2}{3}\gamma-1$. Hence $E[J_7^2]$ is uniformly bounded.
Moreover, we have as $p< 6$ (due to $\gamma>\frac{3}{2}$)
\begin{align*}
\E[|J_7^1|]&\leq\,C\,\E\bigg[\int_0^t\| \varrho_\delta\|_\gamma\| \bfu_\delta\|_6\| \Delta^{-1}\nabla(\Div(\varrho^{\Theta}_\delta\bfu_\delta))\|_p\dt\bigg]\\
&\leq\,C\,\E\bigg[\int_0^t\| \varrho_\delta\|_\gamma\| \bfu_\delta\|_6\| \varrho^{\Theta}_\delta\bfu_\delta\|_p\dt\bigg]\\
&\leq\,C\,\E\bigg[\int_0^t\| \varrho_\delta\|_\gamma\| \bfu_\delta\|^2_6\| \varrho^{\Theta}_\delta\|_r\dt\bigg],
\end{align*}
where $r=\frac{3\gamma}{2\gamma-3}$. We proceed by
\begin{align*}
\E[|J_7^1|]&\leq\,C\,\E\bigg[\Big(\sup_{0\leq s\leq t}\| \varrho_\delta\|_\gamma\Big)\Big(\sup_{0\leq s\leq t}\| \varrho_\delta^{\Theta}\|_r\Big)\int_0^t\|\bfu_\delta\|_2^2+\|\nabla \bfu_\delta\|_2^2\ds\bigg]\\
&\leq\,C\,\bigg(\E\sup_{0\leq s\leq t}\| \varrho_\delta\|_\gamma^{q_1}\bigg)^{\frac{1}{q_1}}\bigg(\E\sup_{0\leq s\leq t}\| \varrho_\delta^{\Theta}\|_r^{q_2}\bigg)^{\frac{1}{q_2}}\bigg(\E\bigg[\int_0^t\|\bfu_\delta\|_2^2+\|\nabla \bfu_\delta\|_2^2\ds\bigg]^{q_3}\bigg)^{\frac{1}{q_3}}\leq C,
\end{align*}
using again \eqref{apvdelta} and \eqref{aprhodelta}.
Finally, we can conclude for all $\Theta\leq\frac{2}{3}\gamma-1$ the claimed estimate.
\end{proof}
\end{proposition}


Now we can perform the compactness argument similarly to Subsection \ref{subsec:compactness}. More precisely, we set $\mathcal{X}=\mathcal{X}_\varrho\times\mathcal{X}_\bu\times\mathcal{X}_{\varrho\bu}\times\mathcal{X}_W$ where
\begin{align*}
\mathcal{X}_\varrho&=C_w([0,T];L^\gamma(\mt))\cap \big(L^{\gamma+\Theta}(Q),w\big),&\mathcal{X}_\bu&=\big(L^2(0,T;W^{1,2}(\mt)),w\big),\\
\mathcal{X}_{\varrho\bu}&=C_w([0,T];L^\frac{2\gamma}{\gamma+1}(\mt)),&\mathcal{X}_W&=C([0,T];\mathfrak{U}_0)
\end{align*}
and remark that the only change lies in the proof of tightness for $\{\mu_{\varrho_\delta\bfu_\delta};\,\delta\in(0,1)\}$.

\begin{proposition}
The set $\{\mu_{\varrho_\delta\bu_\delta};\,\delta\in(0,1)\}$ is tight on $\mathcal{X}_{\varrho\bu}$.

\begin{proof}
We proceed similarly as in Proposition \ref{prop:rhoutight} and Proposition \ref{rhoutight1} and decompose $\varrho_\delta\bu_\delta$ into two parts, namely, $\varrho_\delta\bu_\delta(t)=Y^\delta(t)+Z^\delta(t)$, where
\begin{equation*}
 \begin{split}
Y^\delta(t)&=\bfq(0)-\int_0^t\big[\diver(\varrho_\delta\bu_\delta\otimes\bu_\delta)+\nu\Delta\bu_\delta+(\lambda+\nu)\nabla\diver\bfu_\delta-a\nabla \varrho_\delta^\gamma\big]\dif s\\
&\hspace{3.8cm}+\int_0^t\,\varPhi(\varrho_\delta,\varrho_\delta\bu_\delta) \,\dif W(s),\\
Z^\delta(t)&=-\delta\int_0^t\nabla \varrho_\delta^\beta\,\dif s.
 \end{split}
\end{equation*}
By the approach of Proposition \ref{rhoutight1} (where we employ \eqref{eq:gamma+theta} instead of \eqref{eq:gamma+1}), we obtain H\"older continuity of $Y^\delta$, namely, there exist $\vartheta>0$ and $m>3/2$ such that
\begin{equation*}
\stred\big\|Y^\delta\|_{C^\vartheta([0,T];W^{-m,2}(\mt))}\leq C.
\end{equation*}

Next, we show that the set of laws $\{\prst\circ[Z^\delta]^{-1};\,\delta\in(0,1)\}$ is tight on $C([0,T];W^{-1,\frac{\beta+\Theta}{\beta}}(\mt))$ and the conclusion follows by the lines of Proposition \ref{rhoutight1}.
There holds due to \eqref{eq:gamma+theta} that (up to a subsequence)
$$\delta\varrho_\delta^\beta\rightarrow 0\quad\text{ in }\quad L^{\frac{\beta+\Theta}{\beta}}(Q)\quad\text{a.s.}$$
hence
$$\delta\nabla\varrho_\delta^\beta\rightarrow 0\quad\text{ in }\quad L^{\frac{\beta+\Theta}{\beta}}(0,T;W^{-1,\frac{\beta+\Theta}{\beta}}(\mt))\quad\text{a.s.}$$
and
$$Z^\delta\rightarrow0\quad\text{ in }\quad C([0,T];W^{-1,\frac{\beta+\Theta}{\beta}}(\mt))\quad\text{a.s.}$$
This leads to the convergence in law
$$Z^\delta\overset{d\hspace{2.5mm}}{\rightarrow0}\quad\text{ on }\quad C([0,T];W^{-1,\frac{\beta+\Theta}{\beta}}(\mt))$$
and the claim follows.
\end{proof}
\end{proposition}

We apply the Jakubowski-Skorokhod representation theorem and mimicking the technique of Subsection \ref{subsec:compactness}. We obtain the existence of a probability space $(\tilde\Omega,\tilde\mf,\tilde\prst)$ and $\mathcal{X}$-valued random variables $(\tilde\varrho_\delta,\tilde\bu_\delta,\tilde W_\delta)$, $\delta\in(0,1)$, and $(\tilde\varrho,\tilde\bu,\tilde W)$ together with their $\tilde\prst$-augmented canonical filtrations $(\tilde\mf_t^\delta)$ and $(\tilde\mf_t)$, respectively, such that the corresponding counterparts of Lemma \ref{lemma:identif2} and Proposition \ref{prop:martsol} are valid. Let us summarize the result in the following proposition.

\begin{proposition}\label{prop:neu}
The following convergences hold true $\tilde\prst$-a.s.
\begin{align*}
\tilde\bu_\delta&\rightharpoonup  \tilde{\bfu}\qquad\text{in}\qquad L^2([0,T];W^{1,2}( \mathbb T^3)),\\
\tilde{\varrho}_\delta&\rightarrow   \tilde{\varrho}\qquad\text{in}\qquad C_w(0,T;L^{\gamma}( \mathbb T^3)),\label{conv:rho2}\\
\tilde\varrho_\delta\tilde\bu_\delta&\rightarrow  \tilde{\varrho}  \tilde{\bfu}\qquad\text{in}\qquad C_w([0,T];L^{\frac{2\gamma}{\gamma+1}}( \mathbb T^3)),\\
\tilde{\varrho}_\delta  \tilde{\bfu}_\delta\otimes  \tilde{\bfu}_\delta&\rightharpoonup  \tilde{\varrho}  \tilde{\bfu}\otimes  \tilde{\bfu}\qquad\text{in}\qquad L^1(0,T;L^{1}( \mathbb T^3)),
\end{align*}
Furthermore, for every $\delta\in(0,1)$, $\big((\tilde{\Omega},\tilde{\mf},(\tilde{\mf}^\delta_t),\tilde{\prst}),\tilde\varrho_\delta,\tilde{\bu}_\delta,\tilde{W}_\delta\big)$ is a weak martingale solution to \eqref{eq:approximation2} with the initial law $\Lambda^\delta$ and there exists $b> \frac{3}{2}$ together with a $W^{-b,2}(\mt)$-valued continuous square integrable $(\tilde\mf_t)$-martingale $\tilde M$ and
$$\tilde p\in L^\frac{\gamma+\Theta}{\gamma}(\tilde\Omega\times Q)$$
such that 
$\big((\tilde{\Omega},\tilde{\mf},(\tilde{\mf}_t),\tilde{\prst}),\tilde\varrho,\tilde{\bu},\tilde p,\tilde{M}\big)$ is a weak martingale solution to
\begin{subequations}\label{eq:limiteq}
 \begin{align}
  \dif \tilde\varrho+\diver(\tilde\varrho\tilde\bu)\dif t&=0,\\
  \dif(\tilde\varrho\tilde\bu)+\big[\diver(\tilde\varrho\tilde\bu\otimes\tilde\bu)-\nu\Delta\tilde\bu-(\lambda+\nu)\nabla\diver\tilde\bfu+\nabla\tilde p\, \big]\dif t&=\dif\tilde M
 \end{align}
\end{subequations}
with the initial law $\Lambda$.

\begin{proof}
Let us only make a short remark concerning the pressure: $a\tilde\varrho_\delta^\gamma$ converges to $\tilde p$ in $L^\frac{\gamma+\Theta}{\gamma}(\tilde\Omega\times Q)$ whereas the artificial pressure $\delta\tilde\varrho_\delta^\beta$ vanishes as $\delta\rightarrow 0$.
\end{proof}
\end{proposition}

Let us proceed with an application of the fundamental theorem on Young measures that will be used several times in what follows. The result is taken from \cite[Theorem 4.2.1, Corollary 4.2.10]{malek} and modified to our setting.

\begin{corollary}\label{cor:young}
Let $z_n:\mt\rightarrow\mr$ be a sequence of functions weakly converging in $L^p(\mt)$ for some $p\in[1,\infty)$. Then there exists a Young measure $\nu$ such that for every $H\in C(\mr)$ satisfying for some $q>0$ the growth condition
$$|H(\xi)|\leq C(1+|\xi|^q)\qquad\forall \xi\in \mr$$
it holds that
$$H(z_n)\rightharpoonup \bar H\quad\text{in}\quad L^r(\mt)\qquad\text{where}\quad
\bar H(x)=\langle\nu_x, H\rangle,$$
provided
$$1<r\leq\frac{p}{q}.$$

\end{corollary}

\subsection{The effective viscous flux}

It remains to show that $\tilde{p}=a\tilde{\varrho}^\gamma$. Here it is not possible to test by $\Delta^{-1}\nabla\rho$ as in Subsection \ref{subsec:strongconvdensity} so we test by $\Delta^{-1}\nabla T_k(\rho)$ instead, where we employ the cut-off functions
$$T_k(z)=k\,T\Big(\frac{z}{k}\Big)\qquad z\in\mr\quad k\in\N,$$
with $T$ being a smooth concave function on $\mr$ such that $T(z)=z$ for $z\leq 1$ and $T(z)=2$ for $z\geq3$. To this end, we can choose $b=T_k$ in the renormalized continuity equation for $\tilde\varrho_\delta$ (cf. proof of Proposition \ref{prop:martsol}) which leads to
\begin{align*}
\partial_t T_k(\tilde{\varrho}_\delta)+\Div\big(T_k(\tilde{\varrho}_\delta)\bfu_\delta\big)+\big(T_k'(\tilde{\varrho}_\delta)\tilde{\varrho}_\delta-T_k(\tilde{\varrho}_\delta)\big)\Div\tilde{\bfu}_\delta=0
\end{align*}
in the sense of distributions. In order to pass to the limit in this equation, let $\tilde{T}^{1,k}$ denote the weak limit of $T_k(\tilde{\varrho}_\delta)$ given by Corollary \ref{cor:young} and let $\tilde{T}^{2,k}$ denote the weak limit of $\big(T_k'(\tilde{\varrho}_\delta)\tilde{\varrho}_\delta-T_k(\tilde{\varrho}_\delta)\big)\Div\tilde{\bfu}_\delta$  in $L^2(\tilde{\Omega}\times Q)$ (here it might be necessary to pass to a subsequence).
To be more precise, there holds
\begin{align}
 T_k(\tilde{\varrho}_\delta)&\rightarrow \tilde{T}^{1,k}\quad\text{in}\quad C_w([0,T];L^p(\mathbb T^3))\quad \tilde{\p}\text{-a.s.}\quad\forall p\in[1,\infty),\label{eq:Tk1}\\
\big(T_k'(\tilde{\varrho}_\delta)\tilde{\varrho}_\delta-T_k(\tilde{\varrho}_\delta)\big)\Div\tilde{\bfu}_\delta&\rightharpoonup\tilde{T}^{2,k}
\quad\text{in}\quad L^2(\tilde{\Omega}\times Q).\label{eq:Tk2}
\end{align}
So letting $\delta\rightarrow0$ yields
\begin{align}\label{eq:Tk}
\partial_t \tilde{T}^{1,k}+\Div\big(\tilde T^{1,k}\tilde{\bfu}\big)+\tilde{T}^{2,k}=0.
\end{align}
Here we used that $\tilde\p$-a.s.
\begin{align*}
 T_k(\tilde{\varrho}_\delta)&\rightarrow \tilde{T}^{1,k}\quad\text{in}\quad L^2([0,T];W^{-1,2}(\mathbb T^3)),\\
\tilde\bfu_\delta&\rightharpoonup\tilde\bfu\quad\text{in}\quad L^2([0,T];W^{-1,2}(\mathbb T^3)),
\end{align*}
which is a consequence of \eqref{eq:Tk1} (with $p>\frac{6}{5}$) and Proposition \ref{prop:neu}.
Next, for the approximate system \eqref{eq:approximation2} we apply It\^{o}'s formula to the function $f(\rho,\bfq)=\int_{\mt}\bfq\cdot\Delta^{-1}\nabla T_k(\rho)\dx$ and gain similarly to Subsection \ref{subsec:strongconvdensity}
\begin{align*}
\begin{aligned}
&\tilde\stred\int_ {\mathbb T^3}  \tilde\varrho_\delta\tilde \bfu_\delta\cdot \Delta^{-1}\nabla T_k(\tilde\varrho_\delta)\dx\\
&=\tilde\stred\int_ {\mathbb T^3}  \tilde\varrho_\delta\tilde \bfu_\delta(0)\cdot \Delta^{-1}\nabla T_k(\tilde\varrho_\delta(0))\dx-\nu\,\tilde\stred\int_0^t\int_ {\mathbb T^3} \nabla \tilde\bfu_\delta:\nabla\Delta^{-1} \nabla T_k(\tilde\varrho_\delta)\dxs\\
&\;-(\lambda+\nu)\,\tilde\stred\int_0^t\int_ {\mathbb T^3} \diver\tilde\bfu_\delta\,T_k(\tilde\varrho_\delta)\dxs+\tilde\stred\int_0^t\int_ {\mathbb T^3}  \tilde\varrho_\delta\tilde \bfu_\delta\otimes \tilde\bfu_\delta:\nabla\Delta^{-1} \nabla T_k(\tilde\varrho_\delta)\dxs\\
&\,+\tilde\stred\int_0^t\int_ {\mathbb T^3}a\tilde \varrho_\delta^\gamma \,T_k(\tilde\varrho_\delta)\dxs-\tilde\stred\int_0^t(T_k(\tilde\varrho_\delta))_{\mt}\int_ {\mathbb T^3}a\tilde \varrho_\delta^\gamma \,\dxs\\
&\;+\tilde\stred\int_0^t\int_ {\mathbb T^3}\delta\tilde \varrho_\delta^\beta \,(T_k(\tilde\varrho_\delta)-(T_k(\tilde\varrho_\delta))_{\mt})\dxs
-\tilde\stred\int_0^t\int_{\mt}\tilde\varrho_\delta\tilde\bfu_\delta\Delta^{-1}\nabla\diver\big(T_k(\tilde\varrho_\delta)\tilde\bfu_\delta)\dxs\\
&\;-\tilde\stred\int_0^t\int_{\mt}\tilde\varrho_\delta\tilde\bfu_\delta\Delta^{-1}\nabla(T'_k(\tilde\varrho_\delta)\tilde\varrho_\delta-T_k(\tilde\varrho_\delta))\diver\tilde\bfu_\delta\dxs\\
&=\tilde\stred J_1+\cdots +\tilde\stred J_{9}.
\end{aligned}
\end{align*}
This can finally be written as
\begin{align*}
\begin{aligned}
\tilde\E\int_Q (  \tilde{\varrho}_\delta^\gamma&-(\lambda+2\nu)\Div  \tilde{\bfu}_\delta)\,  T_k(\tilde{\varrho}_\delta)\dxt=\tilde\E\big[J_0-J_1-J_6-J_7-J_9\big]\\
&+\tilde\E\int_Q\Big(  T_k(\tilde{\varrho}_\delta)\mathcal R_{ij}[  \tilde{\varrho}_\delta  \tilde{\bfu}^j_\delta]-  \tilde{\varrho}_\delta  \tilde{\bfu}_\delta^j\mathcal R_{ij}[  T_k(\tilde{\varrho}_\delta)]\Big)  \tilde{\bfu}^i_\delta\dxt.
\end{aligned}
\end{align*}
Whereas for the limit system \eqref{eq:limiteq}, It\^{o}'s formula leads to
\begin{align*}
\begin{aligned}
&\tilde\stred\int_ {\mathbb T^3}  \tilde\varrho\tilde \bfu\cdot \Delta^{-1}\nabla \tilde T^{1,k}\dx=\tilde\stred\int_ {\mathbb T^3}  \tilde\varrho\tilde \bfu(0)\cdot \Delta^{-1}\nabla \tilde T^{1,k}(0)\dx\\
&\quad-\nu\,\tilde\stred\int_0^t\int_ {\mathbb T^3} \nabla \tilde\bfu:\nabla\Delta^{-1}\nabla \tilde T^{1,k}\dxs-(\lambda+\nu)\,\tilde\stred\int_0^t\int_ {\mathbb T^3} \diver\tilde\bfu\,\tilde T^{1,k}\dxs\\
&\quad+\tilde\stred\int_0^t\int_ {\mathbb T^3}  \tilde\varrho\tilde \bfu\otimes \tilde\bfu:\nabla\Delta^{-1}\nabla \tilde T^{1,k}\dxs+\tilde\stred\int_0^t\int_ {\mathbb T^3}\tilde p \,\tilde T^{1,k}\dxs\\
&\quad-\tilde\stred\int_0^t(\,\tilde T^{1,k})_{\mt}\int_ {\mathbb T^3}\tilde p \dxs-\tilde\stred\int_0^t\int_{\mt}\tilde\varrho\tilde\bfu\cdot\nabla\Delta^{-1}\diver\big(\tilde T^{1,k}\tilde\bfu\big)\dxs\\
&-\tilde\stred\int_0^t\int_{\mt}\tilde\varrho\tilde\bfu\cdot\Delta^{-1}\nabla\tilde T^{2,k}\dxs
=\tilde\stred K_1+\cdots +\tilde\stred K_{8}.
\end{aligned}
\end{align*}
From this we infer
\begin{align*}
\begin{aligned}
\tilde\E\int_Q (  \tilde{\varrho}^\gamma&-(\lambda+2\nu)\Div  \tilde{\bfu})\,  T_k(\tilde{\varrho})\dxt=\tilde\E\big[K_0-K_1-K_6-K_8\big]\\
&+\tilde\E\int_Q\Big(  \tilde T^{1,k}\mathcal R_{ij}[  \tilde{\varrho}  \tilde{\bfu}^j]-  \tilde{\varrho}  \tilde{\bfu}^j\mathcal R_{ij}[  \tilde T^{1,k}]\Big)  \tilde{\bfu}^i\dxt.
\end{aligned}
\end{align*}
The limit procedure is now very similar to the vanishing viscosity limit.
Finally this implies
\begin{align*}
 T_k(\tilde{\varrho}_\delta)\mathcal R[ \tilde{\varrho}_\delta \tilde{\bfu}_\delta]- \tilde{\varrho}_\delta \tilde{\bfu}_\delta\mathcal R[ T_k(\tilde{\varrho}_\delta)]\rightarrow  \tilde{T}^{1,k}\mathcal R[ \tilde{\varrho} \tilde{\bfu}]- \tilde{\varrho} \tilde{\bfu}\mathcal R[ \tilde{T}^ {1,k}]
\end{align*}
in $L^2(\Omega;L^2(0,T;W^{-1,2}( \mathbb T^3)))$ as $\frac{2\gamma}{\gamma+1}>\frac{6}{5}$ (using Proposition \ref{prop:neu} and \eqref{eq:Tk1}). Hence
\begin{align}\label{eq:eq11}
\begin{aligned}
\lim_{\delta\rightarrow0}&\tilde\E\int_Q(  T_k(\tilde{\varrho}_\delta)\mathcal R_{ij}[  \tilde{\varrho}_\delta  \tilde{\bfu}^j_\delta]-  \tilde{\varrho}  \tilde{\bfu}^j\mathcal R_{ij}[  T_k(\tilde{\varrho}_\delta)])  \tilde{\bfu}^i_\delta\dxt\\
&=\tilde\E\int_Q(  \tilde{T}^{1,k}\mathcal R_{ij}[  \tilde{\varrho}  \tilde{\bfu}^j]-  \tilde{\varrho}  \tilde{\bfu}^j\mathcal R_{ij}[  \tilde{T}^{1,k}])  \tilde{\bfu}^i\dxt.
\end{aligned}
\end{align}
In order to pass to the limit in the effective viscous flux we have to study in addition the term $J_8$. As a consequence of (\ref{eq:Tk2}) it suffices to show
\begin{align}\label{eq:DeltaDiv}
\Delta^{-1}\Div\big( \tilde{\varrho}_\delta\tilde{\bfu}_{\delta}\big)\rightarrow\Delta^{-1}\Div\big( \tilde{\varrho}\tilde{\bfu}\big)\quad\text{in}\quad L^2(\tilde{\Omega}\times Q).
\end{align}
Due to the weak convergence of  $\tilde\varrho_\delta\tilde\bu_\delta$ in $L^\frac{2\gamma}{\gamma+1}$ for a.e. $(\omega,t)$ we gain \eqref{eq:DeltaDiv} as a consequence of the compactness of the operator $\nabla^{-1}\Div:L^\frac{2\gamma}{\gamma+1}\rightarrow L^2$ (recall that $\gamma>\frac{3}{2}$) and the uniform integrability from \eqref{estrhou2delta}. So we have $\tilde\E J_8\rightarrow \tilde\E K_7$ for $\delta\rightarrow0$.
Due to (\ref{eq:eq11}) we obtain
\begin{align}\label{eq:flux1}
\lim_{\delta\rightarrow 0}\tilde\E\bigg[\int_Q ( \tilde{\varrho}_\delta^\gamma-\Div \tilde{\bfu}_\delta)\, T_k(\tilde{\varrho}_\delta)\dxt\bigg]&=\tilde\E\bigg[\int_Q ( \tilde{p}-\Div \tilde{\bfu})\, \tilde{T}^{1,k}\dxt\bigg].
\end{align}

\subsection{Renormalized solutions}

In order to proceed we have to show
\begin{align}\label{eq.amplosc}
\limsup_{\delta\rightarrow0}\tilde\E \int_{Q}|T_k(\tilde{\varrho}_\delta)-T_k(\tilde{\varrho})|^{\gamma+1}\dxt\leq C,
\end{align}
where $C$ does not depend on $k$. The proof of \eqref{eq.amplosc} follows exactly the arguments from the deterministic problem
 in \cite[Lemma 4.3]{feireisl1} using (\ref{eq:Tk1}) and (\ref{eq:flux1}). We omit the details.

By a standard smoothing procedure we can follow from (\ref{eq:Tk}) that
\begin{align}\label{eq:Tkren}
\partial_t b(\tilde{T}^{1,k})+\Div\big(b(\tilde{T}^{1,k})\tilde{\bfu}\big)
+\big(b'(\tilde{T}^{1,k})\tilde{T}^{1,k}-b(\tilde{T}^{1,k})\big)\Div\tilde{\bfu}=b'(\tilde{T}^{1,k})\tilde{T}^{2,k}
\end{align}
in the sense of distributions.
We want to pass to the limit $k\rightarrow\infty$. On account of (\ref{eq.amplosc}) we have
for all $p\in(1,\gamma)$
\begin{align*}
\|\tilde{T}^{1,k}-\tilde{\varrho}\|_{L^p(\tilde{\Omega}\times Q)}^p&\leq \liminf_{\delta\rightarrow0}\|T_k(\tilde{\varrho}_\delta)-\tilde{\varrho}_\delta\|_{L^p(\tilde{\Omega}\times Q)}^p\\
&\leq 2^p\liminf_{\delta\rightarrow0}\tilde\E\int_{[|\tilde{\varrho}_\delta|\geq k]}|\tilde{\varrho}_\delta|^p\dxt\\\
&\leq 2^pk^{p-\gamma}\liminf_{\delta\rightarrow0}\tilde\E\int_{Q}|\tilde{\varrho}_\delta|^\gamma\dxt\longrightarrow 0,\quad k\rightarrow\infty.
\end{align*}
So we have
\begin{align}
\tilde{T}^{1,k}\rightarrow\tilde{\varrho}\quad\text{in}\quad L^p(\tilde{\Omega}\times Q).\label{eq:T1klim}
\end{align}
In order to pass to the limit in (\ref{eq:Tkren}) we have to show
\begin{align}
b'(\tilde{T}^{1,k})\tilde{T}^{2,k}\rightarrow0\quad\text{in}\quad L^1(\tilde{\Omega}\times Q).\label{eq:T2klim}
\end{align}
Recall that $b$ has to satisfy $b'(z)=0$ for all $z\geq M$ for some $M=M(b)$.
We define
\begin{align*}
Q_{k,M}:=\big\{(\omega,t,x)\in \tilde\Omega\times[0,T]\times\mt;\;\tilde{T}^{1,k}\leq M\big\}
\end{align*}
and gain
\begin{align*}
\begin{aligned}
\tilde\E\int_{Q}&|b'(\tilde{T}^{1,k})\tilde{T}^{2,k}|\dxt\leq \sup_{z\leq M}|b'(z)|\tilde\E\int_Q\chi_{Q_{k,M}}|\tilde{T}^{2,k}|\dxt\\
&\leq \,C\,\liminf_{\delta\rightarrow0}\tilde\E\int_Q\chi_{Q_{k,M}}\big|(T_k'(\tilde{\varrho}_\delta)\tilde{\varrho}_\delta-T_k(\tilde{\varrho}_\delta))\Div\tilde{\bfu}_\delta\big|\dxt\\
&\leq\,C\,\sup_\delta\|\Div\tilde{\bfu}_\delta\|_{L^2(\tilde{\Omega}\times Q)}\liminf_{\delta\rightarrow0}\|T_k'(\tilde{\varrho}_\delta)\tilde{\varrho}_\delta-T_k(\tilde{\varrho}_\delta)\|_{L^2(Q_{k,M})}.
\end{aligned}
\end{align*}
It follows from interpolation that
\begin{align}\label{eq:425}
\begin{aligned}
&\|T_k'(\tilde{\varrho}_\delta)\tilde{\varrho}_\delta-T_k(\tilde{\varrho}_\delta)\|^2_{L^2(Q_{k,M})}\\&\qquad\qquad\leq \|T_k'(\tilde{\varrho}_\delta)\tilde{\varrho}_\delta-T_k(\tilde{\varrho}_\delta)\|_{L^1(\tilde{\Omega}\times Q)}^\alpha\|T_k'(\tilde{\varrho}_\delta)\tilde{\varrho}_\delta-T_k(\tilde{\varrho}_\delta)\|_{L^{\gamma+1}(Q_{k,M})}^{(1-\alpha)(\gamma+1)},
\end{aligned}
\end{align}
where $\alpha=\frac{\gamma-1}{\gamma}$. Moreover, we can show similarly to the proof of
(\ref{eq:T1klim})
\begin{align}\label{eq:426}
\begin{aligned}
\|T_k'(\tilde{\varrho}_\delta)\tilde{\varrho}_\delta-T_k(\tilde{\varrho}_\delta)\|_{L^1(\tilde{\Omega}\times Q)}
&\leq \,C\,k^{1-\gamma}\sup_\delta\tilde\E\int_{Q}|\tilde{\varrho}_\delta|^\gamma\dxt\\
&\longrightarrow 0,\quad k\rightarrow\infty.
\end{aligned}
\end{align}
So it is enough to prove
\begin{align}\label{eq:427}
\sup_\delta\|T_k'(\tilde{\varrho}_\delta)\tilde{\varrho}_\delta-T_k(\tilde{\varrho}_\delta)\|_{L^{\gamma+1}(Q_{k,M})}\leq C,
\end{align}
independently of $k$. As $T_k'(z)z\leq T_k(z)$ there holds by the definition of $Q_{k,M}$
\begin{align*}
&\|T_k'(\tilde{\varrho}_\delta)\tilde{\varrho}_\delta-T_k(\tilde{\varrho}_\delta)\|_{L^{\gamma+1}(Q_{k,M})}\\&\leq \,2\Big(\|T_k(\tilde{\varrho}_\delta)-T_k(\tilde{\varrho})\|_{L^{\gamma+1}(\tilde{\Omega}\times Q)}+\|T_k(\tilde{\varrho}_\delta)\|_{L^{\gamma+1}(Q_{k,M})}\Big)\\
&\leq \,2\Big(\|T_k(\tilde{\varrho}_\delta)-T_k(\tilde{\varrho})\|_{L^{\gamma+1}(\tilde{\Omega}\times Q)}+\|T_k(\tilde{\varrho}_\delta)-\tilde{T}^{1,k}\|_{L^{\gamma+1}(\tilde{\Omega}\times Q)}+\|\tilde{T}^{1,k}\|_{L^{\gamma+1}(Q_{k,M})}\Big).\\
&\leq \,2\Big(\|T_k(\tilde{\varrho}_\delta)-T_k(\tilde{\varrho})\|_{L^{\gamma+1}(\tilde{\Omega}\times Q)}+\|T_k(\tilde{\varrho}_\delta)-\tilde{T}^{1,k}\|_{L^{\gamma+1}(\tilde{\Omega}\times Q)}\Big)+CM.
\end{align*}
Now (\ref{eq.amplosc}) and \eqref{eq:Tk1} imply (\ref{eq:427}). On the other hand (\ref{eq:425})-(\ref{eq:427}) imply (\ref{eq:T2klim}). So we can pass to the limit in (\ref{eq:Tkren}) and gain
\begin{align}\label{eq:ren}
\partial_t b(\tilde{\varrho})+\Div\big(b(\varrho)\tilde{\bfu}\big)
+\big(b'(\tilde{\varrho})\tilde{\varrho}-b(\tilde{\varrho})\big)\Div\tilde{\bfu}=0
\end{align}
in the sense of distributions. 

\subsection{Strong convergence of the density}

We introduce the functions $L_k$ by
\begin{align*}
L_k(z)=\begin{cases}z\ln z,&0\leq z< k\\
z\ln k+z\,\int_k^z T_k(s)/s^2\,ds,&z\geq k
\end{cases}
\end{align*}
We can choose $b=L_k$ in (\ref{eq:ren}) such that
\begin{align*}
\partial_t L_k(\tilde{\varrho})+\Div\big(L_k(\tilde{\varrho})\tilde{\bfu}\big)
+T_k(\tilde{\varrho})\Div\tilde{\bfu}=0.
\end{align*}
We also have that
\begin{align*}
\partial_t L_k(\tilde{\varrho}_\delta)+\Div\big(L_k(\tilde{\varrho}_\delta)\tilde{\bfu}_\delta\big)
+T_k(\tilde{\varrho}_\delta)\Div\tilde{\bfu}_\delta=0.
\end{align*}
The difference of both equations reads as 
\begin{align*}
\int_{\mathbb T^3}\big(L_k(\tilde{\varrho}_\delta)(t)-L_k(\tilde{\varrho})(t)\big)\,\varphi\dx&=\int_{\mathbb T^3}\big(L_k(\tilde{\varrho}_\delta)(0)-L_k(\tilde{\varrho})(0)\big)\,\varphi\dx\\
&=\int_0^t\int_{\mathbb T^3}\big(L_k(\tilde{\varrho}_\delta)\tilde{\bfu}_\delta-L_k(\tilde{\varrho})\tilde{\bfu}\big)\cdot\nabla\varphi\dxs\\
&+\int_0^t\int_{\mathbb T^3}\big(T_k(\tilde{\varrho})\Div\tilde{\bfu}-T_k(\tilde{\varrho}_\delta)\Div\tilde{\bfu}_\delta\big)\,\varphi\dxs
\end{align*}
for all $\varphi\in C^\infty(\mathbb T^3)$.
We have the following convergences $\tilde{\p}$-a.s. for all $p\in(1,\gamma)$ 
\begin{align*}
L_k(\tilde{\varrho}_\delta)\rightarrow\tilde{L}^{1,k}\quad\text{in}\quad C_w([0,T];L^p(\mathbb T^3)),\quad\delta\rightarrow0,\\
\tilde{\varrho}_\delta\ln(\tilde{\varrho}_\delta)\rightarrow\tilde{L}^{2,k}\quad\text{in}\quad C_w([0,T];L^p(\mathbb T^3)),\quad\delta\rightarrow0.
\end{align*}
which is a consequence of Corollary \ref{cor:young} and the $\tilde{\p}$-a.s. convergence of $\tilde{\varrho}_\delta$ in $C_w([0,T];L^\gamma(\mt))$. We also have as $\gamma>\frac{6}{5}$
\begin{align*}
L_k(\tilde{\varrho}_\delta)\rightarrow\tilde{L}^{1,k}\quad\text{in}\quad C([0,T];W^{-1,2}(\mathbb T^3)),\quad\delta\rightarrow0,
\end{align*}
$\tilde{\p}$-a.s.
So we gain using $\Lambda_\delta\rightarrow\Lambda$ (weakly in the sense of measures) for the initial condition
\begin{align*}
\tilde\stred\int_{\mathbb T^3}\big(\tilde{L}^{1,k}(t)-L_k(\tilde{\varrho})(t)\big)\,\varphi\dx&\leq \tilde\stred\int_0^t\int_{\mathbb T^3}\big(\tilde{L}^{1,k}\tilde{\bfu}-L_k(\tilde{\varrho})\tilde{\bfu}\big)\cdot\nabla\varphi\dxs\\
&+\limsup_\delta\tilde\stred\int_0^t\int_{\mathbb T^3}\big(T_k(\tilde{\varrho})\Div\tilde{\bfu}-T_k(\tilde{\varrho}_\delta)\Div\tilde{\bfu}_\delta\big)\,\varphi\dxs.
\end{align*}
This and the choice $\varphi=1$ imply as a consequence of (\ref{eq:eq11})
\begin{align*}
\tilde\E\int_{\mathbb T^3}\big(\tilde{L}^{1,k}(t)-L_k(\tilde{\varrho})(t)\big)\dx&=
\tilde\E\int_0^t\int_{\mathbb T^3}T_k(\tilde{\varrho})\Div\tilde{\bfu}\,\dxs\\&\qquad-\liminf_\delta\tilde\E\int_0^t\int_{\mathbb T^3}T_k(\tilde{\varrho}_\delta)\Div\tilde{\bfu}_\delta\dxs\\
&\leq\tilde\E\int_0^t\int_{\mathbb T^3}\big(T_k(\tilde{\varrho})-\tilde{T}^{1,k}\big)\Div\tilde{\bfu}\,\dxs.
\end{align*}
Due to (\ref{eq:T1klim}) the right hand side tends to zero if $k\rightarrow\infty$ such that
\begin{align*}
\lim_{k\rightarrow\infty}\tilde\E\int_{\mathbb T^3}\big(\tilde{L}^{1,k}(t)-L_k(\tilde{\varrho})(t)\big)\dx=0.
\end{align*}
This finally means that
\begin{align*}
\tilde\E\int_{Q} \tilde{\varrho}_\delta\ln\tilde{\varrho}_\delta\dxt\longrightarrow\tilde \E\int_{Q} \tilde{\varrho}\ln\tilde{\varrho}\dxt.
\end{align*}
Convexity of $z\mapsto z\ln z$ yields strong convergence of $\tilde{\varrho}_\delta$.

This means we can pass to the limit in all terms of the system \eqref{eq:approximation2} and obtain a solution to \eqref{eq:} in the sense of Definition \ref{def:sol}.

\begin{proposition}\label{prop:martsol2}
$\big((\tilde{\Omega},\tilde{\mf},(\tilde{\mf}_t),\tilde{\prst}),\tilde\varrho,\tilde{\bu},\tilde{W}\big)$
is a finite energy weak martingale solution to \eqref{eq:} with the initial law $\Lambda$.
\end{proposition}
\begin{proof}
Having the strong convergence of the density the proof follows the ideas of Proposition \ref{prop:martsol1}.
\end{proof}

\appendix

\section{Auxiliary lemma}
\label{sec:appendix}

In this section, we establish some further properties of the operator $\mathcal{M}$ defined in \eqref{M}. We start with an easy observation:
\begin{equation}\label{23}
\|\mathcal{M}[\rho]\|_{\mathcal{L}(X_N,X_N)}\leq \sup_{x\in\mt}\rho.
\end{equation}
Next, we prove Lipschitz continuity of the operator $\rho\mapsto\mathcal{M}^{\frac{1}{2}}[\rho]$.

\begin{lemma}\label{lemma:aux}
Let $\rho_1,\,\rho_2\in L^2(\mt)$ and assume that there exists constant $\kappa>0$ such that
\begin{equation}\label{1}
\rho_1,\,\rho_2\geq \kappa.
\end{equation}
Then there exists a constant $C=C(\kappa ,N)>0$ such that
\begin{equation}\label{eq:1}
\big\|\mathcal{M}^{\frac{1}{2}}[\rho_1]-\mathcal{M}^{\frac{1}{2}}[\rho_2]\big\|_{\mathcal{L}(X_N,X_N)}\leq C(\kappa ,N)\|\rho_1-\rho_2\|_{L^2}.
\end{equation}

\begin{proof}
As we intend to apply the mean value theorem, let us first calculate the derivative of $\mathcal{M}^{\frac12}[\rho]$ with respect to $\rho$.
There holds
\begin{align*}
\lim_{h\rightarrow0}\frac{\mathcal{M}^{\frac{1}{2}}[\rho+hv]-\mathcal{M}^{\frac{1}{2}}[\rho]}{h}&=\frac{1}{2}\mathcal{M}^{-\frac{1}{2}}[\rho]\mathcal{M}[v].
\end{align*}
The mean value theorem now yields for some $\rho=\alpha\rho_1+(1-\alpha)\rho_2$
\begin{align}\label{2}
\begin{aligned}
\big\|&\mathcal{M}^{\frac{1}{2}}[\rho_1]-\mathcal{M}^{\frac{1}{2}}[\rho_2]\big\|_{\mathcal{L}(X_N,X_N)}\leq\frac{1}{2}\big\|\mathcal{M}^{-\frac{1}{2}}[\rho]\mathcal{M}[\rho_1-\rho_2]\big\|_{\mathcal{L}(X_N,X_N)}\\
&\qquad\leq \frac{1}{2}\big\|\mathcal{M}^{-\frac{1}{2}}[\rho]\big\|_{\mathcal{L}(X_N,X_N)}\big\|\mathcal{M}[\,\cdot\,]\big\|_{\mathcal{L}(L^2,\mathcal{L}(X_N,X_N))}\|\rho_1-\rho_2\|_{L^2}.
\end{aligned}
\end{align}
Since by definition of the operator $\mathcal{M}^{-1}$ (cf. \cite[Section 2.2]{feireisl1})
\begin{align*}
\|\mathcal{M}^{-1}[\rho]\|_{\mathcal{L}(X_N,X_N)}\leq \big(\inf_{x\in\mt}\rho\big)^{-1},
\end{align*}
we deduce that
\begin{align*}
\big\|\mathcal{M}^{-\frac{1}{2}}[\rho]\big\|_{\mathcal{L}(X_N,X_N)}\leq \kappa^{-\frac{1}{2}}
\end{align*}
whenever $\rho$ satisfies the bound \eqref{1}. Besides, we have
\begin{align*}
\|\mathcal{M}[\,\cdot\,]\|_{\mathcal{L}(L^2,\mathcal{L}(X_N,X_N))}&=\sup_{\substack{\rho\in L^2\\
\|\rho\|_{L^2}\leq 1}}\|\mathcal{M}[\rho]\|_{\mathcal{L}(X_N,X_N)}=\sup_{\substack{\rho\in L^2\\
\|\rho\|_{L^2}\leq 1}}\sup_{\substack{\bfpsi\in X_N\\\|\bfpsi\|_{X_N}\leq 1}}\|P_N(\rho\bfpsi)\|_{X_N}\\
&\leq \sup_{\substack{\rho\in L^2\\\|\rho\|_{L^2}\leq 1}}\sup_{\substack{\bfpsi\in X_N\\\|\bfpsi\|_{X_N}\leq 1}}\|\rho\bfpsi\|_{L^2}
\leq C(N)\sup_{\substack{\rho\in L^2\\\|\rho\|_{L^2}\leq 1}}\sup_{\substack{\bfpsi\in X_N\\\|\bfpsi\|_{X_N}\leq 1}}\|\rho\|_{L^2}\|\bfpsi\|_{L^\infty}\\
&\leq C(N).
\end{align*}
Hence \eqref{2} yields the claim.
\end{proof}
\end{lemma}

\par\noindent
{\bf Acknowledgements}. The authors wish to thank E. Feireisl for
useful discussions and suggestions. They are
also grateful to the referee for the careful reading of the paper,
and for valuable suggestions.


\begin{thebibliography}{19}
\bibitem{bensoussan} A. Bensoussan, Stochastic Navier-Stokes equations, Acta Appl. Math., 38 (3) (1995), 267-304.
\bibitem{BeTe} A. Bensoussan, R. Temam, \'{E}quations stochastiques du type Navier-Stokes. J. Functional Analysis 13 (1973), 195--222.
\bibitem{berthvovelle} F. Berthelin, J. Vovelle, Stochastic isentropic Euler equations, arXiv:1310.8093.
\bibitem{Br} D. Breit, Existence theory for stochastic power law fluids, J. Math. Fluid Mech. 17 (2015), 295--326.
\bibitem{b2} Z. Brze\'zniak, Stochastic partial differential equations in M-type 2 Banach spaces, Potential Anal. 4 (1995), 1-45.
\bibitem{on1} Z. Brze\'zniak, M. Ondrej\'at, Strong solutions to stochastic wave equations with values in Riemannian manifolds, J. Funct. Anal. 253 (2007) 449-481.
\bibitem{daprato} G. Da Prato, J. Zabczyk, Stochastic Equations in Infinite Dimensions, Encyclopedia Math. Appl., vol. 44, Cambridge University Press, Cambridge, 1992. 
\bibitem{debusergo} A. Debussche, Ergodicity results for the stochastic Navier-Stokes equations: an introduction, Lecture Notes in Math., Springer, Heidelberg, 2073, 23--108 (2013).
\bibitem{debussche1} A. Debussche, N. Glatt-Holtz, R. Temam, Local Martingale and Pathwise Solutions for an Abstract Fluids Model, Physica D 14-15 (2011), 1123-1144.
\bibitem{DL} R. J. DiPerna and P.-L. Lions, Ordinary differential equations, transport theory and
Sobolev spaces, Invent. Math. 98 (1989), 511--547.
\bibitem{dudley} R.\,M. Dudley, Real analysis and probability, Cambridge University Press, Cambridge 2002.
\bibitem{fei3} E. Feireisl, Dynamics of Compressible Flow, Oxford Lecture  Series in Mathematics and its Applications, Oxford University Press, Oxford, 2004.
\bibitem{feireisl2} E. Feireisl, B. Maslowski, A. Novotn\'y, Compressible fluid flows driven by stochastic forcing, J. Differential Equations 254 (2013) 1342-1358.
\bibitem{feireisl1} E. Feireisl, A. Novotn\'y, H. Petzeltov\'a, On the existence of globally defined weak solutions to the Navier-Stokes equations, J. Math. Fluid. Mech. 3 (2001) 358-392.
\bibitem{Fl2} F. Flandoli (2008): An introduction to 3D stochastic fluid dynamics. In SPDE in Hydrodynamic:
Recent Progress and Prospects. Lecture Notes in Math. 1942 51--150. Springer,
Berlin.
\bibitem{franco} F. Flandoli, A. Mahalov, Stochastic three-dimensional rotating Navier-Stokes equations: averaging, convergence and regularity, Arch. Rational Mech. Anal. 205 (2012) 195--237.
\bibitem{FMS} Frehse, J.; Steinhauer, M.; Weigant, W. The Dirichlet problem for steady viscous compressible flow in three dimensions. J. Math. Pures Appl. (9) 97 (2012), no. 2, 85--97.
\bibitem{Fu} H. Fujita Yashima, Equations de Navier-Stokes stochastiques non homog\`enes et applications, Tesi di Perfezionamento, Scuola Normale Superiore, Pisa, 1992.
\bibitem{krylov} I. Gy\"{o}ngy, N. Krylov, Existence of strong solutions for It\^o's stochastic equations via approximations, Probab. Theory Related Fields 105 (2) (1996) 143-158.
\bibitem{jakubow} A. Jakubowski, The almost sure Skorokhod representation for subsequences in nonmetric spaces, Teor. Veroyatnost. i Primenen 42 (1997), no. 1, 209-216; translation in Theory Probab. Appl. 42 (1997), no. 1, 167-174 (1998).
\bibitem{Li1}P.\,L. Lions, Mathematical topics in fluid mechanics. Vol. 1. Incompressible models. Oxford Lecture Series in Mathematics and its Applications, 3. Oxford Science Publications, The Clarendon Press, Oxford University Press, New York, 1996.
\bibitem{Li2} P.\,L. Lions, Mathematical topics in fluid mechanics. Vol. 2. Compressible models. Oxford Lecture Series in Mathematics and its Applications, 10. Oxford Science Publications, The Clarendon Press, Oxford University Press, New York, 1998.
\bibitem{Ma} J. Mattingly, On recent progress for the stochastic Navier Stokes equations, in: Journ\'{e}e Equations aux Deriv\'{e}es Partielles,
Forges-les-Eaux, 2003.
\bibitem{malek} J. M\'alek, J. Ne\v{c}as, M. Rokyta, M. R\r{u}\v{z}i\v{c}ka, Weak and Measure-valued Solutions to Evolutionary PDEs, Chapman \& Hall, London, Weinheim, New York, 1996.
\bibitem{mikul} R. Mikulevicius, B.\,L. Rozovskii, Stochastic Navier-Stokes equations for turbulent flows, SIAM J. Math. Anal. 35 (5) (2004) 1250-1310.
\bibitem{veraar} J.\,M.\,A.\,M. van Neerven, M.\,C. Veraar, L. Weis, Stochastic integration in UMD Banach spaces, Annals Probab. 35 (2007), 1438-1478.
\bibitem{novot} A. Novotn\'y, I. Stra\v{s}kraba, Introduction to the Mathematical Theory of Compressible Flow, Oxford Lecture Series in Mathematics and its Applications, Oxford University Press, Oxford, 2004.
\bibitem{on2} M. Ondrej\'at, Stochastic nonlinear wave equations in
local Sobolev spaces, Electronic Journal of Probability 15 (33) (2010) 1041-1091.
\bibitem{ondrejat3} M. Ondrej\'at, Uniqueness for stochastic evolution equations in Banach spaces, Dissertationes Mathematicae 426 (2004), 1-63.
\bibitem{onsebr} M. Ondrej\'at, J. Seidler, Z. Brze\'zniak, Invariant measures for stochastic nonlinear beam and wave equations, to appear.
\bibitem{PlVa} Plotnikov, P. I.; Weigant, W. Steady 3D viscous compressible flows with adiabatic exponent $\gamma\in(1,\infty)$. J. Math. Pures Appl. (9) 104 (2015), no. 1, 58--82.
\bibitem{pr07} C. Pr\'ev\^ot, M. R\"ockner, A concise course on stochastic partial differential equations, vol. 1905 of Lecture Notes in Math., Springer, Berlin, 2007.
\bibitem{Sa} M. Sango, Density dependent stochastic Navier--Stokes equations with non-Lipschitz random forcing, Rev. Math. Phys. 22
(2010) 669--697.
\bibitem{TeYo} Y. Terasawa, N. Yoshida (2011): Stochastic power-law fluids:
existence and uniqueness of weak solutions. The Annals of Applied Probability, Vol. 21, No. 5, 1827--1859.
\bibitem{To} E. Tornatore (2000): Global solution of bi-dimensional stochastic equation for a viscous gas, NoDEA Nonlinear Differential Equations
Appl. 7 (4), 343--360.
\bibitem{Yo} N. Yoshida (2012): Stochastic Shear thickenning fluids: strong convergence of the
Galerkin approximation and the energy inequality. The Annals of Applied Probability
2012, Vol. 22, No. 3, 1215--1242
\end{thebibliography}
\end{document}